\documentclass[12pt]{amsart}
\usepackage{fullpage}
\usepackage{graphicx}
\usepackage{todonotes}
\usepackage{tikz}
\usepackage{mathtools}
\usetikzlibrary{decorations.markings}
\usetikzlibrary{patterns}
\makeatletter
\@namedef{subjclassname@2020}{\textup{2020} Mathematics Subject Classification}
\makeatother

\newtheorem{theorem}{Theorem}[section]
\newtheorem{corollary}[theorem]{Corollary}
\newtheorem{lemma}[theorem]{Lemma}
\newtheorem{proposition}[theorem]{Proposition}
\newtheorem{definition}[theorem]{Definition}
\newtheorem{question}[theorem]{Question}
\newtheorem{remark}[theorem]{Remark}

\numberwithin{equation}{section}

\newcommand\CC {{\mathbb C}}

\newcommand\NN {{\mathbb N}}
\newcommand\QQ {{\mathbb Q}}
\newcommand\RR {{\mathbb R}}

\newcommand\ZZ {{\mathbb Z}}

\newcommand\slgroup{{\rm SL}}

\newcommand\im{{\rm Im }}
\newcommand\distance{{\rm Dist }}

\newcommand\id{{\rm Id}}
\newcommand\lipschitz{{\rm Lip}}

\newcommand\interior{{\rm Int}}
\newcommand\closure{{\rm Cl}}

\newcommand\cA{{\mathcal{A}  }}
\newcommand\cB{{\mathcal{B}  }}
\newcommand\cC{{\mathcal{C}  }}
\newcommand\cD{{\mathcal{D}  }}

\newcommand\cF{{\mathcal{F}  }}

\newcommand\cH{{\mathcal{H}  }}
\newcommand\cI{{\mathcal{I}  }}

\newcommand\cM{{\mathcal{M}  }}

\newcommand\cO{{\mathcal{O}  }}
\newcommand\cP{{\mathcal{P}  }}
\newcommand\cQ{{\mathcal{Q}  }}
\newcommand\cR{{\mathcal{R}  }}

\newcommand\cT{{\mathcal{T}  }}
\newcommand\cU{{\mathcal{U}  }}
\newcommand\cV{{\mathcal{V}  }}

\begin{document}

\title{Higher genus Cherry flows and full families of GIETs}

\author[L. Marchese]{Luca Marchese}
\address{Dipartimento di Matematica, Universit\`a di Bologna, Piazza di Porta San Donato 5, 40126, Bologna, Italia}

\email{luca.marchese4@unibo.it}

\author[L. Palmisano]{Liviana Palmisano}
\address{Department of Mathematics, KTH - Royal Institute of Technology, Stockholm, Sweden}

\email{liviana@kth.se}




\begin{abstract}
Cherry flows are classical examples of $C^\infty$ flows on the two-dimensional torus exhibiting non-trivial recurrent dynamics. For every genus $g\geq1$ we construct a $C^\infty$ parameter family of flows whose first return map is a generalized interval exchange transformation (GIET) with flat pieces. In such family, for every interval exchange transformation $T$ satisfying the Keane property, there are parameters corresponding to a Cherry flow whose return map is semi-conjugate to $T$. In particular, for each $k=1,\ldots,g$, our family contains a Cherry flow whose unique quasi-minimal set supports exactly $k$ ergodic invariant measures. The construction relies on a Full Family Theorem for GIETs with flat pieces, which establishes the realization of every admissible Rauzy renormalization path within a specific finite-dimensional family with the optimal number of parameters.
\end{abstract}

\maketitle


\section{Introduction}

The study of non-trivial recurrent dynamics for flows on the two-dimensional torus is a classical topic in dynamical systems. In 1886, Poincar\'e \cite{Poincare} asked whether a $C^\infty$ flow on the torus could admit a non-trivial minimal set strictly contained in the torus. Here, a minimal set is a nonempty compact invariant set in which every orbit is dense, and non-triviality excludes equilibria and periodic orbits. Denjoy later gave a negative answer, ruling out such examples even in the $C^2$ category \cite{Denjoy}. In 1938, Cherry \cite{Cherry} showed that a smooth flow on the torus can nevertheless have a non-trivial recurrent orbit whose closure is a proper subset of the torus. Such a closure is called a \emph{quasi-minimal} set and may contain singularities. His examples are $C^\infty$ flows with a hyperbolic saddle, a sink, and a unique quasi-minimal set. These are now known as \emph{Cherry flows}. The dynamics of a Cherry flow is naturally encoded by the first return map to a suitable circle section. Extending this map across the interval of points that do not return gives a continuous, non-decreasing circle map of degree one, constant on that interval and semi-conjugate to an irrational rotation. The connection with one-dimensional dynamics has played a central role in the study of the topological, geometric, and ergodic properties of Cherry flows; see, for example,
\cite{GraczykJonkerSwiatekTangermanVeerman,MartensvanStriendeMelo, Mendes,MoreiraGaspar,Palmisano1,Palmisano2,Palmisano3,Palmisano4}.

A natural question is whether a similar picture holds beyond the torus. In this paper, we construct $C^\infty$ analogues of Cherry flows on compact orientable surfaces of every positive genus. These flows have one sink and finitely many saddles. One saddle has a separatrix connecting it to the sink, and there are no saddle connections. Each flow has a unique quasi-minimal set; see Section~\ref{SectionHigherGenusCherryFlows} for precise definitions. 
The one-dimensional models change accordingly. Linear flows on translation surfaces have interval exchange transformations as their first return maps. For the Cherry flows considered here, the extended return maps retain the interval-exchange structure away from an interval on which they are constant. This leads to generalized interval exchange transformations (GIETs) with flat pieces; see Section~\ref{SectionIETsDeformationsFlatPiece} and Figure~\ref{FigureIETsFlatPiece}.

Our first main result realizes prescribed interval exchange dynamics through the return maps of higher-genus Cherry flows. More precisely, for every genus $g\geq1$ and every admissible collection of saddle indices, we construct a parameter family of flows such that, for every interval exchange transformation $T$ satisfying the Keane property and having compatible combinatorics, the family contains a Cherry flow whose return map is semi-conjugate to $T$ (Theorem~\ref{TheoremMainTheoremCherryFlow}). This realization also gives control over the ergodic structure of the quasi-minimal set. In particular, for each $k=1,\ldots,g$, our family contains a Cherry flow whose unique quasi-minimal set supports exactly $k$ ergodic, non necessarily finite, measures (Corollary~\ref{Cor:uniqueQMSandmeasures}). 

A related realization result was obtained by Carrand \cite{Carrand} for self-similar interval exchange transformations and $C^1$ surface flows. Our construction applies to general admissible Keane interval exchange transformations, produces $C^\infty$ flows, and includes examples with different numbers of ergodic measures within the same family. A complementary perspective is provided by the recent work of Fougeron, Schmidhuber and Ulcigrai \cite{FougeronSchmidhuberUlcigrai} where the authors extend Rauzy--Veech induction to GIETs with gaps to study the decomposition of a given GIET into periodic and quasi-minimal domains, together with regions without recurrence. Our work addresses a complementary realization problem: we construct finite-dimensional families of GIETs with flat pieces realizing every admissible Rauzy renormalization path and use them to produce smooth Cherry flows with controlled ergodic properties.

The one-dimensional mechanism underlying our construction is the second main result of the paper: a Full Family Theorem for GIETs with flat pieces (Theorem~\ref{TheoremFullFamilyTheorem}). It extends the theorem established in \cite{MarchesePalmisano} for invertible GIETs to a class whose branches are continuous and non-decreasing but may contain flat intervals; see Section~\ref{SectionClassGIETs}. In this broader class, we introduce finite-dimensional parameter families with the optimal number of real parameters and define the notion of a \emph{full family} (Definition~\ref{DefinitionFullFamily}). Under a natural analytic condition and a mild combinatorial assumption, we prove that every admissible Rauzy renormalization path is realized by a member of the family. In particular, for every interval exchange transformation satisfying the Keane property and the prescribed combinatorial conditions, the family contains a map semi-conjugate to that transformation. This is the realization mechanism used in the construction of the surface flows.

Full family theorems express a form of universality: a finite-dimensional family is rich enough to realize all admissible combinatorial dynamics. Results of this type have been established for continuous non-decreasing circle maps \cite{deMelovanStrien}, multimodal circle maps \cite{deMeloSalomaoVargas}, tent maps and multimodal interval maps \cite{MilnorThurston}, Lorenz maps \cite{MartensdeMelo}, the complex quadratic family \cite{DouadyHubbard}, and rational maps \cite{deFariadeMeloSalomaoVargas}. Such results do not hold in complete generality, as illustrated by the negative result for H\'enon maps in \cite{HazardMartensTresser}. The mechanisms behind these theorems vary considerably. For continuous non-decreasing circle maps, fullness follows from continuity of the rotation number and the intermediate value theorem, whereas for the complex quadratic family it relies on complex-analytic techniques. In several other settings, proofs use a method introduced by Thurston; see \cite{DongMartensPalmisano} for a presentation of this approach. Our proof develops this method in the setting of generalized interval exchange transformation with flat pieces.

For each finite Rauzy path, we introduce a configuration space and a Thurston map such that a fixed point of the map determines a parameter realizing that path. We analyse the closure of the configuration space and extend the map continuously to it, establishing the conditions needed to apply Brouwer's fixed-point theorem. We then exclude fixed points on the boundary, so that the resulting fixed point belongs to the configuration space. By realizing successively longer initial segments of a prescribed infinite path and passing to the limit, we obtain a GIET realizing the entire path. The defining properties of a full family (Definition~\ref{DefinitionFullFamily}) emerge naturally from this construction.

The main new difficulties compared with \cite{MarchesePalmisano} concern the geometry of the configuration space and the boundary behaviour of the Thurston map. In \cite{MarchesePalmisano}, the configuration space was the full simplex, and the boundary dynamics had a cyclic structure. Here, the configuration space has a more complicated geometry within the ambient simplex, and the boundary dynamics no longer has this cyclic structure. The proof therefore requires both an analysis of the geometry needed for the fixed-point argument and new arguments to exclude boundary fixed points. 
An alternative approach to a full family theorem would be to consider a parametric family like in 
Definition \ref{DefinitionFullFamily} and proving that its topological/analytic properties remain stable after renormalizing the family with the Rauzy-Veech induction. This requires the study of topological and analytic properties of the renormalization operator. The Thurston's approach adopted in our paper bypasses these difficulties.

The Full Family Theorem also opens a new direction of investigation into rigidity for GIETs with flat pieces. Having established existence for every admissible Rauzy renormalization path, we can ask whether, and under what additional assumptions, rigidity results analogous to those known for invertible GIETs \cite{Ghazouani,GhazouaniUlcigrai1,GhazouaniUlcigrai2} hold in this setting. Such a study could also clarify the geometric and ergodic properties of the higher-genus Cherry flows constructed here, building on investigations of these properties on the two-dimensional torus \cite{Palmisano1,Palmisano2,Palmisano3}. Our results broaden the classical correspondence between one-dimensional dynamics and surface flows. On the one-dimensional side, they extend the full-family and Thurston realization approach to non-invertible GIETs with flat pieces. On the surface side, they extend Cherry dynamics beyond the torus through families of smooth flows with controlled ergodic properties. This connection provides a framework for studying the topological, geometric, and ergodic properties of quasi-minimal sets on higher-genus surfaces through their one-dimensional return maps.

\subsection{IETs with a flat piece}
\label{SectionIETsDeformationsFlatPiece}

An alphabet is a finite set $\cA$ with $d\geq2$ letters. A \emph{combinatorial datum} $\pi$ is a pair of bijections $\pi_t,\pi_b:\cA\to\{1,\dots,d\}$. Such $\pi$ is \emph{admissible} if  
$$
(\pi_t)^{-1}\{1,\dots,k\}\not=(\pi_b)^{-1}\{1,\dots,k\}
\quad\text{ for any }\quad
1\leq k\leq d-1.
$$
A \emph{length datum} is a vector $\lambda\in\RR_+^\cA$ with all entries positive. For any such pair of data $(\pi,\lambda)$ consider the interval 
$
I:=[0,\sum_{\chi\in\cA}\lambda_\chi)
$ 
and its two partitions 
$
I=\bigsqcup_{\alpha\in\cA}I^t_\alpha
$ 
and 
$
I=\bigsqcup_{\alpha\in\cA}I^b_\alpha
$, 
where for any $\alpha$ we set
\begin{equation}
\label{EquationIntervalsIET}
I^t_\alpha:=
\Big[
\sum_{\pi_t(\chi)\leq\pi_t(\alpha)-1}\lambda_\chi,
\sum_{\pi_t(\chi)\leq\pi_t(\alpha)}\lambda_\chi
\Big)
\quad
\text{ and }
\quad
I^b_\alpha:=
\Big[
\sum_{\pi_b(\chi)\leq\pi_b(\alpha)-1}\lambda_\chi,
\sum_{\pi_b(\chi)\leq\pi_b(\alpha)}\lambda_\chi
\Big).
\end{equation}
An IET is the map $T:I\to I$ that for any $\alpha$ sends $I^t_\alpha$ onto $I^b_\alpha$ via a translation. Such map is  uniquely determined by the data $(\pi,\lambda)$ and is denoted $T(\pi,\lambda)$. We often consider IETs acting on $[0,1)$, that is IETs defined by length data $\lambda$ in the simplex 
$
\Delta^{d}:=\{\lambda\in\RR^\cA_+:\sum_{\chi\in\cA}\lambda_\chi=1\}
$.

\smallskip

Let $g:[0,1)\to[0,1)$ be a non-decreasing and surjective function such that there exist four points $0<a<b<c<d\leq1$ with 
\begin{equation}
\label{EquationPerturbationOfIdentity}
\left\{
\begin{array}{l}
\bullet\quad
g(x)=x \text{ for any }x\in[0,1)\setminus(a,d)
\\
\bullet\quad
g \text{ is strictly increasing on }[a,b)\cup[c,d) 
\\
\bullet\quad
g \text{ is constant on }[b,c].  
\end{array}
\right.
\end{equation}
Any such $g$ is continuous since non-decreasing and surjective. An \emph{IET with a flat piece} over $\pi$ is a map $f:[0,1)\to[0,1)$ of the form $f=T\circ g$, where $T$ is an IET with combinatorial datum $\pi$ and $g$ is a map as in Equation~\eqref{EquationPerturbationOfIdentity}, such that $[a,d)$ is contained in some continuity interval $I^t_\alpha$ of $T$. The interval $U:=[b,c]$ is called the \emph{flat interval} of $f$. See Figure~\ref{FigureIETsFlatPiece}.

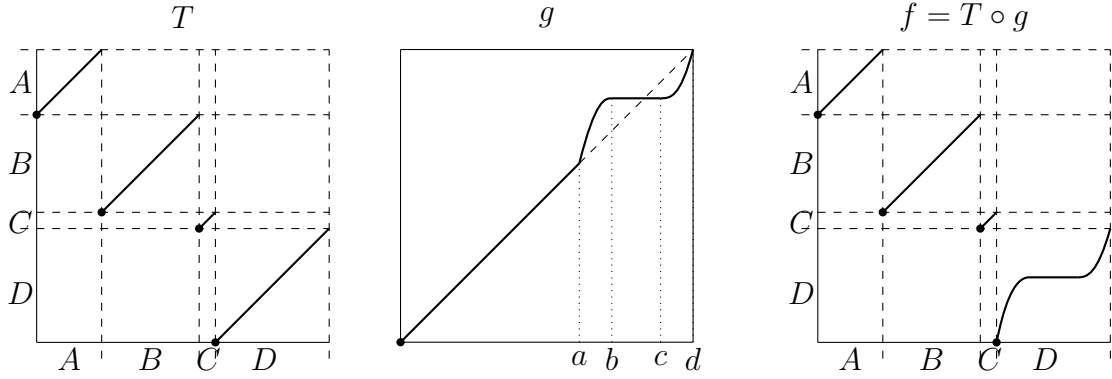
\begin{figure}
\begin{center}
{\begin{tikzpicture}[scale=0.043]



\node at (45,100) {$T$};

\draw[-] (0,0) -- (90,0) {};
\draw[-,thin,dashed] (-5,35) -- (90,35) {};
\draw[-,thin,dashed] (-5,40) -- (90,40) {};
\draw[-,thin,dashed] (-5,70) -- (90,70) {};
\draw[-,thin,dashed] (-5,90) -- (90,90) {};

\draw[-] (0,0) -- (0,90) {};
\draw[-,thin,dashed] (20,-5) -- (20,90) {};
\draw[-,thin,dashed] (50,-5) -- (50,90) {};
\draw[-,thin,dashed] (55,-5) -- (55,90) {};
\draw[-,thin,dashed] (90,-5) -- (90,90) {};


\node [circle,fill,inner sep=1.1pt] at (0,70) {};
\draw[-,thick] (0,70)-- (20,90) {};
\node at (10,-5) {$A$};
\node at (-5,80) {$A$};


\node [circle,fill,inner sep=1.1pt] at (20,40) {};
\draw[-,thick] (20,40)-- (50,70) {};
\node at (35,-5) {$B$};
\node at (-5,55) {$B$};


\node [circle,fill,inner sep=1.1pt] at (50,35) {};
\draw[-,thick] (50,35)-- (55,40) {};
\node at (52.9,-5) {$C$};
\node at (-5,37) {$C$};


\node [circle,fill,inner sep=1.1pt] at (55,0) {};
\draw[-,thick] (55,0)-- (90,35) {};
\node at (70,-5) {$D$};
D\node at (-5,15) {$D$};



\end{tikzpicture}}
\hspace{0.6 cm}
{\begin{tikzpicture}[scale=0.043]

\node at (45,100) {$g$};

\node [circle,fill,inner sep=1.1pt] at (0,0) {};
\draw[-] (0,0) -- (90,0) {};
\draw[-] (0,90) -- (90,90) {};
\draw[-] (90,0) -- (90,90) {};

\draw[-] (0,0) -- (0,90) {};
\draw[-,thick] (0,0) -- (55,55) {};
\draw[-,thin,dashed] (55,55) -- (90,90) {};

\draw[-,thick] (65,75) -- (80,75) {};
\draw[-,thick] (55,55) .. controls (60,75) and (63,75) .. (65,75) {};
\draw[-,thick] (80,75) .. controls (83,75) and (86,75) .. (90,90) {};

\draw[-,thin,dotted] (55,0) -- (55,55) {};
\node at (55,-5) {$a$};
\draw[-,thin,dotted] (65,0) -- (65,75) {};
\node at (65,-5) {$b$};
\draw[-,thin,dotted] (80,0) -- (80,75) {};
\node at (80,-5) {$c$};
\draw[-,thin,dotted] (90,0) -- (90,90) {};
\node at (90,-5) {$d$};

\end{tikzpicture}}
\hspace{0.6 cm}
{\begin{tikzpicture}[scale=0.043]

\node at (45,100) {$f=T\circ g$};

\draw[-] (0,0) -- (90,0) {};
\draw[-,thin,dashed] (-5,35) -- (90,35) {};
\draw[-,thin,dashed] (-5,40) -- (90,40) {};
\draw[-,thin,dashed] (-5,70) -- (90,70) {};
\draw[-,thin,dashed] (-5,90) -- (90,90) {};

\draw[-] (0,0) -- (0,90) {};
\draw[-,thin,dashed] (20,-5) -- (20,90) {};
\draw[-,thin,dashed] (50,-5) -- (50,90) {};
\draw[-,thin,dashed] (55,-5) -- (55,90) {};
\draw[-,thin,dashed] (90,-5) -- (90,90) {};


\node [circle,fill,inner sep=1.1pt] at (0,70) {};
\draw[-,thick] (0,70)-- (20,90) {};
\node at (10,-5) {$A$};
\node at (-5,80) {$A$};


\node [circle,fill,inner sep=1.1pt] at (20,40) {};
\draw[-,thick] (20,40)-- (50,70) {};
\node at (35,-5) {$B$};
\node at (-5,55) {$B$};


\node [circle,fill,inner sep=1.1pt] at (50,35) {};
\draw[-,thick] (50,35)-- (55,40) {};
\node at (52.9,-5) {$C$};
\node at (-5,37) {$C$};


\node [circle,fill,inner sep=1.1pt] at (55,0) {};
\draw[-,thick] (55,0) .. controls (57,10) and (60,20) .. (65,20) {};
\draw[-,thick] (80,20) .. controls (83,20) and (86,21) .. (90,35) {};
\draw[-,thick] (65,20)-- (80,20) {};
\node at (70,-5) {$D$};
D\node at (-5,15) {$D$};



\end{tikzpicture}}
\end{center}
\caption{On the left is represented a standard IET $T$ and in the center a map $g$ as in Equation~\eqref{EquationPerturbationOfIdentity}. The map $f=T\circ g$ on the right represents an IET with a flat piece, whose flat interval is $U:=[b,c]$. We have $[a,d)\subset I^t_D$.}
\label{FigureIETsFlatPiece}
\end{figure}

\subsection{Cherry flows}
\label{SectionHigherGenusCherryFlows}

Let $Z:\RR^2\to\RR^2$ be the smooth vector field defined by
\begin{equation}
\label{EquationSaddleSinkNode}
Z(x,y):=\big(-x,y^2-1\big).
\end{equation}
The only zeros of $Z$ are at points $s:=(0,-1)$, which has index\footnote{If $p$ is an isolated zero of a smooth vector field $V:\RR^2\to\RR^2$, then its index $\iota_V(p)$ at $p$ is defined as the degree of the map $x\mapsto |V(x)|^{-1}\cdot V(x)$ from a small circle around $p$ onto the unitary circle.} $\iota_Z(s)=1$ and is a sink, and $p:=(0,1)$, which is a saddle with index $\iota_Z(p)=-1$. 
Moreover the vertical segment connecting $s$ to $p$ is an integral curve of $Z$. A representation of the integral curves of $Z$ appears in the right part of Figure~\ref{FigureZipperedRectanglesAndSaddleSink}. A smooth vector field $Y$ over a smooth surface $X$ has a saddle-sink node $(s,p)$ if there exists an open neighborhood $U$ in $X$ such that, modulo a change of coordinates, the restriction of $Y$ to $U$ coincides with the restriction of $Z$ in Equation~\eqref{EquationSaddleSinkNode} to the ball $B(0,4)$.

Let $X$ be a smooth, compact and boundaryless surface of genus $g\geq1$ and fix positive integers $k_1,\dots,k_m$. Let $Y$ be a smooth vector field on $X$ with isolated zeros at points $p_1,\dots,p_m$. We say that these zeros have type $(k_1,\dots,k_m)$ if their indices satisfy  $\iota_Y(p_i)=-k_i$ for any $i=1,\dots,m$. If such $Y$ has also a saddle-sink node $(s_0,p_0)$ and no other zeros, the Theorem of Poincar\'e-Hopf gives
\begin{equation}
\label{EquatinPoincareHopf}
2g-2=
\sum_{i=1}^m k_i=
-\sum_{i=1}^m \iota_Y(p_i)-\bigg(\iota_Y(s_0)+\iota_Y(p_0)\bigg).
\end{equation}
A \emph{Cherry flow} on $X$ is the integral flow of a field $Y$ as above, such that there is no integral curve of $Y$ connecting two zeros except the saddle-sink node. 

\smallskip

A \emph{translation surface} $X$ is a flat surface of genus $g\geq1$ with a finite number of conical singularities, whose conical angles are integer multiples of $2\pi$. Such data induce the structure of a Riemann surface with an holomorphic one form, and thus a smooth atlas. On such surface one can define \emph{linear flows}, whose \emph{first return maps} are standard IETs (see \S~\ref{SectionTranslationSurfaces}).

Admissible combinatorial data are elements of sets called \emph{Rauzy classes} (see \S~\ref{SectionRauzyClasses}). A connected component $\cH$ of the moduli space of translation surfaces of genus $g\geq1$ is an \emph{orbifold} of complex dimension $2g+m-1$, where $k_1+\dots+k_m=2g-2$ and $k_1,\dots,k_m$ are the orders of the conical singularities of the surfaces in $\cH$. Let $\pi$ be an admissible combinatorial datum over an alphabet $\cA$ with $d$ letters.  
Veech's \emph{zippered rectangles} construction (see \cite{Veech}, or \S~\ref{SectionZipperedRectangles} of this paper) gives a connected component $\cH$ as above and a subset $\cH(\pi)\subset\cH$ of surfaces represented as \emph{suspension} of IETs with combinatorial datum $\pi$. All combinatorial data in the same Rauzy class as $\pi$ give rise to surfaces in $\cH$. Moreover, any connected component is associated to an unique Rauzy class by such construction. 

Fix a letter $\chi\in\cA$ and replace it by a pair of letters $\chi_L,\chi_R$. This replaces $\pi$ by an  admissible combinatorial datum $\widehat{\pi}$ over the alphabet 
$
\widehat{\cA}:=\{\chi_L,\chi_R\}\cup\cA\setminus\{\chi\}
$, 
which has $d+1$ letters. We say that $\pi$ is \emph{allowed} if there exists a letter $\chi$ such that 
$\widehat{\pi}$ constructed as above belongs to a Rauzy class containing a \emph{cyclic} combinatorial datum, in the sense of the next Definition~\ref{DefinitionCyclicCombinatorialDatum}. This construction is necessary in the proof of the next Theorem~\ref{TheoremMainTheoremCherryFlow}. Therefore we use parametric families of vector fields on surfaces in $\cH(\pi)$ parametrized by IETs with $d+1$ letters. The natural parameter space for such families is the simplex 
$\Delta^{d+1}$.

In this paper we construct Cherry flows modelled on linear flows on translation surfaces. More precisely, we define parametric families of flows whose zeros are as in Equation~\eqref{EquatinPoincareHopf} and whose first return map, in the sense explained in \S~\ref{SectionParameterFamilyFirstReturnMaps}, is an IET with a flat piece. An example of such a map is represented in Figure~\ref{FigureIETsFlatPiece}, moreover maps arising in this way belong to a larger class of maps introduced in the next \S~\ref{SectionFullFamilyTheoremForFlatIntervals}. Then, inside our parametric family, we prove the existence of parameters such that the corresponding flow is a Cherry flow. This is stated in the next Theorem~\ref{TheoremMainTheoremCherryFlow}.

\begin{theorem}
\label{TheoremMainTheoremCherryFlow}
Assume that the combinatorial datum $\pi$ is allowed. 
There exists a family 
$
(X_{\widehat{\lambda}},Y_{\widehat{\lambda}})_{\widehat{\lambda}\in\Delta^{d+1}}
$, 
where $X_{\widehat{\lambda}}\in\cH(\pi)$ and $Y_{\widehat{\lambda}}$ is a smooth vector field on 
$X_{\widehat{\lambda}}$, such that the following holds.
\begin{enumerate}
\item
For any $\widehat{\lambda}$ the zeros of $Y_{\widehat{\lambda}}$ are as in Equation~\eqref{EquatinPoincareHopf}.
\item
For any $\widehat{\lambda}$, the flow $Y_{\widehat{\lambda}}$ has a smooth first return map\footnote{In the sense explained in \S~\ref{SectionParameterFamilyFirstReturnMaps}} $f_{\widehat{\lambda}}:[0,1)\to[0,1)$ which is an IET with a flat piece over $\pi$. 
\item
If $T$ is a Keane IET with combinatorial datum $\pi$, then there exists 
$\widehat{\lambda}\in\Delta^{d+1}$ such that $Y_{\widehat{\lambda}}$ generates a Cherry flow and
$f_{\widehat{\lambda}}$ is semi-conjugated to $T$. 
\end{enumerate}
\end{theorem}

A \emph{quasi-minimal set} for the flow of a smooth vector field $Y$ on a surface $X$ is the closure of a recurrent orbit. The next Corollary is proved in \S~\ref{SectionProofOfCorollary}.

\begin{corollary}
\label{Cor:uniqueQMSandmeasures}
For any $\widehat{\lambda}$ as in Point (3) of Theorem~\ref{TheoremMainTheoremCherryFlow} the flow of 
$Y_{\widehat{\lambda}}$ has only one quasi-minimal set.  Moreover for any $k=1,\dots,g$ there exists parameters $\widehat{\lambda}$ such that, excluding the Dirac measures at the zeros, the quasi-minimal set of 
$Y_{\widehat{\lambda}}$ is the support of exactly $k$ ergodic measures. These measures are non necessarily finite. 
\end{corollary}

Let $\widehat{\lambda}$ be a parameter as in Point (3) of Theorem~\ref{TheoremMainTheoremCherryFlow}. Locally, the corresponding quasi minimal set is the product of a Cantor set with an interval. Globally, the quasi minimal set intersects a specific section in the Cantor set appearing in Equation \eqref{EquationQuasiMinimalSet}, that is
$$
K:=[0,1)\setminus
\bigcup_{y\in[0,1)}\operatorname{Int}h^{-1}(y),
$$
where $h:[0,1)\to[0,1)$ is the \emph{semi-conjugacy} map in the next Equation \eqref{EquationSemiconjugation}, and where the intervals $\operatorname{Int}h^{-1}(y)$ in the union above are disjoint and thus countable. These intervals are the preimages of the flat interval, together with the \emph{wandering intervals} of $f_{\widehat{\lambda}}$.

\subsection{Full Family Theorem for a class of generalizations of IETs}
\label{SectionFullFamilyTheoremForFlatIntervals}

A \emph{right-open interval} is an interval $[a,b)\subset\RR$ closed on the left and open on the right. Denote by $\interior(I):=(a,b)$ its interior and by $\closure(I):=[a,b]$ its closure.

\subsubsection{The class $\cC(\pi,I)$ of GIETs}
\label{SectionClassGIETs}

Let $\cA$ be a finite alphabet with $d\geq2$ letters.  Given a combinatorial datum $\pi$ over $\cA$ and a right-open interval $I$, consider two partitions of $I$ into $d$ right-open subintervals
$
\cP_t=\{I^t_\alpha\}_{\alpha\in\cA}
$ 
and 
$
\cP_b=\{I^b_\alpha\}_{\alpha\in\cA}
$, 
where for any $\alpha\in\cA$ the subintervals   
$
I_\alpha^t
$ 
and 
$
I_\alpha^b
$ 
appear respectively in $\pi_t(\alpha)$-th and $\pi_b(\alpha)$-th position in $I$ counting from the left.   
A \emph{generalized interval exchange transformation}, of briefly GIET, is a map $f:I\to I$ such that for any 
$\alpha\in\cA$ the restriction satisfies  
\begin{equation}
\label{EquationPiecewiseHomeomorphism}
f|_{I_\alpha^t}=f_\alpha,
\end{equation}
where $f_\alpha:I_\alpha^t\to I_\alpha^b$ is a continuous, surjective and non-decreasing map, and there exists a right open interval 
$
J_\alpha\subset I^t_\alpha
$ 
such that the restriction 
$
f_\alpha|_{J_\alpha}:J_\alpha\to I^b_\alpha
$ 
is an homeomorphism. We have $\sup J_\alpha=\sup I^t_\alpha$ and $f_\alpha$ is an homeomorphism if 
$\inf I^t_\alpha=\inf J_\alpha$. In general, the map $f_\alpha$ is constant on the closed interval 
\begin{equation}
\label{EquationFlatInterval}
L_\alpha:=[\inf I^t_\alpha,\inf J_\alpha].
\end{equation}
Denote $\cC(\pi,I)$ the set of GIETs of the interval $I$ with combinatorial datum $\pi$. Two maps in this class are represented in the left and the central part of Figure~\ref{FigureDegenerationGIETs}. In \cite{MarchesePalmisano} it was considered the smaller class of maps $f$ such that all continuous branches $f_\alpha$ are increasing homeomorphisms, that is 
\begin{equation}
\label{EquationGIETsDefinedByHomeomorphisms}
\big\{f\in\cC(\pi,I):J_\alpha=I^t_\alpha\quad\forall\alpha\in\cA\big\}.
\end{equation}

\subsubsection{The singular values map}
\label{SectionSingularValuesMap}

Consider $f\in\cC(\pi,I)$ and the partitions 
$
\cP_t=\{I^t_\alpha\}_{\alpha\in\cA}
$ 
and 
$
\cP_b=\{I^b_\alpha\}_{\alpha\in\cA}
$ 
of $I$. It is practical to keep track of the dependence on $f$ of such partitions and their atoms, so we write
\begin{equation}
\label{EquationDynamicalPartitionTimeZero}
\cP_t(f)=\{I^t_\alpha(f)\}_{\alpha\in\cA}
\quad
\textrm{ and }
\quad
\cP_b(f)=\{I^b_\alpha(f)\}_{\alpha\in\cA}.
\end{equation}
For $\alpha\in\cA$, the \emph{singular points} $u_\alpha^t(f)$ and the \emph{singular values} 
$u_\alpha^b(f)$ of $f$ are respectively 
\begin{equation}
\label{EquationSingularitiesTimeZero}
u_\alpha^t(f):=\inf I^t_\alpha(f)
\quad
\textrm{ and }
\quad
u_\alpha^b(f):=\inf I^b_\alpha(f).
\end{equation}
In particular, if $f=T(\pi,\lambda)$ is the IET determined by combinatorial-length data $(\pi,\lambda)$, then for $\alpha\in\cA$ its singularities 
$
u_\alpha^t(f)=u_\alpha^t(\pi,\lambda)
$ 
and 
$
u_\alpha^b(f)=u_\alpha^b(\pi,\lambda)
$ 
are given by the linear relations 
\begin{equation}
\label{EquationSingularitiesTimeZeroIET}
u_\alpha^t(f)=
\sum_{\pi_t(\chi)\leq\pi_t(\alpha)-1}\lambda_\chi
\quad
\textrm{ and }
\quad
u_\alpha^b(f)=
\sum_{\pi_b(\chi)\leq\pi_b(\alpha)-1}\lambda_\chi.
\end{equation}

Let $\nabla_\pi\subset\RR_+^\cA$ be the simplex of vectors 
$u=(u_\alpha)_{\alpha\in\cA}$ in $\RR_+^\cA$ with 
$u_{\beta_0}=0$ for the letter $\beta_0$ with $\pi_b(\beta_0)=1$ and such that $u_\alpha<u_\beta<1$ whenever 
$
\pi_b(\alpha)\leq\pi_b(\beta)-1
$. 
With this notation, we introduce the \emph{singular values map}
\begin{equation}
\label{EquationSingularValuesMap}
u^b:\cC\big(\pi,[0,1)\big)\to \nabla_\pi,
\quad f\mapsto \big(u^b_\alpha(f)\big)_{\alpha\in\cA}.
\end{equation}

For example, if $f\in\cC\big(\pi,[0,1)\big)$ has alphabet $\cA=\{A,B,C,D\}$ and combinatorial datum which satisfies  $\pi_b(D)=1$, $\pi_b(C)=2$, $\pi_b(B)=3$ and $\pi_b(A)=4$, then the entires of the vector $(u_A,u_B,u_C,u_D):=u^b(f)$ satisfy $0=u_D<u_C<u_B<u_A<1$. We also introduce the \emph{singular points map}
\begin{equation}
\label{EquationSingularPointsMap}
u^t:\cC\big(\pi,[0,1)\big)\to \nabla^t_\pi,
\quad f\mapsto \big(u^t_\alpha(f)\big)_{\alpha\in\cA},
\end{equation}
where $\nabla^t_\pi$ is the simplex  of vectors $u=(u_\alpha)_{\alpha\in\cA}$ in $\RR_+^\cA$ with $u_{\alpha_0}=0$ for the letter $\alpha_0$ with $\pi_t(\alpha_0)=1$ and such that $u_\alpha<u_\beta<1$ whenever 
$
\pi_t(\alpha)\leq\pi_t(\beta)-1
$. 
Finally for any $\alpha\in\cA$ it is useful to introduce also the map 
\begin{equation}
\label{EquationRightEndpointFlatInterval}
j_\alpha:\cC\big(\pi,[0,1)\big)\to[0,1)
\quad f\mapsto j_\alpha(f):=\sup L_\alpha(f)=\inf J_\alpha(f),
\end{equation}
where $L_\alpha(f)$ is the flat interval of the $\alpha$-branch of $f$ in Equation \eqref{EquationFlatInterval}.

\begin{figure}
\begin{center}
{\begin{tikzpicture}[scale=0.043]



\node at (50,110) {$g$};

\draw[-] (0,0) -- (100,0) {};
\draw[-,thin,dashed] (-5,5) -- (100,5) {};
\draw[-,thin,dashed] (-5,20) -- (100,20) {};
\draw[-,thin,dashed] (-5,40) -- (100,40) {};
\draw[-,thin,dashed] (-5,70) -- (100,70) {};
\draw[-,thin,dashed] (-5,100) -- (100,100) {};

\draw[-] (0,0) -- (0,100) {};
\draw[-,thin,dashed] (15,-5) -- (15,100) {};
\draw[-,thin,dashed] (45,-5) -- (45,100) {};
\draw[-,thin,dashed] (65,-5) -- (65,100) {};
\draw[-,thin,dashed] (95,-5) -- (95,100) {};
\draw[-,thin,dashed] (100,-5) -- (100,100) {};


\node [circle,fill,inner sep=1.1pt] at (0,5) {};
\draw[-,thick] (0,5) .. controls (10,5) and (10,15) .. (15,20) {};
\node [circle,inner sep=1.5pt] at (20,35) {};
\node at (7,-5) {$A$};
\node at (-5,13) {$A$};


\node [circle,fill,inner sep=1.1pt] at (15,70) {};
\draw[-,thick] (15,70) .. controls (30,90) and (40,100) .. (45,100) {};
\node [circle,inner sep=1.5pt] at (60,100) {};
\node at (30,-5) {$B$};
\node at (-5,85) {$B$};


\node [circle,fill,inner sep=1.1pt] at (45,20) {};
\draw[-,thick] (45,20) .. controls (53,38) and (58,40) .. (65,40) {};
\node at (54,-5) {$E$};
\node at (-5,30) {$E$};


\node [circle,fill,inner sep=1.1pt] at (65,40) {};
\draw[-,thick] (65,40) --(80,40) {};
\draw[-,thick] (80,40) .. controls (85,40) and (87,45) .. (95,70) {};
\node [circle,inner sep=1.1pt] at (65,70) {};
\node at (80,-5) {$C$};
\node at (-5,55) {$C$};


\node [circle,fill,inner sep=1.1pt] at (95,0) {};
\draw[-,thick] (95,0) -- (100,5) {};
\node [circle,inner sep=1.1pt] at (95,0) {};
\node at (97,-5) {$D$};

\node at (-5,2) {$D$};

\end{tikzpicture}}
\hspace{0.5 cm}
{\begin{tikzpicture}[scale=0.043]


\node at (50,110) {$f$};


\draw[-] (0,0) -- (100,0) {};
\draw[-,thin,dashed] (-5,5) -- (100,5) {};
\draw[-,thin,dashed] (-5,25) -- (100,25) {};
\draw[-,thin,dashed] (-5,30) -- (100,30) {};
\draw[-,thin,dashed] (-5,60) -- (100,60) {};
\draw[-,thin,dashed] (-5,100) -- (100,100) {};

\draw[-] (0,0) -- (0,100) {};
\draw[-,thin,dashed] (20,-5) -- (20,100) {};
\draw[-,thin,dashed] (60,-5) -- (60,100) {};
\draw[-,thin,dashed] (65,-5) -- (65,100) {};
\draw[-,thin,dashed] (95,-5) -- (95,100) {};
\draw[-,thin,dashed] (100,-5) -- (100,100) {};


\node [circle,fill,inner sep=1.2pt] at (0,5) {};
\draw[-,thick] (0,5) -- (8,5) {};
\draw[-,thick] (8,5) .. controls (10,5) and (10,25) .. (20,25) {};
\node [circle,inner sep=1.2pt] at (20,25) {};
\node at (10,-5) {$A$};
\node at (-5,15) {$A$};


\node [circle,fill,inner sep=1.2pt] at (20,60) {};
\draw[-,thick] (20,60) .. controls (25,75) and (30,86) .. (60,100) {};
\node [circle,inner sep=1.2pt] at (60,100) {};
\node at (40,-5) {$B$};
\node at (-5,79) {$B$};


\node [circle,fill,inner sep=1.2pt] at (60,0) {};
\draw[-,thick] (60,0) .. controls (62,4) .. (65,5) {};
\node [circle,inner sep=1.2pt] at (65,5) {};
\node at (62,-5) {$E$};
\node at (-5,2) {$E$};


\node [circle,fill,inner sep=1.2pt] at (65,30) {};
\draw[-,thick] (65,30) --(80,30) {};
\draw[-,thick] (80,30) .. controls (90,30) and (93,55) .. (95,60) {};
\node [circle,inner sep=1.2pt] at (65,70) {};
\node at (79,-5) {$C$};
\node at (-5,46) {$C$};


\node [circle,fill,inner sep=1.2pt] at (95,25) {};
\draw[-,thick] (95,25) -- (100,30) {};
\node [circle,inner sep=1.2pt] at (100,30) {};
\node at (97,-5) {$D$};
\node at (-5,28) {$D$};

\end{tikzpicture}}
\hspace{0.5 cm}
{\begin{tikzpicture}[scale=0.043]


\node at (50,110) {$\cD$};


\draw[-] (0,0) -- (100,0) {};
\draw[-,thin,dashed] (-5,60) -- (100,60) {};
\draw[-,thin,dashed] (-5,25) -- (100,25) {};
\draw[-,thin,dashed] (-5,100) -- (100,100) {};

\draw[-] (0,0) -- (0,100) {};
\draw[-,thin,dashed] (25,-5) -- (25,100) {};
\draw[-,thin,dashed] (65,-5) -- (65,100) {};
\draw[-,thin,dashed] (100,-5) -- (100,100) {};


\node [circle,fill,inner sep=1.2pt] at (0,0) {};
\draw[-,thick] (0,0) -- (10,0) {};
\draw[-,thick] (10,0) .. controls (13,13) and (15,20) .. (25,25) {};
\node [circle,inner sep=1.2pt] at (20,25) {};
\node at (10,-5) {$A$};
\node at (-5,13) {$A$};


\node [circle,fill,inner sep=1.2pt] at (25,60) {};
\draw[-,thick] (25,60) .. controls (30,75) and (46,96) .. (65,100) {};

\node [circle,inner sep=1.2pt] at (65,100) {};
\node at (45,-5) {$B$};
\node at (-5,79) {$B$};


\node [circle,fill,inner sep=1.2pt] at (65,0) {};


\node [circle,fill,inner sep=1.2pt] at (65,25) {};
\draw[-,thick] (65,25) --(80,25) {};
\draw[-,thick] (80,25) .. controls (90,25) and (95,50) .. (100,60) {};
\node [circle,inner sep=1.2pt] at (65,70) {};
\node at (83,-5) {$C$};
\node at (-5,43) {$C$};


\node [circle,fill,inner sep=1.2pt] at (100,25) {};

\end{tikzpicture}}
\end{center}
\caption{In the center a GIET $f$ with five intervals, on the right a degeneration $D$, which can be obtained from $f$ shrinking to zero the branches $f_D$ and $f_E$. This gives rise to two points, which are the singular part of $D$. The graph of $f$ is close to $D$ in the Hausdorff distance. On the left a GIET $g$ with five intervals, acting continuously on the union
$I^{t}_C(g)\cup I^{t}_E(g)$. IETs with a flat piece are a special example of maps like $g$. Interpreting these maps as GIETs with one more interval enables to apply Theorem~\ref{TheoremFullFamilyTheorem}.}
\label{FigureDegenerationGIETs}
\end{figure}
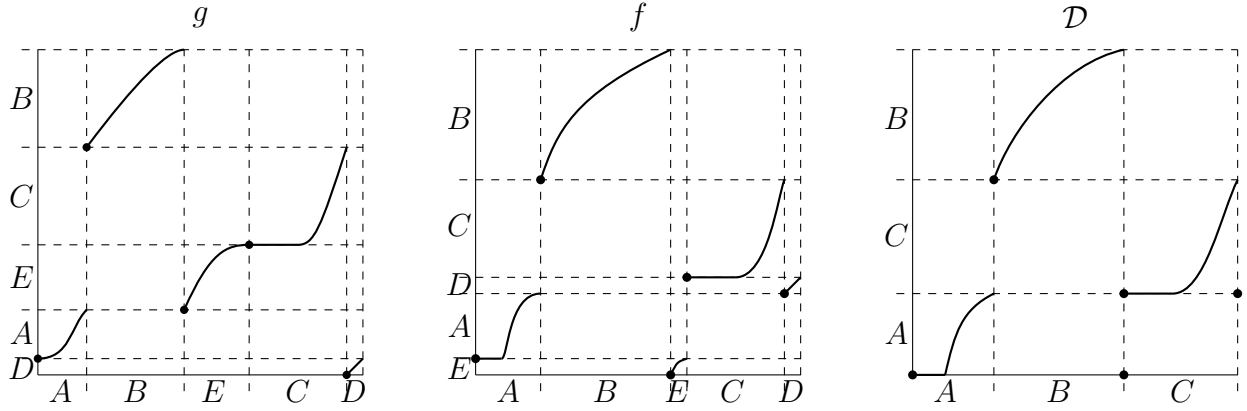

\subsubsection{Distance on $\cC\big(\pi,[0,1)\big)$}
\label{SectionDistanceOnGIETs}

The box norm $\|(x,y)\|:=\max\big\{|x|,|y|\big\}$ induces the corresponding distance $d(\cdot,\cdot)$ on $\RR^2$ and the \emph{Hausdorff distance} 
$\distance_H(E,F)$ between closed sets $E,F\subset \closure(I)^2$, that is 
$$
\distance_H(E,F):=
\max\left\{
\quad
\sup_{x\in E}\inf_{y\in F}d(x,y)
\quad,\quad
\sup_{x\in F}\inf_{y\in E}d(x,y)
\quad
\right\}.
$$

For $f\in\cC(\pi,I)$ and $\alpha\in\cA$ denote by $G_{f_\alpha}$ the graph of the continuous branch $f_\alpha$ in  Equation~\eqref{EquationPiecewiseHomeomorphism} and by $\overline{G_{f_\alpha}}$ its closure in $\closure(I)^2$. The set $\cC(\pi,I)$ is a metric space setting 
\begin{equation}
\label{EquationDinstanceGIETs}
\distance(f,f'):=
\sum_{\alpha\in\cA}
\distance_H(\overline{G_{f_\alpha}},\overline{G_{f'_\alpha}})+\big|j_\alpha(f)-j_\alpha(f')\big|.
\end{equation}

Two maps $f,f'$ are close in such metric if for any $\alpha\in\cA$ both the closures of the graphs of the branches $f_\alpha$, $f'_\alpha$ and the flat intervals $L_\alpha(f), L_\alpha(f')$ are close in the Hausdorff distance.

\subsubsection{Degenerations}
\label{SectionDescriptionAllowedDegenerations}

Consider a combinatorial datum $\pi$ over $\cA$ and a sub-alphabet $\cA'\subset\cA$ with $d'\leq d-1$ elements. 
For $\epsilon\in\{t,b\}$ let 
$
\rho_\epsilon:\pi_\epsilon(\cA')\to\{1,\dots,d'\}
$ 
be the corresponding increasing bijections. Define 
$
\pi'_t,\pi'_b:\cA'\to\{1,\dots,d'\}
$ 
by 
$
\pi'_t:=\rho_t\circ\pi_t|_{\cA'}
$ 
and 
$
\pi'_b:=\rho_b\circ\pi_b|_{\cA'}
$. 
The combinatorial datum $\pi'=(\pi'_t,\pi'_b)$ over $\cA'$ is called a \emph{reduction} of $\pi$. In general $\pi'$ is not admissible even if $\pi$ is.

Let $I$ be a right-open interval. A \emph{degeneration} of a GIET in $\cC(\pi,I)$ is a pair $D=(f,S)$. 
The \emph{regular part} of $D$ is a GIET $f\in\cC(\pi',I)$ whose combinatorial datum $\pi'$ is a reduction of $\pi$ with alphabet $\cA'\subset\cA$. The \emph{singular part} of $D$ is a finite set $S\subset\closure(I)^2$ of 
$|\cA|-|\cA'|$ points with coordinates on the boundary of the continuity intervals of $f$ and $f^{-1}$, that is 
$$
S\subset
\big(\{u^t_\alpha(f),\alpha\in\cA'\}\cup\{\sup I\}\big)
\times
\big(\{u^b_\alpha(f),\alpha\in\cA'\}\cup\{\sup I\}\big).
$$
Let $\cD(\pi',I)$ be the set of degenerations $D$ whose regular part $f:I\to I$ has combinatorial datum $\pi'$. The right part of Figure~\ref{FigureDegenerationGIETs} shows that such degenerations can be obtained as limits of sequences of maps in $\cC(\pi,I)$: the branches corresponding to letters $\alpha\in\cA'$ remain maps as in Equation~\eqref{EquationPiecewiseHomeomorphism} also in the limit, while the branches corresponding to letters in $\cA\setminus\cA'$ are contracted to a point in $S$. Finally consider the set  
$$
\widehat{\cC}(\pi,I):=
\cC(\pi,I)\cup
\bigcup_{\pi'\textrm{ reduction of }\pi}\cD(\pi',I).
$$ 
Consider $D=(f,S)$ in $\widehat{\cC}(\pi,I)$ and let $\cA(D)\subset\cA$ be the alphabet of the regular part $f$ of $D$, where $\cA(D)=\cA$ and $f\in\cC(\pi,I)$ if $S=\emptyset$. 
When $S\not=\emptyset$, denote its points by $s_\alpha=(x_\alpha,y_\alpha)$ with $\alpha\in\cA\setminus\cA(D)$. For any 
$\alpha\in\cA$ set 
$
D_\alpha:=\overline{\text{Graph}(f_\alpha)}
$ 
if $\alpha\in\cA(D)$, and $D_\alpha:=\{s_\alpha\}$ otherwise. Set also $j_\alpha(D):=x_\alpha$ for any 
$\alpha\in\cA\setminus\cA(D)$.
The distance in Equation~\eqref{EquationDinstanceGIETs} extends to pairs $D=(f,S)$ and $D'=(f',S')$ in 
$\widehat{\cC}(\pi,I)$ setting 
$$
\distance(D,D'):=
\sum_{\alpha\in\cA}
\distance_H\Big(D_\alpha,D'_\alpha)\Big)
+
\big|j_\alpha(D)-j_\alpha(D')\big|.
$$

\subsubsection{The full family Theorem} 

IETs with the \emph{Keane property} represent the good generalization of rotations with irrational rotation number (see \S~\ref{SectionIterationRauzyInductionGIETs} for definitions). If $f$ is a map in 
$\cC\big(\pi,[0,1)\big)$ and $T:[0,1)\to[0,1)$ is a Keane IET with combinatorial datum $\pi$, then we say that $f$ is \emph{semi-conjugated} to $T$ if there exists a continuous, surjective and non-decreasing map $h:[0,1)\to[0,1)$ such that 
\begin{equation}
\label{EquationSemiconjugation}
h\circ f=T\circ h.
\end{equation}
According to Proposition~\ref{PropositionSemiconjugation}, the condition above holds if and only if $T$ and $f$ have the same \emph{Rauzy renormalization path}. In this paper we define parameter families of GIETs in $\cC\big(\pi,[0,1)\big)$ such that the following holds: if $T$ is any Keane IET with combinatorial datum $\pi$, then there exists a map $f$ in the family which is semi-conjugated to $T$. Definition \ref{DefinitionFullFamily} describes analytic conditions required for such parametric family. Our methods require also an additional combinatorial assumption on the underlying combinatorial datum $\pi$, or more precisely on its Rauzy class $\cR$, which must contain a \emph{cyclic} element, in the sense of Definition~\ref{DefinitionCyclicCombinatorialDatum} below. This assumption seems to be satisfied by most Rauzy classes, and we don't know any counterexample (see Question \ref{QuestionCyclicElements} below).

\begin{definition}
\label{DefinitionFullFamily}
A \emph{full family} over $\pi$ is a continuous map 
$
F:\Delta^d\to\cC\big(\pi,[0,1)\big)
$ 
satisfying the properties below.
\begin{enumerate}
\item
The composition with the singular values map gives a bijective map 
$$
u^b\circ F:\Delta^d\to\nabla_\pi
$$
which extends to an homeomorphism of $\overline{\Delta^d}$ onto $\overline{\nabla_\pi}$.
\item
The composition with the singular points map also gives a bijective map 
$$
u^t\circ F:\Delta^d\to\nabla^t_\pi
$$
which extends to an homeomorphism of $\overline{\Delta^d}$ onto $\overline{\nabla^t_\pi}$.
\item
The map $F$ has a continuous extension 
$$
\widehat{F}:\overline{\Delta^d}\to\widehat{\cC}\big(\pi,[0,1)\big).
$$
\end{enumerate}
\end{definition}

The maps in Equations \eqref{EquationSingularValuesMap} and \eqref{EquationSingularPointsMap} are continuous with respect to the distance in Equation~\eqref{EquationDinstanceGIETs}, so the maps $u^{t/b}\circ F$ in Points (1) and (2) of Definition~\ref{DefinitionFullFamily} are continuous too. By Brouwer's Invariance of Domain Theorem the maps 
$u^{t/b}\circ F$ are bijective if and only if they are homeomorphisms. Fix $\alpha\in\cA$. According to Point (3), when $f$ varies in $F(\Delta^d)$ the maps $f_\alpha:I^t_\alpha(f)\to I^b_\alpha(f)$ in Equation \eqref{EquationPiecewiseHomeomorphism} and their inverse branches $f_\alpha^{-1}:I^b_\alpha(f)\to J_\alpha(f)$ form relatively compact families. A variation of the easy part of the Ascoli-Arzel\`a Theorem (see \S~\ref{AppendixUniformContinuityModulus}) implies that we have a uniform modulus of continuity, that is $\omega:[0,1]\to[0,1]$ monotone with $\omega(\delta)\to0$ for $\delta\to0$ such that for any 
$f\in\cC\big(\pi,[0,1)\big)$ and $\alpha\in\cA$ we have
\begin{equation}
\label{EquationUniformContinuityModulus}
\left\{
\begin{array}{l}
\big|f_\alpha(x)-f_\alpha(x')\big|\leq\omega\big(|x-x'|\big)
\quad\text{for any}\quad 
x,x'\in I^t_\alpha(f)
\\
\big|f_\alpha^{-1}(y)-f_\alpha^{-1}(y')\big|\leq\omega\big(|y-y'|\big)
\quad\text{for any}\quad 
y,y'\in I^b_\chi(f).
\end{array}
\right.
\end{equation}
Point (3) also implies the uniform continuity of the map $F$. Moreover, according to Points (1) and (3), if 
$\lambda\in\partial\Delta^d$ then 
$
u^b\big(\widehat{F}(\lambda)\big)\in\partial\nabla_{\pi}
$, 
so that $\widehat{F}(\lambda)\not\in\cC\big(\pi,[0,1)\big)$, that is the map $F$ is proper. Finally, the definition of the distance in Equation \eqref{EquationDinstanceGIETs} implies that for any $\alpha$ we have a continuous map 
$j_\alpha\circ F:\Delta^d\to[0,1)$, $\lambda\mapsto j_\alpha\big(F(\lambda)\big)$.

\begin{definition}
\label{DefinitionCyclicCombinatorialDatum}
Let $\pi$ be a combinatorial datum and   
$
\pi_t,\pi_b:\cA\to\{1,\dots,d\}
$ 
be the corresponding pair of bijections. Let 
$
\sigma=\sigma(\pi):\{1,\dots,d\}\to\{1,\dots,d\}
$ 
be the bijection given by 
$
\sigma:=\pi_b\circ\pi_t^{-1}
$. 
We say that $\pi$ is a \emph{cyclic} combinatorial datum if $\sigma(\pi)$ is cyclic of maximal order $d$.
\end{definition}

\begin{theorem}
\label{TheoremFullFamilyTheorem}
Fix an admissible combinatorial datum $\pi$ and assume that the Rauzy class $\cR$ of $\pi$ contains a cyclic combinatorial datum $\pi^{(\ast)}$. 
Let 
$
F:\Delta^d\to\cC\big(\pi,[0,1)\big)
$ 
be a full family as in Definition \ref{DefinitionFullFamily}. Then for any Keane $T$ with combinatorial datum $\pi$ there exists $\lambda\in\Delta^d$ such that $f_\lambda:=F(\lambda)$ is semi-conjugated to $T$.
\end{theorem}

Theorem \ref{TheoremFullFamilyTheorem} above covers a version of Theorem 1.5 in \cite{MarchesePalmisano}, indeed it is enough to consider a family $F$ as in Definition \ref{DefinitionFullFamily} with values in the subclass of $\cC\big(\pi,[0,1)\big)$ defined by Equation \eqref{EquationGIETsDefinedByHomeomorphisms}. Both this paper and \cite{MarchesePalmisano} use the assumption on the existence of cyclic element in Rauzy classes, and the next Question \ref{QuestionCyclicElements} rises naturally.

\begin{question}
\label{QuestionCyclicElements}
Does every Rauzy class $\cR$ contain a cyclic element?
\end{question}

\begin{remark}
\label{RemarkCuttingFlatIntervalMarkingPoints}
Let $g:[0,1)\to[0,1)$ be a map as in Equation~\eqref{EquationPerturbationOfIdentity} and $(\pi,\lambda)$ be data with 
$\lambda\in\Delta^{d}$ defining an IET with flat piece $f:=T\circ g$, where $T=T(\pi,\lambda)$. Let $0<a,b,c,d<1$ be the points as in Equation~\eqref{EquationPerturbationOfIdentity} and $\chi$ be the letter such that $[a,d)\subset I^t_\chi$. Split such interval as
$
I^t_\chi=[\inf I^t_\chi,b)\cup[b,\sup I^t_\chi)
$  
and replace $\chi$ by the pair of letters $\chi_L$, $\chi_R$. Setting
$
I^t_{\chi_L}:=[\inf I^t_\chi,b)
$ 
and 
$
I^t_{\chi_R}:=[b,\sup I^t_\chi)
$ 
we obtain a combinatorial datum $\widehat{\pi}$ over $d+1$ letters, and $f$ can be considered as an element of 
$
\cC\big(\widehat{\pi},[0,1)\big)
$. 
In particular we can deform $f$ inside a full family parametrized by 
$
\widehat{\lambda}\in\Delta^{d+1}
$, 
so that any Keane IET with combinatorial datum $\pi$ is semi-conjugated to some element 
$f_{\widehat{\lambda}}$ of the family. The left part of Figure~\ref{FigureDegenerationGIETs} represents an example with $\chi_L=E$ and $\chi_R=C$.
\end{remark}

\subsection{Content of this paper}

In \S~\ref{SectionBackgroundRauzyVeech} we introduce the Rauzy-Veech algorithm for maps in 
$\cC\big(\pi,[0,1)\big)$. In \S~\ref{SectionDynamicalPartition} we define the dynamical partition of $[0,1)$ induced by a map $f$ in $\cC\big(\pi,[0,1)\big)$ whose first $r$ steps of the algorithm follow a fixed Rauzy path of length $r$.

In \S~\ref{SectionProofFullFamilyTheorem} we prove Theorem \ref{TheoremFullFamilyTheorem} assuming Theorem \ref{TheoremRealizationRauzyPath}, which states that for any fixed length $r$ Rauzy path 
$\gamma$ we have a non-empty set $\Delta_\gamma\subset\Delta^d$ of parameters $\lambda$ such that $f=F(\lambda)$ follows the $r$ steps of the algorithm encoded by $\gamma$. Here $F$ parametrizes the family in Definition \ref{DefinitionFullFamily}.

In \S~\ref{SectionThurstonMap} we fix a Rauzy path $\gamma$ and we introduce the space of configurations, which is an open subset $\cO_\gamma$ in the high dimensional simplex $\Delta^N$, where $N$ is the number of atoms of the dynamical partition associated to $\gamma$. We also define the \emph{Thurston map} $\cT_\gamma$ on $\cO_\gamma$, which is uniformly continuous, so it extends to a continuous map on the closure of $\cO_\gamma$. 

In \S~\ref{SectionExistenceFixedPointThurstonMap} we study the behaviour of the Thurston map on the boundary of the configuration space, proving that it cannot fix any point in $\partial\cO_\gamma$. This step requires the existence of cyclic elements in Rauzy classes. Moreover we prove that the closure of $\cO_\gamma$ is a deformation retract of the closure of $\Delta^N$ (Lemma \ref{LemmaRetraction}). Brouwer's fixed point Theorem implies that the extension of $\cT_\gamma$ fixes a point in the closure of $\cO_\gamma$, and such fixed point belongs to $\cO_\gamma$ because of the boundary behaviour. According to Lemma \ref{LemmaFixedPointsThurstonAndRauzyPath} the fixed point of $\cT_\gamma$ produces a parameter in $\Delta_\gamma$, proving Theorem \ref{TheoremRealizationRauzyPath}.

In \S~\ref{SectionTranslationSurfaces} we recall Veech's \emph{zippered rectangles} construction.  
In \S~\ref{SectionFlowBox} we build a model for the saddle-sink node inside a fixed box in $\RR^2$.
In \S~\ref{SectionConstructionCherryFlows} we build a family of marked translation surfaces parametrized by length data $\widehat{\lambda}\in\Delta^{d+1}$. Each surface $X_{\widehat{\lambda}}$ is endowed with a smooth vector field 
$Y_{\widehat{\lambda}}$, which is everywhere equal to (a time change of) the vertical flow of $X_{\widehat{\lambda}}$ with the exception of a little box, where $Y_{\widehat{\lambda}}$ is a change of variable of the flow in  \S~\ref{SectionFlowBox}. Taking the first return map of $Y_{\widehat{\lambda}}$ we get a parameter family 
$\widehat{\lambda}\mapsto f_{\widehat{\lambda}}$ of IETs with a flat piece which satisfies the requirements in Definition \ref{DefinitionFullFamily}. This proves Theorem \ref{TheoremMainTheoremCherryFlow} 
(see \S~\ref{SectionEndProofTheoremCherryFlows}).

In \S~\ref{SectionEstimatesHausdorffDistance} and \S~\ref{AppendixUniformContinuityModulus} we prove basic estimates for the Hausdorff distance used throughout the paper.

\section{The Rauzy-Veech algorithm}
\label{SectionBackgroundRauzyVeech}

\subsection{Rauzy classes}
\label{SectionRauzyClasses}

Let $\cA$ be a finite alphabet with $d\geq2$ letters. We define two operations $R^t$ and $R^b$ on the set of admissible combinatorial data $\pi$ over $\cA$, where the symbols "$t$" and "$b$" stand for \emph{top} and \emph{bottom} respectively. It is practical to introduce the variable 
$
\varepsilon\in\{t,b\}
$. 
Fix an admissible combinatorial datum $\pi=(\pi_t,\pi_b)$ and let $\alpha_t$ and $\alpha_b$ be the letters in $\cA$ such that 
$
\pi_t(\alpha_t)=\pi_b(\alpha_b)=d
$. 
Below, for $\varepsilon\in\{t,b\}$, we describe the combinatorial datum  
$
\widetilde{\pi}=(\widetilde{\pi}_t,\widetilde{\pi}_b)
$, 
where 
$
\widetilde{\pi}:=R^\varepsilon(\pi)
$.

\begin{description}
\item
[Top operation]
The letter $\alpha_t$ is said the \emph{winner} of the top operation $R^t$ and $\alpha_b$ is said the \emph{loser}. The top operation 
$
\widetilde{\pi}:=R^t(\pi)
$ 
leaves invariant $\pi_t$, that is 
$
\widetilde{\pi}_t:=\pi_t
$, 
and its action on $\pi_b$ is defined by
\begin{eqnarray*}
&&
\widetilde{\pi}_b(\chi):=\pi_b(\chi)
\quad
\textrm{ for }
\quad
1\leq\pi_b(\chi)\leq\pi_b(\alpha_t)
\\
&&
\widetilde{\pi}_b(\alpha_b):=\pi_b(\alpha_t)+1
\\
&&
\widetilde{\pi}_b(\chi):=\pi_b(\chi)+1
\quad
\textrm{ for }
\quad
\pi_b(\alpha_t)+1\leq\pi_b(\chi)\leq d-1.
\end{eqnarray*}
\item
[Bottom operation]
The letter $\alpha_b$ is said the \emph{winner} of the bottom operation $R^b$ and $\alpha_t$ is said the \emph{loser}. The top operation 
$
\widetilde{\pi}:=R^b(\pi)
$ 
leaves invariant $\pi_b$, that is 
$
\widetilde{\pi}_b:=\pi_b
$, 
and its action on $\pi_t$ is defined by
\begin{eqnarray*}
&&
\widetilde{\pi}_t(\chi):=\pi_t(\chi)
\quad
\textrm{ for }
\quad
1\leq\pi_t(\chi)\leq\pi_t(\alpha_b)
\\
&&
\widetilde{\pi}_t(\alpha_t):=\pi_t(\alpha_b)+1
\\
&&
\widetilde{\pi}_t(\chi):=\pi_t(\chi)+1
\quad
\textrm{ for }
\quad
\pi_t(\alpha_b)+1\leq\pi_t(\chi)\leq d-1.
\end{eqnarray*}
\end{description}

It is easy to check that both $R^b(\pi)$ and $R^t(\pi)$ are admissible if $\pi$ is. A \emph{Rauzy class} $\cR$ is a set of admissible combinatorial data which is invariant both under $R^t$ and $R^b$ and which is minimal with this property. The \emph{Rauzy diagram} of $\cR$ is the connected oriented graph whose vertices are the elements of $\cR$ and whose elementary oriented arcs $\gamma$, or \emph{arrows}, correspond to Rauzy elementary operations. The set of arrows $\gamma$ in the Rauzy diagram of $\cR$ is in bijection with the set of pairs $(\pi,\varepsilon)$ with $\pi\in\cR$ and $\varepsilon\in\{t,b\}$. A concatenation of $r$ compatible arrows $\gamma_1,\dots,\gamma_r$ in the Rauzy diagram is called a \emph{Rauzy path} and is denoted $\gamma=\gamma_1\ast\dots\ast\gamma_r$. If $\pi,\pi'$ in $\cR$ are respectively the initial and the final combinatorial data in such chain of combinatorial operations, we write also $\gamma:\pi\to\pi'$. If a path $\gamma$ is concatenation of $r$ simple arrows, we say that $\gamma$ has length $r$. Length one paths are arrows. Elements of $\cR$ are identified with trivial length-zero paths. Let 
$
\{e_\chi\}_{\chi\in\cA}
$ 
be the canonical basis of $\RR^\cA$. For any finite path $\gamma$ define a linear map $B_\gamma\in\slgroup(d,\ZZ)$ as follows. If $\gamma$ is trivial then $B_\gamma:=id$. If $\gamma$ is an arrow with winner $\alpha$ and loser $\beta$ set
$$
B_\gamma e_\alpha=e_\alpha+e_\beta
\quad
\textrm{ and }
\quad
B_\gamma e_\chi=e_\chi
\textrm{ for }
\chi\not=\alpha.
$$
Then extend the definition to paths so that for any concatenation $\gamma_1\ast\gamma_2$ we have 
$$
B_{\gamma_1\ast\gamma_2}=B_{\gamma_2}\cdot B_{\gamma_1}.
$$

\subsection{The Rauzy-Veech map}
\label{SectionRauzyInductionGIETs}

Let $\cR$ be a Rauzy class over $\cA$ and consider $\pi\in\cR$. Consider a right-open interval $I$ and a map 
$f\in\cC(\pi,I)$. Consider the two corresponding partitions 
$
\cP_t(f)=\{I^t_\alpha(f)\}_{\alpha\in\cA}
$ 
and 
$
\cP_b(f)=\{I^b_\alpha(f)\}_{\alpha\in\cA}
$ 
of $I$ as in Equation~\eqref{EquationDynamicalPartitionTimeZero}. For $\alpha\in\cA$ consider the singular points $u_\alpha^t(f)$ and the singular values $u_\alpha^b(f)$ of $f$ given by 
Equation~\eqref{EquationSingularitiesTimeZero}, then consider also $j_\alpha(f)=\inf J_\alpha(f)=\sup L_\alpha(f)$ as in 
Equation \eqref{EquationRightEndpointFlatInterval}. As in \S~\ref{SectionRauzyClasses}, let $\alpha_t$ and $\alpha_b$ be the letters with $\pi_t(\alpha_t)=\pi_b(\alpha_b)=d$. 

\subsubsection{A condition for renormalizability}
\label{SectionRauzyInductionGIETs(Conditions)}

Suppose that the following condition is satisfied
\begin{equation}
\label{EquationConditionRenormalizabilityDifferentSingularities}
\text{either }\quad j_{\alpha_t}(f)<u_{\alpha_b}^b(f)
\quad\text{ or }\quad
u_{\alpha_t}^t(f)>u_{\alpha_b}^b(f),
\end{equation}
then assign a value to the variable 
$
\varepsilon=\varepsilon(f)\in\{t,b\}
$ 
according to the two cases below: 
\begin{eqnarray*}
&&
\varepsilon(f):=t
\quad
\Leftrightarrow
\quad
j_{\alpha_t}(f)<u_{\alpha_b}^b(f)
\\
&&
\varepsilon(f):=b
\quad
\Leftrightarrow
\quad
u_{\alpha_t}^t(f)>u_{\alpha_b}^b(f).
\end{eqnarray*}
For the sub-class of maps in Equation~\eqref{EquationGIETsDefinedByHomeomorphisms} Condition~\eqref{EquationConditionRenormalizabilityDifferentSingularities} takes the form
$$
u_{\alpha_t}^t(f)\not=u_{\alpha_b}^b(f).
$$
In particular, if $f=T(\pi,\lambda)$ is the IET determined by combinatorial-length data $(\pi,\lambda)$, so that its singularities depend on $(\pi,\lambda)$ as in  
Equation~\eqref{EquationSingularitiesTimeZeroIET}, we have 
$
\varepsilon(f)=t\Leftrightarrow\lambda_{\alpha_t}>\lambda_{\alpha_b}
$ 
and 
$
\varepsilon(f)=b\Leftrightarrow\lambda_{\alpha_t}<\lambda_{\alpha_b}
$.

\subsubsection{Formulae for first return}
\label{SectionRauzyInductionGIETs(Formulae)}

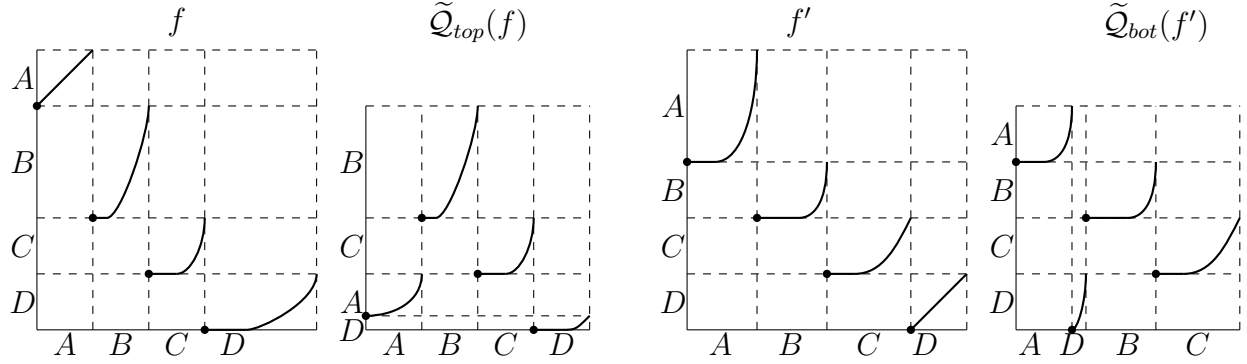
\begin{figure}
\begin{center}
{\begin{tikzpicture}[scale=0.037]


\node at (50,110) {$f$};

\draw[-] (0,0) -- (100,0) {};
\draw[-,thin,dashed] (0,20) -- (100,20) {};
\draw[-,thin,dashed] (0,40) -- (100,40) {};
\draw[-,thin,dashed] (0,80) -- (100,80) {};
\draw[-,thin,dashed] (0,100) -- (100,100) {};

\draw[-] (0,0) -- (0,100) {};
\draw[-,thin,dashed] (20,0) -- (20,100) {};
\draw[-,thin,dashed] (40,0) -- (40,100) {};
\draw[-,thin,dashed] (60,0) -- (60,100) {};
\draw[-,thin,dashed] (100,0) -- (100,100) {};


\node [circle,fill,inner sep=1.1pt] at (0,80) {};
\draw[-,thick] (0,80) -- (20,100) {};
\node at (10,-5) {$A$};
\node at (-5,90) {$A$};


\node [circle,fill,inner sep=1.1pt] at (20,40) {};
\draw[-,thick] (20,40) -- (25,40) {};
\draw[-,thick] (25,40) .. controls (30,40) and (40,70) .. (40,80) {};
\node at (30,-5) {$B$};
\node at (-5,60) {$B$};


\node [circle,fill,inner sep=1.1pt] at (40,20) {};
\draw[-,thick] (40,20) -- (50,20) {};
\draw[-,thick] (50,20) .. controls (55,20) and (60,30) .. (60,40) {};
\node at (50,-5) {$C$};
\node at (-5,30) {$C$};


\node [circle,fill,inner sep=1.1pt] at (60,0) {};
\draw[-,thick] (60,0) -- (75,0) {};
\draw[-,thick] (75,0) .. controls (80,0) and (100,10) .. (100,20) {};
\node at (70,-5) {$D$};
\node at (-5,10) {$D$};

\end{tikzpicture}}
{\begin{tikzpicture}[scale=0.037]


\node at (40,110) {$\widetilde{\cQ}_{top}(f)$};

\draw[-] (0,0) -- (80,0) {};
\draw[-,thin,dashed] (0,5) -- (80,5) {};
\draw[-,thin,dashed] (0,20) -- (80,20) {};
\draw[-,thin,dashed] (0,40) -- (80,40) {};
\draw[-,thin,dashed] (0,80) -- (80,80) {};

\draw[-] (0,0) -- (0,80) {};
\draw[-,thin,dashed] (20,0) -- (20,80) {};
\draw[-,thin,dashed] (40,0) -- (40,80) {};
\draw[-,thin,dashed] (60,0) -- (60,80) {};
\draw[-,thin,dashed] (80,0) -- (80,80) {};


\node [circle,fill,inner sep=1.1pt] at (0,5) {};
\draw[-,thick] (0,5) .. controls (10,5) and (20,10) .. (20,20) {};
\node at (10,-5) {$A$};
\node at (-5,10) {$A$};


\node [circle,fill,inner sep=1.1pt] at (20,40) {};
\draw[-,thick] (20,40) -- (25,40) {};
\draw[-,thick] (25,40) .. controls (30,40) and (40,70) .. (40,80) {};
\node at (30,-5) {$B$};
\node at (-5,60) {$B$};


\node [circle,fill,inner sep=1.1pt] at (40,20) {};
\draw[-,thick] (40,20) -- (50,20) {};
\draw[-,thick] (50,20) .. controls (55,20) and (60,30) .. (60,40) {};
\node at (50,-5) {$C$};
\node at (-5,30) {$C$};


\node [circle,fill,inner sep=1.1pt] at (60,0) {};
\draw[-,thick] (60,0) -- (70,0) {};
\draw[-,thick] (70,0) .. controls (75,0) .. (80,5) {};
\node at (70,-5) {$D$};
\node at (-5,0) {$D$};

\end{tikzpicture}}
\hspace{0.5 cm}
{\begin{tikzpicture}[scale=0.037]



\node at (40,110) {$f'$};

\draw[-] (0,0) -- (100,0) {};
\draw[-,thin,dashed] (0,20) -- (100,20) {};
\draw[-,thin,dashed] (0,40) -- (100,40) {};
\draw[-,thin,dashed] (0,60) -- (100,60) {};
\draw[-,thin,dashed] (0,100) -- (100,100) {};

\draw[-] (0,0) -- (0,100) {};
\draw[-,thin,dashed] (25,0) -- (25,100) {};
\draw[-,thin,dashed] (50,0) -- (50,100) {};
\draw[-,thin,dashed] (80,0) -- (80,100) {};
\draw[-,thin,dashed] (100,0) -- (100,100) {};


\node [circle,fill,inner sep=1.1pt] at (0,60) {};
\draw[-,thick] (0,60) -- (10,60) {};
\draw[-,thick] (10,60) .. controls (20,60) and (25,80) .. (25,100) {};
\node at (12,-5) {$A$};
\node at (-5,80) {$A$};


\node [circle,fill,inner sep=1.1pt] at (25,40) {};
\draw[-,thick] (25,40) -- (40,40) {};
\draw[-,thick] (40,40) .. controls (45,40) and (50,45) .. (50,60) {};
\node at (37,-5) {$B$};
\node at (-5,50) {$B$};


\node [circle,fill,inner sep=1.1pt] at (50,20) {};
\draw[-,thick] (50,20) -- (60,20) {};
\draw[-,thick] (60,20) .. controls (70,20) and (75,30) .. (80,40) {};
\node at (65,-5) {$C$};
\node at (-5,30) {$C$};


\node [circle,fill,inner sep=1.1pt] at (80,0) {};
\draw[-,thick] (80,0) -- (100,20) {};
\node at (85,-5) {$D$};
\node at (-5,10) {$D$};

\end{tikzpicture}}
{\begin{tikzpicture}[scale=0.037]



\node at (50,110) {$\widetilde{\cQ}_{bot}(f')$};

\draw[-] (0,0) -- (80,0) {};
\draw[-,thin,dashed] (0,20) -- (80,20) {};
\draw[-,thin,dashed] (0,40) -- (80,40) {};
\draw[-,thin,dashed] (0,60) -- (80,60) {};
\draw[-,thin,dashed] (0,80) -- (80,80) {};

\draw[-] (0,0) -- (0,80) {};
\draw[-,thin,dashed] (20,0) -- (20,80) {};
\draw[-,thin,dashed] (25,0) -- (25,80) {};
\draw[-,thin,dashed] (50,0) -- (50,80) {};
\draw[-,thin,dashed] (80,0) -- (80,80) {};


\node [circle,fill,inner sep=1.1pt] at (0,60) {};
\draw[-,thick] (0,60) -- (10,60) {};
\draw[-,thick] (10,60) .. controls (15,60) and (20,65) .. (20,80) {};
\node at (5,-5) {$A$};
\node at (-5,70) {$A$};


\node [circle,fill,inner sep=1.1pt] at (25,40) {};
\draw[-,thick] (25,40) -- (40,40) {};
\draw[-,thick] (40,40) .. controls (45,40) and (50,45) .. (50,60) {};
\node at (40,-5) {$B$};
\node at (-5,50) {$B$};


\node [circle,fill,inner sep=1.1pt] at (50,20) {};
\draw[-,thick] (50,20) -- (60,20) {};
\draw[-,thick] (60,20) .. controls (70,20) and (75,30) .. (80,40) {};
\node at (65,-5) {$C$};
\node at (-5,30) {$C$};


\node [circle,fill,inner sep=1.1pt] at (20,0) {};
\draw[-,thick] (20,0) .. controls (22,0) and (25,10) .. (25,20) {};
\node at (20,-5) {$D$};
\node at (-5,10) {$D$};

\end{tikzpicture}}

\end{center}
\caption{On the left a GIET $f$ with $\varepsilon(f)=t$ and its image 
$
\widetilde{f}=\widetilde{\cQ}_{top}(f)
$. 
The branch $\widetilde{f}_D$ is the restriction of $f_D$ to $[u^t_D,u^b_A)$. On the right, a GIET $f'$ with $\varepsilon(f')=b$ and its image 
$
\widetilde{f}'=\widetilde{\cQ}_{bot}(f')
$.}
\label{FigureRauzyVeech}
\end{figure}

Assume that Condition~\eqref{EquationConditionRenormalizabilityDifferentSingularities} holds, then let 
$
\widetilde{\pi}:=R^{\varepsilon(f)}(\pi)
$ 
be the combinatorial datum and 
$
\gamma:\pi\to\widetilde{\pi}
$ 
be the arrow in the Rauzy diagram of $\cR$ corresponding to the pair 
$
(\pi,\varepsilon=\varepsilon(f))
$ 
as in \S~\ref{SectionRauzyClasses}. Let $\widetilde{I}\subset I$ be the subinterval defined by 
\begin{equation}
\label{EquationSubintervalRightmostSingularity}
\inf\widetilde{I}=\inf I
\quad
\textrm{ and }
\quad
\sup\widetilde{I}:=\max\{u_{\alpha_t}^t(f),u_{\alpha_b}^b(f)\}.
\end{equation}
Let $\widetilde{f}:\widetilde{I}\to\widetilde{I}$ be the first return map of $f$ to $\widetilde{I}$. The so-called \emph{Rauzy-Veech map} $\widetilde{\cQ}$, which is defined\footnote{Actually, 
$\widetilde{\cQ}$ is defined on the set of those $f\in\cC(\pi,I)$ which satisfy 
Condition~\eqref{EquationConditionRenormalizabilityDifferentSingularities}.} on $\cC(\pi,I)$, expresses the dependence of $\widetilde{f}$ on $f$. We write
$$
f\mapsto \widetilde{\cQ}(f):=\widetilde{f}.
$$
Recall that for any $\alpha\in\cA$ we denote by $f_\alpha$ the continuous branch of $f$ as in 
Equation~\eqref{EquationPiecewiseHomeomorphism}. The explicit expression for $\widetilde{f}$ is given according to the two cases below. 
\begin{enumerate}
\item
If $\varepsilon(f)=t$ then, observing that 
$
I^b_{\alpha_b}=f_{\alpha_b}(I^t_{\alpha_b})\subset J_{\alpha_t} \subset I^t_{\alpha_t}
$, 
we have
\begin{align*}
&
\widetilde{f}(x)=f^2(x)=f_{\alpha_t}\circ f_{\alpha_b}(x)
&
\quad
\textrm{ for any }
\quad
x\in I_{\alpha_b}^t
\\
&
\widetilde{f}(x)=f(x)
&
\quad
\textrm{ for any }
\quad
x\in\widetilde{I}\setminus I_{\alpha_b}^t.
\end{align*}
In this case, which is represented by the picture on the left in 
Figure~\ref{FigureRauzyVeech}, we define the \emph{top branch} $\widetilde{\cQ}_{top}$ of 
$\widetilde{\cQ}$ by setting 
$$
\widetilde{\cQ}_{top}(f):=\widetilde{f}\in\cC(\widetilde{\pi},\widetilde{I}).
$$
\item
If $\varepsilon(f)=b$ then, observing that  
$
I_{\alpha_t}^t\subset I_{\alpha_b}^b=f_{\alpha_b}(I_{\alpha_b})
$, 
we have
\begin{align*}
&
\widetilde{f}(x)=f^2(x)=f_{\alpha_t}\circ f_{\alpha_b}(x)
&
\quad
\textrm{ for any }
\quad
x\in f_{\alpha_b}^{-1}(I_{\alpha_t}^t) 
\\
&
\widetilde{f}(x)=f(x)
&
\quad
\textrm{ for any }
\quad
x\in\widetilde{I}\setminus f_{\alpha_b}^{-1}(I_{\alpha_t}^t).
\end{align*}
In this case, which is represented by the picture on the right in 
Figure~\ref{FigureRauzyVeech} we define the \emph{bottom branch} $\widetilde{\cQ}_{bot}$ of 
$\widetilde{\cQ}$ by setting 
$$
\widetilde{\cQ}_{bot}(f):=\widetilde{f}\in\cC(\widetilde{\pi},\widetilde{I}).
$$
\end{enumerate}

The map $\widetilde{\cQ}$ preserves the sub-class of maps in 
Equation~\eqref{EquationGIETsDefinedByHomeomorphisms}. The set of IETs is also preserved by $\widetilde{\cQ}$, and for $f=T(\pi,\lambda)$ we have 
$
\widetilde{f}=T(\widetilde{\pi},\widetilde{\lambda})
$,  
where the new length datum is given by 
$$
\left\{
\begin{array}{l}
\widetilde{\lambda}_{\alpha}=\lambda_{\alpha}
\quad
\forall
\quad
\alpha\neq\alpha_t
\\
\widetilde{\lambda}_{\alpha_t}=
\lambda_{\alpha_t}-\lambda_{\alpha_b}
\end{array}
\right.
\text{ if }
\quad
\varepsilon(f)=t
\quad;\quad
\left\{
\begin{array}{l}
\widetilde{\lambda}_{\alpha}=\lambda_{\alpha}
\quad
\forall
\quad
\alpha\neq\alpha_b
\\
\widetilde{\lambda}_{\alpha_b}=
\lambda_{\alpha_b}-\lambda_{\alpha_t}
\end{array}
\right.
\text{ if }
\quad
\varepsilon(f)=b.
$$

\subsubsection{Normalized Rauzy-Veech map}
\label{SectionRauzyInductionGIETs(Normalized)}

Consider $f\in\cC\big(\pi,[0,1)\big)$ satisfying Condition~\eqref{EquationConditionRenormalizabilityDifferentSingularities} and let 
$\widetilde{I}=\widetilde{I}(f)$ be the right-open interval given by 
Equation~\eqref{EquationSubintervalRightmostSingularity} applied to $I:=[0,1)$. Let 
$
\psi_f:[0,1)\to\widetilde{I}
$ 
be the uniquely defined increasing affine bijection. Setting 
$
\widetilde{f}:=\widetilde{\cQ}(f)
$, 
define a map 
$
\widehat{f}\in\cC\big(\widetilde{\pi},[0,1)\big)
$ 
by 
\begin{equation}
\label{EquationRenormalizationFirstReturn}
\widehat{f}:=\psi_f^{-1}\circ\widetilde{f}\circ\psi_f.
\end{equation}
Let $\cC_{top}\big(\pi,[0,1)\big)$ and $\cC_{bot}\big(\pi,[0,1)\big)$ be the two subsets of those 
$
f\in\cC\big(\pi,[0,1)\big)
$ 
satisfying Condition~\eqref{EquationConditionRenormalizabilityDifferentSingularities} and such that respectively $\varepsilon(f)=t$ and $\varepsilon(f)=b$. The construction above defines two maps 
\begin{equation}
\label{EquationNormalizedBrancesRauzyMap}
\left\{
\begin{array}{l}
\cQ_{top}:\cC_{top}\big(\pi,[0,1)\big)\to\cC\big(R^t(\pi),[0,1)\big)
\quad;\quad
f\mapsto \cQ_{top}(f):=\widehat{f}
\\
\cQ_{bot}:\cC_{bot}\big(\pi,[0,1)\big)\to\cC\big(R^b(\pi),[0,1)\big)
\quad;\quad
f\mapsto \cQ_{bot}(f):=\widehat{f},
\end{array}
\right.
\end{equation}
which are the two branches of the so-called \emph{normalized Rauzy-Veech map}.

\subsection{Iteration of the Rauzy-Veech map}
\label{SectionIterationRauzyInductionGIETs}

Fix $f\in\cC(\pi,I)$ and $r\in\NN$ and assume that the $r$-th iterated 
$
f^{(r)}:=\widetilde{\cQ}^r(f)
$ 
of the Rauzy-Veech map is defined on $f=f^{(0)}$. Let $\pi^{(r)}\in\cR$ be the combinatorial datum of $f^{(r)}$ and $I^{(r)}$ be the interval where $f^{(r)}$ acts, so that in our notation we write 
$
f^{(r)}\in\cC(\pi^{(r)},I^{(r)})
$, 
which means that $f^{(r)}:I^{(r)}\to I^{(r)}$ is a GIET with combinatorial datum $\pi^{(r)}$. Finally let $\alpha^{(r)}_t$ and $\alpha^{(r)}_b$ be the letters such that 
$
\pi^{(r)}_t(\alpha^{(r)}_t)=\pi^{(r)}_b(\alpha^{(r)}_b)=d
$ 
and assume that $f^{(r)}$ satisfies 
Condition~\eqref{EquationConditionRenormalizabilityDifferentSingularities}, that is
$$
\left\{
\begin{array}{l}
\text{either }\quad
u^b_{\alpha_b}(f^{(r)})>j_{\alpha^{(r)}_t}(f^{(r)})
\quad\text{ equivalently }\quad
\varepsilon(f^{(r)})=t
\\
\text{or }\quad
u^t_{\alpha^{(r)}_t}(f^{(r)})>u^b_{\alpha^{(r)}_b}(f^{(r)})
\quad\text{ equivalently }\quad
\varepsilon(f^{(r)})=b.
\end{array}
\right.
$$
Then the interval $I^{(r+1)}\subset I^{(r)}$ and the map 
$
f^{(r+1)}=\widetilde{\cQ}(f^{(r)})\in\cC(\pi^{(r+1)},I^{(r+1)})
$ 
are defined inductively as in \S~\ref{SectionRauzyInductionGIETs}, that is
\begin{eqnarray*}
&&
\inf I^{(r+1)}=\inf I^{(r)}
\quad
\textrm{ and }
\quad
\sup I^{(r+1)}:=
\max\{u^t_{\alpha^{(r)}_t}(f^{(r)}),u^b_{\alpha^{(r)}_b}(f^{(r)})\}
\\
&&
\pi^{(r+1)}:=R^{\varepsilon(r)}(\pi^{(r)})
\quad
\textrm{ where }
\quad
\varepsilon(r)=\varepsilon(f^{(r)}).
\end{eqnarray*}
The operation above is encoded by the arrow 
$
\gamma_{r+1}:\pi^{(r)}\to\pi^{(r+1)}
$ 
in the Rauzy diagram corresponding to the data 
$
(\pi^{(r)},\varepsilon(f^{(r)}))
$.

\begin{definition}
\label{DefinitionRauzyPathOfGIET}
Consider $f\in\cC(\pi,I)$.
\begin{enumerate}
\item
The map $f$ is said \emph{infinitely renormalizable} if $f^{(r)}$ satisfies 
Condition~\eqref{EquationConditionRenormalizabilityDifferentSingularities} for any $r\in\NN$ and moreover, for any $\alpha\in\cA$, there exists infinitely many $r$ such that $\alpha$ is the winner of the arrow $\gamma_r$.
\item
The \emph{Rauzy path} 
$
\gamma(f,r):\pi^{(0)}\to\pi^{(r)}
$ 
of $f$ up to time $r$ is the path in the Rauzy diagram of $\cR$ obtained by concatenation of the arrows above, that is
$$
\gamma(f,r):=
\gamma_1\ast\dots\ast\gamma_r.
$$
\item
The infinite Rauzy path of $f$ is the infinite concatenation 
$
\gamma(f,+\infty)=\gamma_1\ast\gamma_2\ast\dots
$. 
Since any letter $\alpha$ wins infinitely many times in $\gamma(f,+\infty)$, we say that such path is \emph{infinite complete}. 
\end{enumerate}
\end{definition}

Condition~\eqref{EquationConditionRenormalizabilityDifferentSingularities} implies that the orbits of the singularities $u^{t/b}_\alpha(f)$ of $f$ never meet its flat intervals, so that the latter are never cut during the Rauzy induction. If $T$ is a standard IET, then $T^{(r)}=\widetilde{\cQ}^r(T)$ satisfies 
Condition~\eqref{EquationConditionRenormalizabilityDifferentSingularities} for any $r\in\NN$ if and only if $T$ has the \emph{Keane property}, that is, for any $\alpha$ and $\beta$ and any $n\in\NN$ we have
$$ 
T^n(u^b_\beta)\not=u^t_\alpha.
$$ 
Moreover $T$ satisfies this last condition if and only if the infinite Rauzy path 
$\gamma(T,\infty)$, which a priori is just the infinite concatenation of compatible Rauzy arrows 
$
\gamma_1,\gamma_2,\dots
$, 
is also infinite complete, that is any letter $\alpha$ is the winner of infinitely many arrows in the infinite path (see \cite{Yoccoz}). Nevertheless, it is easy to construct examples of maps in 
$\cC\big(\pi,[0,1)\big)$ such that 
Condition~\eqref{EquationConditionRenormalizabilityDifferentSingularities} is always satisfied, but infinite completeness is false. For this reason, when seeking for a semiconjugaition to a standard IET, infinite completeness must be added as an extra assumption, as in Point (1) of 
Definition~\ref{DefinitionRauzyPathOfGIET}. According to 
Proposition~\ref{PropositionSemiconjugation} below, infinite Rauzy paths are the unique topological invariant of semi-conjugacy classes.

\begin{proposition}
[Poincar\'e, Yoccoz]
\label{PropositionSemiconjugation}
Let $T=T(\pi,\lambda)$ be a Keane IET determined by combinatorial-length data $(\pi,\lambda)$ and consider $f\in\cC(\pi,I)$, where 
$
I:=\big[0,\sum_{\chi\in\cA}\lambda_\chi\big)
$. 
Then we have 
$
\gamma(f,\infty)=\gamma(T,\infty)
$, 
that is $f$ and $T$ have the same Rauzy path, if and only if $f$ is semi-conjugated to $T$.\end{proposition}

\begin{proof}
A proof is given by Proposition~3.4 in \cite{MarchesePalmisano} for the class of maps in Equation \eqref{EquationGIETsDefinedByHomeomorphisms}. For the larger class $\cC\big(\pi,[0,1)\big)$ considered in this paper the generalization is straightforward. 
\end{proof}

\subsection{Dynamically defined partitions}
\label{SectionDynamicalPartition}

Fix $\pi\in\cR$ and let 
$
\gamma=\gamma_1\ast\dots\ast\gamma_r
$ 
be a finite Rauzy path of length $r$. Let $B_\gamma\in\slgroup(d,\ZZ)$ be the matrix defined in \S~\ref{SectionRauzyClasses}. Denote by 
$
\vec{1}\in\NN^\cA
$ 
the vector with all entries equal to $1$, define the integer vector 
$
q^{(r)}:=B_\gamma\vec{1}\in\NN^\cA
$, 
and for $\alpha\in\cA$ let $q^{(r)}_\alpha$ be its $\alpha$-entry. Denoting by $[A]_{\alpha,\beta}$ the entry of a matrix $A$ in row $\alpha$ and column $\beta$, we have
$$
q^{(r)}_\alpha=
\sum_{\chi\in\cA}[B_\gamma]_{\alpha,\chi}.
$$
Fix $f\in\cC(\pi,I)$ and $r\in\NN$ such that  
$
f^{(k)}:I^{(k)}\to I^{(k)}
$ 
satisfies Condition~\eqref{EquationConditionRenormalizabilityDifferentSingularities} for any $k=0,\dots,r$. In particular, in our notation 
$
f^{(r)}\in\cC(\pi^{(r)},I^{(r)})
$. 
Let 
$
\cP_t(f^{(r)})=\{I^t_\alpha(f^{(r)})\}_{\alpha\in\cA}
$ 
be the partition of $I^{(r)}$ associated to $f^{(r)}$ as in Equation~\eqref{EquationDynamicalPartitionTimeZero}. 
Let $\gamma=\gamma(f,r)$ be the Rauzy path as in Definition~\ref{DefinitionRauzyPathOfGIET} and set 
$
q^{(r)}:=B_\gamma\vec{1}
$, 
which of course depends both on $f$ and on $r$. Define the intervals 

\begin{equation}
\label{EquationDynamicalPartitionHigherOrder}
I(f,r,\alpha,i):=f^i\big(I^t_\alpha(f^{(r)})\big)
\quad
\textrm{ for }
\quad
\alpha\in\cA
\quad,\quad
0\leq i\leq q^{(r)}_\alpha-1.
\end{equation}
Since $f^{(r)}:I^{(r)}\to I^{(r)}$ is the first return of $f$ to $I^{(r)}\subset I$, then it is easy to see by an induction argument that the right-open intervals defined above form a partition of $I$, that we denote by $\cP(f,r)$. Explicitly: 
\begin{equation}
\label{EquationDynamicalPartitionOrder(r)}
\cP(f,r):=
\left\{
I(f,r,\alpha,i)
\quad;\quad
\alpha\in\cA
\quad;\quad
0\leq i\leq q^{(r)}_\alpha-1
\right\}.
\end{equation}
The partition above is called the \emph{dynamical partition} of order $r$, and is a refinement of the partition $\cP_t(f)$ in Equation~\eqref{EquationDynamicalPartitionTimeZero}, which in our terminology corresponds to the dynamical partition of order $0$. The following Lemma is a classical fact. A proof can be found in \S~7.5 in \cite{Yoccoz} for maps in the class defined by Equation \eqref{EquationGIETsDefinedByHomeomorphisms} and it can be easily adapted to the class $\cC(\pi,I)$.

\begin{lemma}
\label{LemmaIntervalsPartitionAndMatrix}
The following holds.
\begin{enumerate}
\item
For any $\alpha\in\cA$ we have
$
f^{(r)}|_{I^t_\alpha(f,r)}=f^{q^{(r)}_\alpha}
$ 
and moreover
$$
q^{(r)}_\alpha=
\min
\left\{
i\geq1\quad;\quad f^i\big(I^t_\alpha(f^{(r)})\big)\subset I^{(r)}
\right\}.
$$
\item
For any $\alpha$ and $\beta$ we have
$$
[B_\gamma]_{\alpha,\beta}=
\sharp
\left\{
0\leq i\leq q^{(r)}_\alpha-1
\quad;\quad
I(f,r,\alpha,i)\subset I^t_\beta
\right\}.
$$
\end{enumerate}
\end{lemma}

Letting $f$ vary in $\cC(\pi,I)$ we obtain different partitions of $I$ of order $r$.

\begin{definition}
\label{DefinitionEquivalenceDynamicalPartitions}
For $f,f'\in\cC(\pi,I)$ we say that two dynamical partitions 
$\cP(f,r)$ and $\cP(f',r)$ are \emph{combinatorially equivalent}, and we write 
$
\cP(f,r)\sim\cP(f',r)
$, 
if the intervals $I(f,r,\alpha,i)$ in $\cP(f,r)$ are as in Equation~\eqref{EquationDynamicalPartitionHigherOrder}, moreover the intervals $I(f',r,\alpha,i)$ in $\cP(f',r)$ are labeled by the same indices $\alpha$ and $i$ as for $f$ and we have 
$$
I(f',r,\alpha,i)=
\varphi\big(I(f,r,\alpha,i)\big)
\quad
\textrm{ for any }
\quad
\alpha\in\cA
\quad,\quad
0\leq i\leq q^{(r)}_\alpha-1,
$$
where $\varphi:I\to I$ is an increasing homeomorphism.  In other words the intervals in the two partitions have the same labels in the same order.
\end{definition}

Considering the inverse and the composition of increasing homeomorphisms of $I$, it is easy to see that the relation above is an equivalence relation.

\begin{proposition}
\label{PropositionPartitionDeterminesRauzyPath}
Let $f$ and $f'$ be two elements of $\cC(\pi,I)$ and fix $r\in\NN$. Then we have the equivalence 
$$
\cP(f,r)\sim\cP(f',r)
\quad
\Leftrightarrow
\quad
\gamma(f,r)=\gamma(f',r).
$$
\end{proposition}

\begin{proof}
We prove the Proposition by induction. For $r=0$ the statement is trivial. Consider $r=1$ and recall \S~\ref{SectionRauzyClasses} and \S~\ref{SectionRauzyInductionGIETs}. We have 
$
\gamma(f,1)=\gamma(f',1)
$ 
iff 
$
\epsilon(f)=\epsilon(f')
$. 
If $\epsilon=t$ then $q^{(1)}_\chi=2$ for $\chi=\alpha_b$ and $q^{(1)}_\chi=1$ for any 
$\chi\not=\alpha_b$, moreover 
\begin{eqnarray*}
&&
I(f,1,\chi,0)=I^t_\chi(f^{(1)})=I^t_\chi(f)
\quad
\textrm{ for }
\quad
\chi\not=\alpha_t,\alpha_b
\\
&&
I(f,1,\alpha_t,0)=I^t_{\alpha_t}(f^{(1)})
=
I^t_{\alpha_t}(f)\setminus f\big(I^t_{\alpha_b}(f)\big)
\\
&&
I(f,1,\alpha_b,1)=f\big(I^t_{\alpha_b}(f)\big),
\end{eqnarray*}
and the same holds for $f'$, therefore $\cP(f,1)\sim\cP(f',1)$. On the other hand, if $\epsilon=b$ then $q^{(1)}_\chi=2$ for $\chi=\alpha_t$ and $q^{(1)}_\chi=1$ for any 
$\chi\not=\alpha_t$, moreover 
\begin{eqnarray*}
&&
I(f,1,\chi,0)=I^t_\chi(f^{(1)})=I^t_\chi(f)
\quad
\textrm{ for }
\quad
\chi\not=\alpha_t,\alpha_b
\\
&&
I(f,1,\alpha_b,0)=I^t_{\alpha_b}(f^{(1)})=
I^t_{\alpha_b}(f)\setminus f^{-1}\big(I^t_{\alpha_t}(f)\big)
\\
&&
I(f,1,\alpha_t,0)=I^t_{\alpha_t}(f^{(1)})=f^{-1}\big(I^t_{\alpha_t}(f)\big)
\\
&&
I(f,1,\alpha_t,1)=f\big(I^t_{\alpha_t}(f^{(1)})\big)=I^t_{\alpha_t}(f),
\end{eqnarray*}
and the same holds for $f'$, therefore again $\cP(f,1)\sim\cP(f',1)$. The argument above proves the implication 
$
\gamma(f,1)=\gamma(f',1)\Rightarrow\cP(f,1)\sim\cP(f',1)
$, 
but since there are only two equivalence classes of dynamical partition of order $r=1$, then the opposite implication also holds. Finally, fix $r\in\NN$ and assume that the equivalence in the statement is proved up to $r$ for any $\pi'\in\cR$, any interval $J$ and any $f,f'\in\cC(\pi',J)$. Fix 
$f,f'\in\cC(\pi,I)$. Observe preliminarily that for a map $f:I\to I$ the intervals in the partition 
$\cP(f,r+1)$ of $I$ are the images under $f$ of the intervals in the partition $\cP(f^{(r)},1)$ of $I^{(r)}$. Then the Proposition follows because the inductive assumption, together with the argument for $r=1$, imply the equivalences
\begin{align*}
\gamma(f,r+1)=\gamma(f',r+1)
&
\Leftrightarrow
\left\{
\begin{array}{l}
\gamma(f,r)=\gamma(f',r)
\\
\epsilon(f^{(r)})=\epsilon\big((f')^{(r)}\big)
\end{array}
\right.
\\
&
\Leftrightarrow
\left\{
\begin{array}{l}
\cP(f,r)\sim\cP(f',r)
\\
\cP(f^{(r)},1)\sim\cP((f')^{(r)},1)
\end{array}
\right.
&
\Leftrightarrow
\cP(f,r+1)\sim\cP(f',r+1).
\end{align*}
\end{proof}

\subsection{Continuity of the Rauzy-Veech map}
\label{SectionContinuityRauzyMap}

Given $f,f'\in\cC\big(\pi,[0,1)\big)$ and $\chi\in\cA$ write for simplicity
$$
\distance(f_\chi,f'_\chi):=\distance\big(\overline{\text{Graph}(f_\chi)},\overline{\text{Graph}(f'_\chi)}\big).
$$

\begin{proposition}
\label{PropositionContinuityRenormalizationHaudorffdistance}
The maps $\cQ_{top}$ and $\cQ_{bot}$ in 
Equation~\eqref{EquationNormalizedBrancesRauzyMap} are continuous with respect to the distance in 
Equation \eqref{EquationDinstanceGIETs}, that is for any $\chi\in\cA$ the following holds.
\begin{enumerate}
\item
The branch $\widetilde{f}_\chi$ of $\widetilde{f}:=\cQ_{top}(f)$ depends continuously on $f$. The same holds for 
$\widetilde{f}:=\cQ_{bot}(f)$. 
\item
We have continuous maps 
$f\mapsto j_\chi\big(\cQ_{top}(f)\big)$ and $f\mapsto j_\chi\big(\cQ_{bot}(f)\big)$. 
\end{enumerate}
\end{proposition}

\begin{proof}
The position of the rightmost singularity depends continuously on $f$ in $\cC\big(\pi,[0,1)\big)$, and therefore also the rescaling map in 
Equation~\eqref{EquationRenormalizationFirstReturn}. Hence it is enough to prove the continuity of  
$\widetilde{\cQ}_{top}$ and $\widetilde{\cQ}_{bot}$. Let $\alpha$ and $\beta$ be the letters with $\pi_t(\alpha)=\pi_b(\beta)=d$. 

\smallskip

Part (I). For $f,f'\in\cC_{top}\big(\pi,[0,1)\big)$ set  
$
\widetilde{f}:=\widetilde{\cQ}_{top}(f)
$ 
and 
$
\widetilde{f}':=\widetilde{\cQ}_{top}(f')
$. 
We fix $f$ and for any $\chi\in\cA$ we prove 
$
\distance(\widetilde{f}_\chi,\widetilde{f}'_\chi)\to0
$ 
when $\distance(f,f')\to 0$.

\emph{Letters $\chi\not=\alpha,\beta$.} Any such $\chi$ neither wins nor looses, thus we have 
$$
\distance(\widetilde{f}_\chi,\widetilde{f}'_\chi)
=
\distance(f_\chi,f'_\chi)
\leq
\distance(f,f')\to0.
$$

\emph{Letter $\chi=\alpha$.}  Parametrize the closure of $I^t_\alpha(\widetilde{f}')$ by
$$
t\mapsto x(f',\alpha,t):=(1-t)u^t_\alpha(f')+tu^b_\beta(f')
\quad\text{with}\quad t\in[0,1].
$$
Then parametrize the closure of the graph of $\widetilde{f}'_\alpha$ by 
$$
t\mapsto \gamma(f',\alpha,t):=\Big(x(f',\alpha,t),f'_\alpha\big(x(f',\alpha,t)\big)\Big).
$$
For any $t$ we have 
$$
\big|x(f',\alpha,t)-x(f,\alpha,t)\big|
\leq
(1-t)\big|u^t_\alpha(f')-u^t_\alpha(f)\big|+t\big|u^b_\beta(f')-u^b_\beta(f)\big|
\leq
\distance(f,f'). 
$$
Therefore Equation \eqref{EquationHausdorffDistanceBoundsDifference} gives
\begin{align*}
&
\big\|\gamma(f',\alpha,t)-\gamma(f,\alpha,t)\big\|
\leq
\big|x(f',\alpha,t)-x(f,\alpha,t)\big|
+
\Big|f'_\alpha\big(x(f',\alpha,t)\big)-f_\alpha\big(x(f,\alpha,t)\big)\Big|
\\
&
\leq
\big|x(f',\alpha,t)-x(f,\alpha,t)\big|
+
\distance(f_\alpha,f'_\alpha)
+
\omega_{f_\alpha}\Big(\distance(f_\alpha,f'_\alpha)+\big|x(f',\alpha,t)-x(f,\alpha,t)\big|\Big)
\\
&
\leq
\distance(f,f')+\distance(f_\alpha,f'_\alpha)
+
\omega_{f_\alpha}\Big(\distance(f_\alpha,f'_\alpha)+\distance(f,f')\Big).
\end{align*}
It follows that
$$
\distance(\widetilde{f}_\alpha,\widetilde{f}'_\alpha)
\leq
\sup_{t\in[0,1]}\big\|\gamma(f',\alpha,t)-\gamma(f,\alpha,t)\big\|
\leq
2\distance(f,f')
+
\omega_{f_\alpha}\Big(2\distance(f,f')\Big)\to0.
$$

\emph{Letter $\chi=\beta$.} Parametrize the closure of 
$
I^t_\beta(\widetilde{f}')=I^t_\beta(f')
$ 
by 
$$
t\mapsto x(f',\beta,t):=(1-t)u^t_\beta(f')+tu^t_{\beta_\ast}(f')
\quad\text{with}\quad t\in[0,1],
$$
where $\pi^t(\beta_\ast)=\pi^t(\beta)+1$. Parametrize the closure of the graph of $\widetilde{f}'_\beta$ by 
$$
t\mapsto\gamma(f',\beta,t):=\Big(x(f',\beta,t),f'_\alpha\circ f'_\beta\big(x(f',\beta,t)\big)\Big).
$$
Arguing as for $\chi=\alpha$ and applying Equation \eqref{EquationHausdorffDistanceBoundsDifference} twice, for any $t$ we get
\begin{align*}
&
\big|
f'_\alpha\circ f'_\beta\big(x(f',\beta,t)\big)-f_\alpha\circ f_\beta\big(x(f,\beta,t)\big)
\big|
\leq
\\
&
\distance(f_\alpha,f'_\alpha)
+
\omega_{f_\alpha}
\Big(
\distance(f_\alpha,f'_\alpha)
+
\big|f'_\beta\big(x(f',\beta,t)\big)-f_\beta\big(x(f,\beta,t)\big)\big|
\Big)
\leq
\\
&
\distance(f_\alpha,f'_\alpha)
+
\omega_{f_\alpha}
\Big(
\distance(f_\alpha,f'_\alpha)
+
\distance(f_\beta,f'_\beta)
+
\omega_{f_\beta}
\big(
\distance(f_\beta,f'_\beta)
+
\big|x(f',\beta,t)-x(f,\beta,t)\big|
\big)
\Big).
\end{align*}
It follows that 
\begin{align*}
&
\distance(\widetilde{f}_\beta,\widetilde{f}'_\beta)
\leq
\sup_{t\in[0,1]}\big\|\gamma(f',\beta,t)-\gamma(f,\beta,t)\big\|
\leq
\\
&
\sup_{t\in[0,1]}
\big|x(f',\beta,t)\big)-x(f,\beta,t)\big|
+
\big|
f'_\alpha\circ f'_\beta\big(x(f',\beta,t)\big)-f_\alpha\circ f_\beta\big(x(f,\beta,t)\big)
\big|
\leq
\\
&
2\distance(f,f')
+
\omega_{f_\alpha}
\Big(
2\distance(f,f')+\omega_{f_\beta}\big(2\distance(f,f')\Big)\to0.
\end{align*}

\smallskip

Part (II). For $f,f'\in\cC_{bot}\big(\pi,[0,1)\big)$ set  
$
\widetilde{f}:=\widetilde{\cQ}_{bot}(f)
$ 
and 
$
\widetilde{f}':=\widetilde{\cQ}_{bot}(f')
$. 
We fix $f$ and for any $\chi\in\cA$ we prove 
$
\distance(\widetilde{f}_\chi,\widetilde{f}'_\chi)\to0
$ 
when $\distance(f,f')\to 0$. 

\emph{Letters $\chi\not=\alpha,\beta$.} As in Part (I), for these letters the statement is trivially true.

\emph{Letter $\chi=\beta$.}  Parametrize the closure of $I^t_\beta(\widetilde{f}')$ by
$$
t\mapsto x(f',\beta,t):=(1-t)u^t_\beta(f')+t(f'_\beta)^{-1}\big(u^t_\alpha(f')\big)
\quad\text{with}\quad t\in[0,1].
$$
Then parametrize the closure of the graph of $\widetilde{f}'_\beta$ by 
$$
t\mapsto \gamma(f',\beta,t):=\Big(x(f',\beta,t),f'_\beta\big(x(f',\beta,t)\big)\Big).
$$
Equation \eqref{EquationHausdorffDistanceBoundsDifference} implies
\begin{align*}
&
\Big|
(f'_\beta)^{-1}\big(u^t_\alpha(f')\big)-f_\beta^{-1}\big(u^t_\alpha(f)\big)
\Big|
\leq
\\
&
\distance((f'_\beta)^{-1},f_\beta^{-1})
+
\omega_{f_\beta^{-1}}\big(\distance((f'_\beta)^{-1},f_\beta^{-1})+|u^t_\alpha(f')-u^t_\alpha(f)|
\big).
\end{align*}
Corollary \ref{CorollaryHausdorffDistanceInverse} implies 
$
\distance((f'_\beta)^{-1},f_\beta^{-1})\leq\distance(f',f)
$, 
thus for any $t$ we have 
\begin{align*}
\big|x(f',\beta,t)-x(f,\beta,t)\big|
&
\leq
(1-t)\big|u^t_\beta(f')-u^t_\beta(f)\big|
+
t\Big|(f'_\beta)^{-1}\big(u^t_\alpha(f')\big)-f_\beta^{-1}\big(u^t_\alpha(f)\big)\Big|
\\
&
\leq
2\distance(f',f)
+
\omega_{f_\beta^{-1}}\big(2\distance(f',f)\big).
\end{align*}
Therefore Equation \eqref{EquationHausdorffDistanceBoundsDifference} gives
\begin{align*}
&
\big\|\gamma(f',\beta,t)-\gamma(f,\beta,t)\big\|
\leq
\big|x(f',\beta,t)-x(f,\beta,t)\big|
+
\Big|f'_\beta\big(x(f',\beta,t)\big)-f_\beta\big(x(f,\beta,t)\big)\Big|
\\
&
\leq
\big|x(f',\beta,t)-x(f,\beta,t)\big|
+
\distance(f_\beta,f'_\beta)
+
\omega_{f_\beta}\Big(\distance(f_\beta,f'_\beta)+\big|x(f',\beta,t)-x(f,\beta,t)\big|\Big)
\\
&
\leq
3\distance(f',f)
+
\omega_{f_\beta^{-1}}\big(2\distance(f',f)\big)
+
\omega_{f_\beta}\Big(3\distance(f',f)
+
\omega_{f_\beta^{-1}}\big(2\distance(f',f)\big)\Big).
\end{align*}
Since 
$
\distance(\widetilde{f}_\beta,\widetilde{f}'_\beta)
\leq
\sup_{t\in[0,1]}\big\|\gamma(f',\beta,t)-\gamma(f,\beta,t)\big\|
$ 
then the argument follows as in the previous cases.

\emph{Letter $\chi=\alpha$.} Parametrize the closure of 
$I^t_\alpha(\widetilde{f}')$ 
by 
$$
t\mapsto x(f',\alpha,t):=(1-t)(f'_\beta)^{-1}\big(u^t_\alpha(f')\big)+tu^t_{\beta_\ast}(f')
\quad\text{with}\quad t\in[0,1],
$$
where $\pi^t(\beta_\ast)=\pi^t(\beta)+1$. Parametrize the closure of the graph of $\widetilde{f}'_\alpha$ by 
$$
t\mapsto\gamma(f',\alpha,t):=\Big(x(f',\alpha,t),f'_\alpha\circ f'_\beta\big(x(f',\alpha,t)\big)\Big).
$$
Arguing as in the previous case for $\chi=\beta$ we have
$$
\big|x(f',\beta,t)-x(f,\beta,t)\big|
\leq
2\distance(f',f)+\omega_{f_\beta^{-1}}\big(2\distance(f',f)\big).
$$
The argument works as in the previous cases because applying Equation \eqref{EquationHausdorffDistanceBoundsDifference} twice we get
\begin{align*}
&
\big|
f'_\alpha\circ f'_\beta\big(x(f',\alpha,t)\big)-f_\alpha\circ f_\beta\big(x(f,\alpha,t)\big)
\big|
\leq
\\
&
\distance(f_\alpha,f'_\alpha)
+
\omega_{f_\alpha}
\Big(
\distance(f_\alpha,f'_\alpha)
+
\big|f'_\beta\big(x(f',\alpha,t)\big)-f_\beta\big(x(f,\alpha,t)\big)\big|
\Big)
\leq
\\
&
\distance(f_\alpha,f'_\alpha)
+
\omega_{f_\alpha}
\Big(
\distance(f_\alpha,f'_\alpha)
+
\distance(f_\beta,f'_\beta)
+
\omega_{f_\beta}
\big(
\distance(f_\beta,f'_\beta)
+
\big|x(f',\alpha,t)-x(f,\alpha,t)\big|
\big)
\Big)
\leq
\\
&
\distance(f,f')
+
\omega_{f_\alpha}
\bigg(
2\distance(f,f')
+
\omega_{f_\beta}
\Big(
3\distance(f',f)+\omega_{f_\beta^{-1}}\big(2\distance(f',f)\big)
\Big)
\bigg).
\end{align*}

Part (III): continuity of the flat intervals. Under $\widetilde{\cQ}_{top}$ we have $j_\chi(\widetilde{f}')=j_\chi(f')$ for any $\chi$, thus the statement holds trivially. Under $\widetilde{\cQ}_{bot}$ we have $j_\chi(\widetilde{f}')=j_\chi(f')$ for any $\chi\not=\alpha$
and 
$
j_\alpha(\widetilde{f}')=(f'_\beta)^{-1}(j_\alpha(f'))
$, 
therefore the statement holds because applying Equation \eqref{EquationHausdorffDistanceBoundsDifference} as in Part (II) we get
$$
\big|
j_\alpha(\widetilde{f}')-j_\alpha(\widetilde{f})
\big|
=
\big|
(f'_\beta)^{-1}(j_\alpha(f'))-f_\beta^{-1}(j_\alpha(f))
\big|
\leq
\distance(f',f)
+
\omega_{f_\beta^{-1}}\big(2\distance(f',f)\big).
$$
The Proposition is proved.
\end{proof}

\section{Proof of Full family Theorem~\ref{TheoremFullFamilyTheorem}}
\label{SectionProofFullFamilyTheorem}

In this section we prove the Full Family Theorem~\ref{TheoremFullFamilyTheorem} assuming Theorem~\ref{TheoremRealizationRauzyPath} below. Let $\pi$ be an admissible combinatorial datum over the alphabet 
$\cA$ and assume that the Rauzy class of $\pi$ contains a cyclic combinatorial datum $\pi^{(\ast)}$. 
Let $F:\Delta^d\to\cC\big(\pi,[0,1)\big)$ be the map in Definition~\ref{DefinitionFullFamily} and consider the family $\cF:=(f_\lambda)_{\lambda\in\Delta^d}$, where $f_\lambda=F(\lambda)$.

\begin{theorem}
\label{TheoremRealizationRauzyPath}
For any finite Rauzy path $\gamma:\pi\to\pi^{(\ast)}$ of length $r$ there exists 
$
\lambda=\lambda(\gamma)\in\Delta^d
$ 
such that 
$$
\gamma(f_\lambda,r)=\gamma.
$$
\end{theorem}

Given any two finite paths $\gamma:\pi\to\pi'$ and $\eta:\pi'\to\pi''$ in $\cR$, where the starting point $\pi'\in\cR$ of $\eta$ coincides with the ending point of $\gamma$, consider the concatenation 
$
\gamma\ast\eta:\pi\to\pi''
$.

\begin{remark}
\label{RemarkRealizationRauzyPath}
The requirement that $\gamma$ ends at $\pi^{(\ast)}$ in Theorem~\ref{TheoremRealizationRauzyPath} can be easily removed a posteriori. Indeed if $\gamma:\pi\to\pi'$ is any finite Rauzy path with length $r$, there exists a finite path 
$
\eta:\pi'\to\pi^{(\ast)}
$ 
starting at the ending point $\pi'$ of $\gamma$, ending at $\pi^{(\ast)}$ and with length $|\eta|$ bounded by the cardinality of $\cR$. For the concatenation $\gamma\ast\eta$ Theorem~\ref{TheoremRealizationRauzyPath} gives $\lambda\in\Delta^d$ with 
$$
\gamma\big(f_\lambda,r+|\eta|\big)=\gamma\ast\eta,
$$ 
so that by truncation we get $\gamma(f_\lambda,r)=\gamma$.
\end{remark}

\subsection{Preliminary Lemmas}

Given a Rauzy path $\gamma:\pi\to\pi'$ define 
$$
\Delta_\gamma:=\{\lambda\in\Delta^d:\gamma(f_\lambda,r)=\gamma\}.
$$
The continuity of the map $\lambda\to f_\lambda$ and Lemma~\ref{LemmaContinuityRauzyPath} imply that $\Delta_\gamma$ is an open subset of $\Delta^d$. The less trivial statement is the non-emptyness of $\Delta_\gamma$, which is the content of Theorem~\ref{TheoremRealizationRauzyPath}.

\begin{lemma}
\label{LemmaContinuityRauzyPath}
Fix any $r\in\NN$ and consider $f\in\cC\big(\pi,[0,1)\big)$ which admits $r$ steps of the Rauzy induction. Then there exists an open set  
$
\cU=\cU(f,r)\subset\cC\big(\pi,[0,1)\big)
$ 
such that any $f'\in\cU$ has the same Rauzy renormalization path as $f$ up to the first $r$ steps, that is 
$$
\gamma(f',r)=\gamma(f,r).
$$
\end{lemma}

\begin{proof}
We prove the Lemma by induction on the number of steps $r$. For $r=1$ the statement is true because the type $\varepsilon(f)=b/t$ depends on the relative positions of the singularities $u^{t/b}_\alpha(f)$ and on the endpoints $j_\alpha(f)$ of the flat interval, and these quantities depend continuously on $f$. Now fix $r$. The inductive step $(r)\Rightarrow(r+1)$ follows from a similar argument because $f^{(r)}=\cQ^r(f)$ and its flat intervals depend continuously on $f$ by Proposition \ref{PropositionContinuityRenormalizationHaudorffdistance}.
\end{proof}

A finite Rauzy path $\nu$ in $\cR$ is \emph{positive} if the matrix $B_\nu\in\slgroup(d,\ZZ)$ defined in \S~\ref{SectionRauzyClasses} has all its entries strictly positive.

\begin{lemma}
\label{LemmaPositiveRauzyPaths}
If $\nu$ is positive then we have $\overline{\Delta_\nu}\subset\Delta^d$.
\end{lemma}

\begin{proof}
Recall the notation in \S~\ref{SectionDynamicalPartition}. Set $q^{(k)}:=B_\nu\vec{1}$, where $k$ is the length of $\nu$. For any $\lambda\in\Delta_\nu$ we have a partition $\cP(f_\lambda,k)$, where the atoms in 
$\cP(f_\lambda,k)$ are the intervals 
$$
I(f_\lambda,k,\alpha,i):=
f_\lambda^i\big(I^t_\alpha(f_\lambda^{(k)})\big)
\quad
\textrm{ for }
\quad
\alpha\in\cA
\quad;\quad
0\leq i\leq q^{(k)}_\alpha-1.
$$
All these partitions are equivalent by Proposition~\ref{PropositionPartitionDeterminesRauzyPath}. 
Consider 
$
\lambda(\infty)\in\overline{\Delta_\nu}
$ 
and $\lambda(n)\in\Delta_\nu$ 
with $\lambda(n)\to\lambda(\infty)$. If $\lambda(\infty)\in\partial\Delta^d$ then $f_{\lambda(\infty)}$ is a degenerate GIET, because $(f_\lambda)$ is a proper family. Hence, in order to prove that $\Delta_\nu$ is relatively compact in $\Delta^d$, we prove that
$$
f_{\lambda(\infty)}\in\cC\big(\pi,[0,1)\big).
$$ 
Set for simplicity $f:=f_{\lambda(\infty)}$ and $f_n:=f_{\lambda(n)}$. The intervals $I^{(k)}_n$ where the first return maps $f_n^{(k)}=\widetilde{\cQ}^k(f_n)$ acts satisfy
$
\inf_{n\in\NN}\big|I^{(k)}_n\big|>0
$. 
Indeed the maps $f_n$ have an uniform modulus of continuity by Equation \eqref{EquationUniformContinuityModulus}, thus if for $\alpha\in\cA$ all the intervals $I^t_\alpha(f_n^{(k)})$ are arbitrarily small, then their iterates $f_n^i\big(I^t_\alpha(f_n^{(k)})\big)$ for $\alpha\in\cA$ and $0\leq i \leq q^{(k)}_\alpha-1$ are also small and cannot form a partition of $[0,1)$. Therefore there exists $\epsilon>0$ such that for any $n$ there exist 
$\alpha(n)\in\cA$ depending on $n$ with
$
\big|I^b_{\alpha(n)}(f_n^{(k)})\big|>\epsilon
$. 
Setting $Q:=q^{(k)}_{\alpha(n)}$ we have 
$$
I^b_{\alpha(n)}(f_n^{(k)})=f_n^{Q-i}\big(I(f_n,k,\alpha(n),i)\big).
$$
The maps $f_n$ have an uniform modulus of continuity, thus if $I(f_n,k,\alpha(n),i)$ is arbitrarily small for some $0\leq i\leq q^{(k)}_{\alpha(n)}-1$ then $I^b_{\alpha(n)}(f_n^{(k)})$ is arbitrarily small too, which is absurd. Hence, in terms of the sequence of letters $\alpha(n)$ defined above, we have 
$$
\inf_{n\in\NN}|I(f_n,k,\alpha(n),i)|>0
\quad
\textrm{ for any}
\quad
0\leq i\leq q^{(k)}_{\alpha(n)}-1.
$$
Since $\nu$ is positive, then Lemma~\ref{LemmaIntervalsPartitionAndMatrix} implies that for any $\beta,\alpha\in\cA$ and any $n$ there exists an integer $j$ with $0\leq j\leq q^{(k)}_\alpha-1$ and   
$
f_n^j\big(I^t_\alpha(f_n^{(k)})\big)\subset I^t_\beta(f_n)
$.
Therefore for any $\beta$ we have
$
\inf_{n\in\NN}|I^t_\beta(f_n)|>0
$ 
and thus $|I^t_\beta(f)|>0$ by continuity of $\lambda\to f_\lambda$. In other words  
$
f_{\lambda(\infty)}\in\cC\big(\pi,[0,1)\big)
$. 
The Lemma is proved.
\end{proof}

\subsection{Proof of Theorem~\ref{TheoremFullFamilyTheorem}}
\label{SectionProofTheoremFullFamily} 

Let $T$ be a minimal IET with combinatorial datum $\pi$. Recalling Definition~\ref{DefinitionRauzyPathOfGIET}, 
let 
$
\gamma(T,\infty)
$ 
be the infinite Rauzy path of $T$. Then fix $r\in\NN$ and let 
$
\gamma(T,r)=\gamma_1\ast\dots\ast\gamma_r
$ 
be the concatenation of the first $r$ arrows of 
$\gamma(T,\infty)$. According to Theorem~\ref{TheoremRealizationRauzyPath} and Remark~\ref{RemarkRealizationRauzyPath} there exists  
$
\lambda(r)\in\Delta^d
$ 
such that 
\begin{equation}
\label{Equation(1)SectionProofTheoremFullFamily}
\gamma(f_{\lambda(r)},r)=\gamma(T,r). 
\end{equation}
Moreover, according to Lemma~\ref{LemmaContinuityRauzyPath}, there is a non-empty open set 
$$
\Delta_{\gamma(T,r)}:=\{\lambda\in\Delta^d,\gamma(f_\lambda,r)=\gamma(T,r)\}.
$$
Increasing $r$ we get a sequence of nested non-empty open sets 
$$
\dots\subset\Delta_{\gamma(T,r+1)}\subset\Delta_{\gamma(T,r)}\subset\dots\subset\Delta^d.
$$
The path $\gamma(T,\infty)$ is infinite-complete, that is any letter $\alpha\in\cA$ is the winner of infinitely many arrows of $\gamma(T,\infty)$. According to Proposition 7.12 in \cite{Yoccoz}, if $r_0$ is big enough, the path 
$\gamma(T,r_0)$ is positive, so that $\Delta_{\gamma(T,r_0)}$ is relatively compact in $\Delta^d$, by Lemma \ref{LemmaPositiveRauzyPaths}. As a consequence of this last property
$$
\Delta_{\gamma(T,\infty)}:=
\bigcap_{r\in\NN}\overline{\Delta_{\gamma(T,r)}}
$$
is a compact subset contained in the interior of $\Delta^d$. Modulo taking a subsequence of the parameters $\lambda(r)\in\Delta^d$, $r\in\NN$ in Equation~\eqref{Equation(1)SectionProofTheoremFullFamily}, there exists 
$
\lambda\in\Delta_{\gamma(T,\infty)}
$ 
such that $\lambda(r)\to\lambda$ as $r\to\infty$, that is 
$
f_{\lambda(r)}\to f_\lambda
$. 
In particular the limit $f_\lambda$ is a non-degenerate element in 
$
\cC\big(\pi,[0,1)\big)
$. 
Moreover $f_\lambda$ is infinitely renormalizable, because 
$
\lambda\in\Delta_{\gamma(T,\infty)}
$. 
By definition we have $\gamma(f_\lambda,r)=\gamma(T,r)$ for any $r$, therefore 
$
\gamma(f_\lambda,\infty)=\gamma(T,\infty)
$. 
Theorem~\ref{TheoremFullFamilyTheorem} is proved. $\qed$

\section{The Thurston map}
\label{SectionThurstonMap}

In this section we introduce the \emph{space of configurations} (\S~\ref{SectionConfigurations}) and show that the \emph{Thurston map} $\cT_\gamma$ acts uniformly continuously over it (\S~\ref{SectionDefinitionThurstonMap}). 
Recall the notation of \S~\ref{SectionProofFullFamilyTheorem}, then fix a finite Rauzy path 
$\gamma:\pi\to\pi^{(\ast)}$, where $\pi^{(\ast)}$ is a cyclic combinatorial datum.

\subsection{Configurations}
\label{SectionConfigurations}

Let $B_\gamma\in\slgroup(d,\ZZ)$ be the matrix associated to $\gamma$ as in \S~\ref{SectionRauzyClasses}. Let 
$
q^{(r)}:=B_\gamma\vec{1}
$ 
be the integer vector of \emph{return times} defined in \S~\ref{SectionDynamicalPartition}. 
Consider the normalized length datum 
$
\lambda^{(\gamma)}\in\QQ_+^\cA\cap\Delta^d
$ 
defined by 
$$
\lambda^{(\gamma)}:=(1/N)^t\!B_\gamma\vec{1}
,\quad\text{where}\quad
N=N(\gamma):=\sum_{\chi\in\cA}q^{(r)}_\chi.
$$ 
Observe that 
$
\sum_{\chi\in\cA}\lambda^{(\gamma)}_\chi=1
$ 
and let 
$
T_\gamma:[0,1)\to[0,1)
$ 
be the IET with combinatorial-length data $(\pi,\lambda^{(\gamma)})$. 
Let
$
T_\gamma^{(r)}:=\cQ^r(T_\gamma)
$ 
be the image of $T_\gamma$ under the $r$-th iteration of the Rauzy map in \S~\ref{SectionRauzyInductionGIETs}. Recalling \S~\ref{SectionDynamicalPartition}, for $\alpha\in\cA$ consider the intervals $I^t_\alpha(T_\gamma^{(r)})$ and their images under $T_\gamma$, that is the intervals $I(T_\gamma,r,\alpha,i)$ with $0\leq i\leq q^{(r)}_\alpha-1$ in Equation~\eqref{EquationDynamicalPartitionHigherOrder}, which are the atoms of the partition $\cP(T_\gamma,r)$. Observe that
$$
T_\gamma^{(r)}=T(\pi^{(\ast)},N^{-1}\cdot \vec{1}),
$$
which is a periodic IET acting on the intervals $I^t_\alpha(T_\gamma^{(r)})$, $\alpha\in\cA$ as the cyclic permutation 
$
\sigma(\pi^{(\ast)})
$. 
Recall also that for $\alpha\in\cA$ we call  
$
u^t_\alpha(T_\gamma):=\min I^t_\alpha(T_\gamma)
$ 
and 
$
u^t_\alpha(T_\gamma^{(r)}):=\min I^t_\alpha(T_\gamma^{(r)})
$ 
the critical points respectively of $T_\gamma$ and of $T^{(r)}$. Let $h(\alpha,r)$ be the unique integer with 
$
0\leq h(\alpha,r)\leq q^{(r)}_\alpha-1
$ 
such that 
\begin{equation}
\label{EquationHalfReturnTimes}
T_\gamma^{h(\alpha,r)}\big(u^t_\alpha(T_\gamma^{(r)})\big)=
u^t_\alpha(T_\gamma).
\end{equation}
In particular $h(\alpha_0,r)=0$ for the letter $\alpha_0$ with $\pi_t(\alpha_0)=1$. 
Since $T_\gamma^{(r)}$ is a first return of $T_\gamma$, then $T_\gamma$ acts as a cyclic permutation on the intervals of the partition $\cP(T_\gamma,r)$ by 
$$
T_\gamma\big(I(T_\gamma,r,\alpha,i)\big)=
I(T_\gamma,r,\alpha,i+1),
$$
modulo identifications on the labels $(\alpha,i)\in\cA\times\ZZ$ given by (see also Figure~\ref{FigureExampleConfiguration}):

\begin{equation}
\label{EquationIdentificationLabelsCriticalPoints}
(\alpha,i)\sim(\beta,j)
\Leftrightarrow
\left\{
\begin{array}{l}
\pi^{(\ast)}_t(\beta)=\pi^{(\ast)}_b(\alpha)
\\
i-j=q^{(r)}_\alpha +h(\beta,r)-h(\alpha,r).
\end{array}
\right.
\end{equation}
Let $\cI_\gamma$ be the quotient of the set $\cA\times\ZZ$ under the equivalence relation $\sim$ in  Equation~\eqref{EquationIdentificationLabelsCriticalPoints}, whose equivalence classes are denoted by $[\alpha,i]$. 
The map 
$
[\alpha,i]\mapsto[\alpha,i+1]
$ 
on equivalence classes corresponds to the cyclic action of $T_\gamma$ on intervals of the partition $\cP(T_\gamma,r)$. With these identifications, the critical points of $T_\gamma$ correspond to classes $[\alpha,0]$ and critical values of $T_\gamma$ correspond to classes $[\alpha,1]$. An example is given in Figure \ref{FigureExampleConfiguration}, where we consider the alphabet $\cA:=\{A,B,C,D\}$ and the admissible combinatorial data 
$$
\pi=
\begin{pmatrix}
A & B & C & D \\
D & C & B & A 
\end{pmatrix}
\quad
\textrm{ and }
\quad
\pi^{(\ast)}=
\begin{pmatrix}
A & B & D & C \\
D & A & C & B 
\end{pmatrix},
$$
which are connected by the Rauzy path $\gamma:\pi\to\pi^{(\ast)}$ of length five whose winner are in the order the letters $A,A,A,D,B$. In particular $\pi^{(\ast)}$ is cyclic with 
$
\sigma(\pi^{(\ast)})=(1,2,4,3)
$, 
which has order $4$. In this case, the matrix $B_\gamma$ defined in \S~\ref{SectionRauzyClasses} is 
$$
B_{\gamma}=
\begin{pmatrix}
2 & 0 & 0 & 1 \\
1 & 1 & 0 & 0 \\
1 & 0 & 1 & 0 \\
2 & 1 & 0 & 1
\end{pmatrix}
\quad
\textrm{ with }
\quad
^tB^{-1}_{\gamma}
=
\begin{pmatrix}
1 &-1 &-1 &-1 \\
1 & 0 &-1 &-2 \\
0 & 0 & 1 & 0 \\
-1& 1 & 1 & 2
\end{pmatrix}.
$$
We have $q^{(r)}=(3,2,2,4)$ and $N(\gamma)=11$, so that 
$
\lambda^{(\gamma)}=(\frac{6}{11},\frac{2}{11},\frac{1}{11},\frac{2}{11})
$. 
Moreover the vector $h\in\NN^\cA$ whose entries $h_\alpha:=h(\alpha,r)$ for $\alpha\in\cA$ are defined by Equation~\eqref{EquationIdentificationLabelsCriticalPoints} is $h=(0,1,1,3)$.

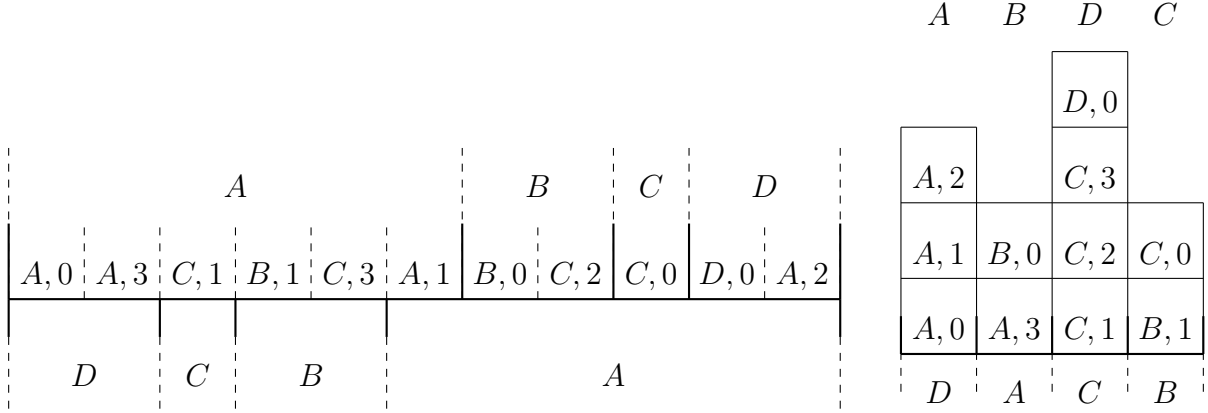
\begin{figure}
\begin{center}
{\begin{tikzpicture}[scale=0.1]



\draw[-,thick] (0,0) -- (110,0) {};
\draw[-,thick] (0,0) -- (0,10) {};
\draw[-,thin,dashed] (0,0) -- (0,20) {};
\node at (30,15) {$A$};
\node at (5,3) {$A,0$};
\draw[-,thin,dashed] (10,0) -- (10,10) {};
\node at (15,3) {$A,3$};
\draw[-,thin,dashed] (20,0) -- (20,10) {};
\node at (25,3) {$C,1$};
\draw[-,thin,dashed] (30,0) -- (30,10) {};
\node at (35,3) {$B,1$};
\draw[-,thin,dashed] (40,0) -- (40,10) {};
\node at (45,3) {$C,3$};
\draw[-,thin,dashed] (50,0) -- (50,10) {};
\node at (55,3) {$A,1$};

\draw[-,thick] (60,0) -- (60,10) {};
\draw[-,thin,dashed] (60,0) -- (60,20) {};
\node at (70,15) {$B$};
\node at (65,3) {$B,0$};
\draw[-,thin,dashed] (70,0) -- (70,10) {};
\node at (75,3) {$C,2$};

\draw[-,thick] (80,0) -- (80,10) {};
\draw[-,thin,dashed] (80,0) -- (80,20) {};
\node at (85,15) {$C$};
\node at (85,3) {$C,0$};

\draw[-,thick] (90,0) -- (90,10) {};
\draw[-,thin,dashed] (90,0) -- (90,20) {};
\node at (100,15) {$D$};
\node at (95,3) {$D,0$};
\draw[-,thin,dashed] (100,0) -- (100,10) {};
\node at (105,3) {$A,2$};
\draw[-,thick] (110,0) -- (110,10) {};
\draw[-,thin,dashed] (110,0) -- (110,20) {};


\draw[-,thick] (0,0) -- (0,-5) {};
\draw[-,thin,dashed] (0,-5) -- (0,-15) {};
\node at (10,-10) {$D$};
\draw[-,thick] (20,0) -- (20,-5) {};
\draw[-,thin,dashed] (20,-5) -- (20,-15) {};
\node at (25,-10) {$C$};
\draw[-,thick] (30,0) -- (30,-5) {};
\draw[-,thin,dashed] (30,-5) -- (30,-15) {};
\node at (40,-10) {$B$};
\draw[-,thick] (50,0) -- (50,-5) {};
\draw[-,thin,dashed] (50,-5) -- (50,-15) {};
\node at (80,-10) {$A$};
\draw[-,thick] (110,0) -- (110,-5) {};
\draw[-,thin,dashed] (110,-5) -- (110,-15) {};

\end{tikzpicture}}
\hspace{0.5 cm}
{\begin{tikzpicture}[scale=0.1]



\draw[-,thick] (0,0) -- (40,0) {};
\draw[-,thick] (0,0) -- (0,5) {};
\node at (5,45) {$A$};
\draw[-,thick] (10,0) -- (10,5) {};
\node at (15,45) {$B$};
\draw[-,thick] (20,0) -- (20,5) {};
\node at (25,45) {$D$};
\draw[-,thick] (30,0) -- (30,5) {};
\node at (35,45) {$C$};
\draw[-,thick] (40,0) -- (40,5) {};

\draw[-,thin,dashed] (0,-5) -- (0,0) {};
\node at (5,-5) {$D$};
\draw[-,thin,dashed] (10,-5) -- (10,0) {};
\node at (15,-5) {$A$};
\draw[-,thin,dashed] (20,-5) -- (20,0) {};
\node at (25,-5) {$C$};
\draw[-,thin,dashed] (30,-5) -- (30,0) {};
\node at (35,-5) {$B$};
\draw[-,thin,dashed] (40,-5) -- (40,0) {};

\draw[-,thin] (0,0) -- (0,30) {};
\node at (5,3) {$A,0$};
\draw[-,thin] (0,10) -- (10,10) {};
\node at (5,13) {$A,1$};
\draw[-,thin] (0,20) -- (10,20) {};
\node at (5,23) {$A,2$};
\draw[-,thin] (0,30) -- (10,30) {};
\draw[-,thin] (10,0) -- (10,30) {};

\draw[-,thin] (10,10) -- (20,10) {};
\node at (15,3) {$A,3$};
\draw[-,thin] (10,20) -- (20,20) {};
\node at (15,13) {$B,0$};

\draw[-,thin] (20,0) -- (20,40) {};
\node at (25,3) {$C,1$};
\draw[-,thin] (20,10) -- (30,10) {};
\node at (25,13) {$C,2$};
\draw[-,thin] (20,20) -- (30,20) {};
\node at (25,23) {$C,3$};
\draw[-,thin] (20,30) -- (30,30) {};
\node at (25,33) {$D,0$};
\draw[-,thin] (20,40) -- (30,40) {};
\draw[-,thin] (30,0) -- (30,40) {};

\draw[-,thin] (30,10) -- (40,10) {};
\node at (35,3) {$B,1$};
\draw[-,thin] (30,20) -- (40,20) {};
\node at (35,13) {$C,0$};
\draw[-,thin] (40,0) -- (40,20) {};

\end{tikzpicture}}
\end{center}
\caption{The dynamical partition $\cP(T_\gamma,5)$, where 
$
T_\gamma=T(\pi,\lambda_\gamma)
$ 
in terms of $\pi$, $\lambda^{(\gamma)}$ and $\gamma:\pi\to\pi^{(\ast)}$ defined in \S~\ref{SectionConfigurations}. The left endpoints of the atoms of $\cP(T_\gamma,5)$ are the points in the standard $\gamma$-configuration $\cV^{(\gamma)}$. They form an unique orbit of length $N=11$ under the action of $T_\gamma$. The points $\underline{v([\alpha,0])}$ for $\alpha\in\cA$ are the critical points of $T_\gamma$.}
\label{FigureExampleConfiguration}
\end{figure}

\smallskip

For any $[\alpha,i]\in\cI_\gamma$ define the rational point 
$$
\underline{v([\alpha,i])}:=
T_\gamma^i\big(u^t_\alpha(T_\gamma)\big)\in [0,1)\cap\QQ,
$$
so that in particular $\underline{v([\alpha,0])}$ and $\underline{v([\alpha,1])}$ for $\alpha\in\cA$ correspond respectively to the critical point and the critical value of $T_\gamma$. The \emph{reference $\gamma$-configuration} $\cV^{(\gamma)}$ is the set of $N$ points in $[0,1)\cap\QQ$ defined by 
\begin{equation}
\label{EquationReferenceConfiguration}
\cV^{(\gamma)}:=
\{\underline{v([\alpha,i])};[\alpha,i]\in\cI_\gamma\}.
\end{equation}
Modulo the equivalence relation on $\cA\times\ZZ$ in Equation~\eqref{EquationIdentificationLabelsCriticalPoints}, 
for any $[\alpha,i]$ we have 
$$
T_\gamma\big(\underline{v([\alpha,i])}\big)=
\underline{v([\alpha,i+1])}.
$$

A preconfiguration is a set $\cV\subset[0,1)$ of $N$ distinct points $v([\alpha,i])$ labeled by classes  
$
[\alpha,i]\in\cI_\gamma
$, 
where 
$
v([\beta_0,1])=v([\alpha_0,0])=0
$ 
for the letter $\alpha_0,\beta_0$ with $\pi_b(\beta_0)=1$ and $\pi_t(\alpha_0)=1$ and moreover 
$$
v([\alpha,i])<v([\beta,j])
\quad
\Leftrightarrow
\quad
\underline{v([\alpha,i])}<\underline{v([\beta,j])}
$$
for any other $[\alpha,i]$ and $[\beta,j]$ in $\cI_\gamma$. In particular the standard configuration $\cV^{(\gamma)}$ in Equation~\eqref{EquationReferenceConfiguration} is a preconfiguration. In general the points in a preconfiguration $\cV$ have the same geometrical order as the points in $\cV^{(\gamma)}$. Notice that the space of all preconfigurations is the open standard simplex $\Delta^N$ in $\RR^N$. 

Given a preconfiguration 
$\cV=\{v([\alpha,i]):[\alpha,i]\in\cI_\gamma\}$ let $V(\cV):=\{v([\alpha,1]):\alpha\in\cA\}$ be the ordered subset of points $v([\alpha,i])$ with $i=1$. This gives a projection map 
$$
V:\Delta^N\to\nabla_\pi,
\quad V(\cV):=\big(v([\alpha,1]):\alpha\in\cA\big).
$$
Let $(u^b\circ F)^{-1}:\nabla_\pi\to\Delta^d$ be the inverse of the homeomorphism in Point (1) of 
Definition~\ref{DefinitionFullFamily}. By composition, we obtain a map 
\begin{equation}
\label{EquationFromConfigurationToGIET(SingularValues)}
\Delta^N \xrightarrow{V(\cdot)} 
\nabla_\pi \xrightarrow{(u^b\circ F)^{-1}} 
\Delta^d \xrightarrow{F} \cC\big(\pi,[0,1)\big),
\qquad
\cV\mapsto f_\cV.
\end{equation}
The map in Equation \eqref{EquationFromConfigurationToGIET(SingularValues)} is uniformly continuous, indeed each of its three factors is uniformly continuous because it admits a continuous extension to the closure of its domain. Similarly, considering the projection $U:\Delta^N\to\nabla^t_\pi$, where 
$U(\cV):=\big(v([\alpha,0]):\alpha\in\cA\big)$ is the ordered set of points $v([\alpha,i])\in\cV$ with $i=0$, we obtain an uniformly continuous map
\begin{equation}
\label{EquationFromConfigurationToGIET(SingularPoints)}
\Delta^N \xrightarrow{U(\cdot)} 
\nabla_\pi^t \xrightarrow{(u^t\circ F)^{-1}} 
\Delta^d \xrightarrow{F} \cC\big(\pi,[0,1)\big),
\qquad
\cV\mapsto g_\cV,
\end{equation}
where $(u^t\circ F)^{-1}:\nabla_\pi^t\to\Delta^d$ is the inverse of the homeomorphism in Point (2) of 
Definition~\ref{DefinitionFullFamily}. 

\begin{lemma}
\label{LemmaFromConfigurationToMap}
Fix any preconfiguration 
$
\cV=\{v([\alpha,i]):[\alpha,i]\in\cI_\gamma\}\in\Delta^N
$.
\begin{enumerate}
\item
The GIET $f_\cV$ is the unique element of the family $\cF$ such that for any $\alpha\in\cA$ we have
$$
u^b_\alpha(f_\cV)=v([\alpha,1]).
$$ 
\item
The GIET $g_\cV$ is the unique element of the family $\cF$ such that for any $\alpha\in\cA$ we have
$$
u^t_\alpha(g_\cV)=v([\alpha,0]).
$$ 
\end{enumerate}
\end{lemma}

\begin{proof}
The unicity properties of $f_\cV$ and $g_\cV$ follow respectively from Points (1) and (2) in Definition~\ref{DefinitionFullFamily}. The Lemma is proved.
\end{proof}

Let $[\alpha_m,i_m]\in\cI_\gamma$ be the label such that 
$
\underline{v([\alpha_m,i_m])}
$ 
is the rightmost element in the standard configuration $\cV^{(\gamma)}$, that is 
$\underline{v([\alpha_m,i_m])}=1-N^{-1}$, 
or in other words
\begin{equation}
\label{EquationRightmostPointConfiguration}
\underline{v([\alpha_m,i_m])}:=
\max\big\{v([\alpha_m,i_m]):[\alpha,i]\in\cI_\gamma\big\}.
\end{equation}
The set $\cI_\gamma$ admits a \emph{geometrical order}, where for any $[\alpha,i]\in\cI_\gamma$ we define 
\begin{equation}
\label{EquationGeometricOrder}
[\alpha,i]^\ast:=
\left\{
\begin{array}{ccc}
1&\quad\text{ if }&\quad[\alpha,i]=[\alpha_m,i_m]
\\
\min\{\underline{v([\chi,l])}\textrm{ ; }\underline{v([\chi,l])}>\underline{v([\alpha,i])}\} 
&\quad\text{ if }&\quad[\alpha,i]\not=[\alpha_m,i_m].
\end{array}
\right.
\end{equation}
In other words 
$
\underline{v([\alpha,i]^\ast)}=\underline{v([\alpha,i])}+1/N
$. 
Recall Equation \eqref{EquationRightEndpointFlatInterval}.

\begin{definition}
\label{DefinitionGammaConfiguration}
A configuration is a preconfiguration $\cV$ such that for any $\alpha$ we have
$$
v([\alpha,0])^\ast>j_\alpha(g_\cV)
$$
strictly, that is the flat interval $L_\alpha(g_\cV)$ of $g_\cV$ does not contain any point of $\cV$ other than $u^t_\alpha(g_\cV)=v([\alpha,0])$. The \emph{configuration space} $\cO_\gamma\subset\Delta^N$ is the set of all configurations $\cV$.
\end{definition}

The map $\cV\to v([\alpha,0])^\ast-j_\alpha(g_\cV)$ is continuous, thus the configuration space $\cO_\gamma$ is an open subset of the simplex $\Delta^N$.

\subsection{Definition of the Thurston map} 
\label{SectionDefinitionThurstonMap}

For $\cV\in\cO_\gamma$ let $f_\cV$ be the map in Equation \eqref{EquationFromConfigurationToGIET(SingularValues)}. Let
$
\cV'=\{v'([\alpha,i]);[\alpha,i]\in\cI_\gamma\}\subset[0,1)
$ 
be the set of points defined by 
\begin{equation}
\label{EquationDefinitionThurstonMap}
v'([\alpha,i]):=
\left\{
\begin{array}{cccc}
u^t_\alpha(f_\cV) & \text{for} & i=0
\\
f_\cV^{-1}\big(v([\alpha,i+1])\big) & \text{for} & i\neq0.
\end{array}
\right.
\end{equation}

Observe that the maps in Lemmma \ref{LemmaFromConfigurationToMap} satisfy
\begin{equation}
\label{EquationRelationFunctionsFandG}
g_{\cV'}=f_{\cV}.
\end{equation}

\begin{proposition}
\label{PropositionDefinitionThurstonMap}
For any configuration $\cV$ the set $\cV'$ in Equation~\eqref{EquationDefinitionThurstonMap} is a configuration. In other words we have a well defined map 
$$
\cT_\gamma:\cO_\gamma\to\cO_\gamma
\quad
\textrm{ ; }
\quad
\cV\mapsto\cV'
$$
called the \emph{Thurston map}.
\end{proposition}

\begin{proof}
Part (1): we first prove that $\cV'$ is a preconfiguration. We have $v'([\alpha,0])=u^t_\alpha(f_\cV)$ for any 
$\alpha\in\cA$. Since $f_\cV$ and $T_\gamma$ have the same combinatorial datum $\pi$, then $v'([\alpha_0,0])=0$ for the letter with $\pi_t(\alpha_0)=1$. Moreover for all other critical points we have
$$
v'([\alpha,0])< v'([\beta,0])  
\quad
\Leftrightarrow
\quad
\underline{v([\alpha,0])}< \underline{v([\beta,0])}.
$$
Now consider any pair of different labels $[\alpha,i],[\beta,j]\in\cI_\gamma$ and assume without loss of generality that 
$
\underline{v([\alpha,i])}<\underline{v([\beta,j])}
$. 
The Proposition follows proving that $v'([\alpha,i])<v'([\beta,j])$. Recall the notation in \S~\ref{SectionSingularValuesMap} and that for any $\alpha\in\cA$ we have 
$\min I^t_\alpha(T_\gamma)=\underline{v([\alpha,0])}$ and 
$\min I^b_\alpha(T_\gamma)=\underline{v([\alpha,1])}$. We consider separately two cases. Assume first that 
$
\underline{v([\alpha,i])}
$ 
and 
$
\underline{v([\beta,j])}
$ 
are in the same continuity interval of 
$
T_\gamma
$, 
that is there is a letter $\chi\in\cA$ such that  
$$
\min I^t_\chi(T_\gamma)
=
\underline{v([\chi,0])}
\leq
\underline{v([\alpha,i])}<\underline{v([\beta,j])}
<
\sup I^t_\chi(T_\gamma).
$$ 
In this case we have
\begin{align*}
\min I^b_\chi(T_\gamma)=
\underline{v([\chi,1])}
\leq
&
T_\gamma\big(\underline{v([\alpha,i])}\big)=
\underline{v([\alpha,i+1])}<
\\
&
\underline{v([\beta,j+1])}=
T_\gamma\big(\underline{v([\beta,j])}\big)
<
\sup I^b_\chi(T_\gamma).
\end{align*}
The geometrical order of the points in $\cV$ is the same as in $\cV^{(\gamma)}$, hence we have 
$$
\min I^b_\chi(f_\cV)
=
v([\chi,1])
\leq
v([\alpha,i+1])<v([\beta,j+1])
<
\sup I^b_\chi(f_\cV).
$$ 
It follows that $v([\alpha,i+1])$ and $v([\beta,j+1])$ belong to the same (right-open) continuity interval of $f^{-1}_\cV$, and thus 
$$
v'([\alpha,i])=
f_\cV^{-1}\big(v([\alpha,i+1])\big) 
<
f_\cV^{-1}\big(v([\beta,j+1])\big)
=
v'([\beta,j]).
$$

Assume now that 
$
\underline{v([\alpha,i])}
$ 
and 
$
\underline{v([\beta,j])}
$ 
belong to two different continuity intervals of $T_\gamma$, labelled respectively by letters $\chi_1$ and $\chi_2$ with $\pi_t(\chi_2)\geq\pi_t(\chi_1)+1$, that is  
\begin{align*}
\min I^t_{\chi_1}(T_\gamma)
\leq
&
\underline{v([\alpha,i])}<
\sup I^t_{\chi_1}(T_\gamma)
\leq
\min I^t_{\chi_2}(T_\gamma)
\leq
\underline{v([\beta,j])}<
\sup I^t_{\chi_2}(T_\gamma,0).
\end{align*}
In this case we have
\begin{align*}
&
\min I^b_{\chi_1}(T_\gamma)=
\underline{v([\chi_1,1])}\leq
\underline{v([\alpha,i+1])}<
\sup I^b_{\chi_1}(T_\gamma)
\\
&
\min I^b_{\chi_2}(T_\gamma)=
\underline{v([\chi_2,1])}\leq
\underline{v([\beta,j+1])}<
\sup I^b_{\chi_2}(T_\gamma).
\end{align*}
As in the previous case we have 
\begin{align*}
&
\min I^b_{\chi_1}(f_\cV)=
v([\chi_1,1])\leq
v([\alpha,i+1])<
\sup I^b_{\chi_1}(f_\cV)
\\
&
\min I^b_{\chi_2}(f_\cV)=
v([\chi_2,1])\leq
v([\beta,j+1])<
\sup I^b_{\chi_2}(f_\cV).
\end{align*}
Recalling Equation~\eqref{EquationDefinitionThurstonMap} and arguing as in the previous case we get
\begin{align*}
&
\min I^t_{\chi_1}(f_\cV)=
v'([\chi_1,0])\leq
v'([\alpha,i])<
\sup I^t_{\chi_1}(f_\cV)
\\
&
\min I^t_{\chi_2}(f_\cV)=
v'([\chi_2,0])\leq
v'([\beta,j])<
\sup I^t_{\chi_2}(f_\cV).
\end{align*}
The proof is completed observing that 
$
\sup I^t_{\chi_1}(f_\cV)\leq \sup I^t_{\chi_2}(f_\cV)
$, 
because we proved yet that the critical points of $f_\cV$ have the same geometrical order as those of $T_\gamma$.

\smallskip

Part (2): now we prove that the preconfiguration $\cV'$ is a configuration. 
Assume that we have $i\not=0$ and a point $v'([\alpha,i])\in L_\beta(g_{\cV'})$ . Since 
$
v'([\alpha,i])=f_\cV^{-1}(v([\alpha,i+1]))
$ 
then Equation \eqref{EquationRelationFunctionsFandG} implies 
$$
v([\alpha,i+1])\in f_\cV\big( L_\beta(g_{\cV'})\big)
=
f_\cV\big( L_\beta(f_{\cV})\big)=v([\beta,1]).
$$
Thus $[\alpha,i+1]=[\beta,1]$ and hence $[\alpha,i]=[\beta,0]$. We get 
$v'([\alpha,i])=v'([\beta,0])$, that is the only element of $\cV'$ intersecting $L_\beta(g_{\cV'})$ is the critical point $v'([\beta,0])$. The Proposition is proved.
\end{proof}

\begin{lemma}
\label{LemmaUniformContinuityThurstonMap}
The Thurston map 
$
\cT_\gamma:\cO_\gamma\to\cO_\gamma
$ 
defined by Equation~\eqref{EquationDefinitionThurstonMap} is uniformly continuous.
\end{lemma}

\begin{proof}
For a configuration $\cV\in\cO_\gamma$ set 
$
\cV':=\widehat{\cT}_\gamma(\cV)
$, 
then for $[\alpha,i]\in\cI_\gamma$ let $v([\alpha,i])$ and $v'([\alpha,i])$ be the points in $\cV$ and $\cV'$ respectively. We prove that for any $[\alpha,i]\in\cI_\gamma$ the map
$$
\cV\mapsto v'([\alpha,i])
$$
is uniformly continuous. This is true for $i=0$ because the map $\cV\mapsto f_\cV$ in 
Equation \eqref{EquationFromConfigurationToGIET(SingularValues)} is uniformly continuous and for any $\alpha\in\cA$ we have 
$v'([\alpha,0])=u^t_\alpha(f_\cV)$.

In order to treat the other points, take a second configuration $\widetilde{\cV}$ and let 
$
\widetilde{v}([\alpha,i])
$ 
with $[\alpha,i]\in\cI_\gamma$ be its elements. 
Set 
$
\widetilde{\cV}':=\cT_\gamma(\widetilde{\cV})
$ 
and denote $\widetilde{v}'([\alpha,i])$ the points in $\widetilde{\cV}'$. Both $\cV$ and $\widetilde{\cV}$ are configurations, thus for $[\alpha,i]\in\cI_\gamma$ with $i\not=0$ let $\chi\in\cA$ be the letter with
$$
u^b_\chi(f_\cV)<v([\alpha,i+1])<\sup I^b_\chi(f_\cV)
\quad\text{and}\quad
u^b_\chi(f_{\widetilde{\cV}})<\widetilde{v}([\alpha,i+1])<\sup I^b_\chi(f_{\widetilde{\cV}}).
$$
Denote 
$
f_\chi:=f_\cV|_{J_\chi(f_\cV)}
$ 
and 
$
\widetilde{f}_\chi:=f_{\widetilde{\cV}}|_{J_\chi(f_{\widetilde{\cV}})}
$ 
the restrictions of $f_\cV$ and $f_{\widetilde{\cV}}$ to $J_\chi(f_\cV)$ and to $J_\chi(f_{\widetilde{\cV}})$ respectively. We have 
\begin{align*}
&
|v'([\alpha,i])-\widetilde{v}'([\alpha,i)]|
=
\big|
f_\chi^{-1}\big(v([\alpha,i+1])\big)-\widetilde{f}_\chi^{-1}\big(\widetilde{v}([\alpha,i+1])\big)
\big|\leq
\\
&
\distance(f_\chi^{-1},\widetilde{f}_\chi^{-1})
+
\omega\Big(
\distance(f_\chi^{-1},\widetilde{f}_\chi^{-1})
+
\big|v([\alpha,i+1])\big)-\widetilde{v}([\alpha,i+1])\big|
\Big)\leq
\\
&
\distance(f_\cV,f_{\widetilde{\cV}})
+
\omega\Big(
\distance(f_\cV,f_{\widetilde{\cV}})
+
|\cV-\widetilde{\cV}|
\Big)
\leq
\omega_F(|\cV-\widetilde{\cV}|)
+
\omega\Big(
\omega_F(|\cV-\widetilde{\cV}|)
+
|\cV-\widetilde{\cV}|
\Big).
\end{align*}
Here the first inequality follows combining Equations \eqref{EquationHausdorffDistanceBoundsDifference} and \eqref{EquationUniformContinuityModulus}, while the second follows from Corollary \ref{CorollaryHausdorffDistanceInverse}. The third inequality holds because the map in 
Equation \eqref{EquationFromConfigurationToGIET(SingularValues)} is uniformly continuous, thus there exists a modulus of continuity $\omega_F$ such that we have 
$
\distance(f_\cV,f_{\widetilde{\cV}})\leq\omega_F\big(|\cV-\widetilde{\cV}|\big)
$. 
Finally observe that for $\delta\to0$ we have
$$
\omega_F(\delta)+\omega\Big(\omega_F(\delta)+|\delta|\Big)\to0,
$$
thus the expression above is a modulus of continuity for $\cT_\gamma$. The Lemma is proved.
\end{proof}

\subsection{Extended map to the closure of the configuration space}
\label{SectionBoundaryConfigurationSpace}

According to Definition \ref{DefinitionGammaConfiguration} the configuration space $\cO_\gamma$ is an open subset of 
$\Delta^N$. A \emph{degenerate configuration} $\cV\in\overline{\Delta^N}$ is a set of points $v([\alpha,i])\in[0,1]$ labeled by $[\alpha,i]\in\cI_\gamma$ such that there exist configurations $\cV_n\in\cO_\gamma$, where $n\in\NN$ and 
$
\cV_n=\{v_n([\alpha,i]);[\alpha,i]\in\cI_\gamma\}
$,  
with $\cV_n\to\cV$, that is 
$$
v([\alpha,i])=\lim_{n\to\infty}v_n([\alpha,i])
\quad
\textrm{ for any }
\quad
[\alpha,i]\in\cI_\gamma.
$$
We also require that degenerate configurations are not configurations in the standard sense, that is the sequence 
$\cV_n$ satisfies one of the following properties:
\begin{enumerate}
\item
There exist two labels $[\alpha,i]\not=[\beta,j]$ such that 
$
v_n([\alpha,i])-v_n([\beta,j])\to0
$.
\item
We have $v_n([\alpha_m,i_m])\to1$ for $n\to\infty$. 
\item
A point in the configuration converges to a flat interval, that is we have $\alpha\in\cA$ with 
$$
v_n([\alpha,0]^\ast)-j_\alpha(g_{\cV_n})\to0.
$$
\end{enumerate}

The boundary $\partial\cO_\gamma$ of $\cO_\gamma$ is the set of all degenerate configurations. The closure of the configuration space is 
$
\overline{\cO(\gamma)}:=\cO(\gamma)\cup\partial\cO_\gamma
$. 
The Thurston map 
$
\cT_\gamma:\cO_\gamma\to\cO_\gamma
$ 
is uniformly continuous by Lemma \ref{LemmaUniformContinuityThurstonMap}, therefore it admits an unique continuous extension to the closure $\overline{\cO_\gamma}$ of its domain
\begin{equation}
\label{EquationDefinitionExtensionThurstonMap}
\widehat{\cT}_\gamma:\overline{\cO_\gamma}\to\overline{\cO_\gamma}
\quad
\textrm{ ; }
\quad
\cV\mapsto\widehat{\cT}_\gamma(\cV):=\lim_{n\to\infty}\cT_\gamma(\cV_n),
\end{equation}
where for any a configuration (degenerate or not) $\cV\in\overline{\cO_\gamma}$ and any sequence of configurations 
$\cV_n\in\cO_\gamma$ with $\cV_n\to\cV$ the limit above is well defined, it belongs to  $\overline{\cO_\gamma}$ and it does not depend on the specific choice of the sequence $\cV_n$ converging to $\cV$. Moreover if $\cV\in\cO_\gamma$ then we have $\widehat{\cT}_\gamma(\cV)=\cT_\gamma(\cV)$. 
Setting $\cV':=\widehat{\cT}_\gamma(\cV)$ and writing 
$
\cV'=\{v'([\alpha,i]);[\alpha,i]\in\cI_\gamma\}
$ 
we have
$$
v'([\alpha,i]):=
\left\{
\begin{array}{cccc}
\lim_{n\to\infty} u^t_\alpha(f_{\cV_n}) & \text{for} & i=0
\\
\lim_{n\to\infty}f_{\cV_n}^{-1}\big(v_n([\alpha,i+1])\big) & \text{for} & i\neq0.
\end{array}
\right.
$$

\section{Fixed points for the Thurston map: proof of Theorem~\ref{TheoremRealizationRauzyPath}}
\label{SectionExistenceFixedPointThurstonMap}

In this section we show that the Thurston map cannot fix an element of the boundary of the configuration space (\S~\ref{SectionCyclicBehavior}). This completes the proof of Theorem~\ref{TheoremRealizationRauzyPath} (\S~\ref{SectionProofTheoremRealizationRauzyPath}). We use the same notation as in \S~\ref{SectionProofFullFamilyTheorem} and \S~\ref{SectionThurstonMap}.

\subsection{Degenerate GIETs associated to configurations in the boundary}

Each factor map in Equation \eqref{EquationFromConfigurationToGIET(SingularValues)} can be continuously extended to the closure of its domain, thus by composition we get a map
$$
\overline{\Delta^N} \xrightarrow{V(\cdot)} 
\overline{\nabla_\pi} \xrightarrow{(u^b\circ F)^{-1}} 
\overline{\Delta^d} \xrightarrow{\widehat{F}} 
\widehat{\cC}\big(\pi,[0,1)\big),
$$
where the second factor is the extension to the closed simplex of the homeomorphism in Point (1) of Definition~\ref{DefinitionFullFamily} and the third factor is the extended map in Point (3). 
Postcomposing the map above with the inclusion $\overline{\cO(\gamma)}\subset\overline{\Delta^N}$ we get a map
\begin{equation}
\label{EquationFromConfigurationToGIET(Degenerate)}
\overline{\cO(\gamma)}\to\widehat{\cC}(\pi,[0,1))
\textrm{ ; }
\quad
\cV\mapsto D_\cV=(f_\cV,S_\cV).
\end{equation}

\begin{remark}
The map $\cV\mapsto D_\cV$ is the degenerate counterpart of the map $\cV\mapsto f_\cV$ in 
Equation \eqref{EquationFromConfigurationToGIET(SingularValues)}. We don't need to define a counterpart of the map $\cV\mapsto g_\cV$, indeed in the next Lemmas \ref{LemmaCyclicBehavior} and \ref{LemmaCyclicBehaviorBis} we look at flat intervals only \emph{after} the action of $\widehat{\cT}_\gamma$.
\end{remark}

We have  
$
D_\cV=\widehat{F}\big(\lambda(\cV)\big)
$, 
where $\lambda(\cdot)$ is the composition of the first two factor maps:
$$
\overline{\Delta^N} \xrightarrow{V(\cdot)} 
\overline{\nabla_\pi} \xrightarrow{(u^b\circ F)^{-1}} 
\overline{\Delta^d},
\qquad
\cV\mapsto\lambda(\cV).
$$
For a degenerate configuration $\cV\in\partial\cO(\gamma)$ it is possible to have a non degenerate vector  
$\lambda(\cV)\in\Delta^d$ and not necessarily $\lambda(\cV)\in\partial\Delta^d$. 
If $\lambda=\lambda(\cV)\in\Delta^d$, then 
$
\widehat{F}(\lambda)=f_\lambda\in\cC\big(\pi,[0,1)\big)
$. 
Otherwise if $\lambda=\lambda(\cV)\in\partial\Delta^d$, then we have a degeneration 
$
\widehat{F}(\lambda)=D_\lambda=(f_\lambda, S_\lambda)$, where $S_\lambda$ is the singular part and 
$
f_\lambda:[0,1)\to[0,1)
$ 
is a GIET whose combinatorial datum is a reduction of $\pi$. In both cases we say that $f_\lambda$ is the regular part of $\widehat{F}(\lambda)$.

\begin{lemma}
\label{LemmaControlExpansionInverseBranches}
Let $\lambda\in\overline{\Delta^d}$ and 
$
(\lambda_n)_{n\in\NN}\subset\Delta^d
$ 
be a sequence with $\lambda_n\to\lambda$ as $n\to\infty$. Let $I_n$ be a sequence of open intervals such that there exists $\alpha\in\cA$ with 
$$
I_n\subset \interior I^b_\alpha(f_{\lambda_n})
\quad
\textrm{ for any }
\quad
n\in\NN,
$$
so that $f_{\lambda_n}^{-1}(I_n)$ is an interval too with 
$
f_{\lambda_n}^{-1}(I_n)\subset J_\alpha(f_{\lambda_n})
$.
\begin{enumerate}
\item
If $|I_n|\to0$ as $n\to\infty$ then we have $|f_{\lambda_n}^{-1}(I_n)|\to0$. 
\item
If $\inf_{n\in\NN}|I_n|>0$ then we have $\inf_{n\in\NN}|f_{\lambda_n}^{-1}(I_n)|>0$.
\end{enumerate}
\end{lemma}

\begin{proof}
For any $n\in\NN$ set $c_n:=\inf I_n$, $d_n:=\sup I_n$, $a_n:=f_{\lambda_n}^{-1}(c_n)$ and 
$b_n:=f_{\lambda_n}^{-1}(d_n)$. Observe that Point (1) is obvious if $|I^t_\alpha(f_{\lambda_n})|\to0$. Otherwise there exists non-empty right-open intervals $J_\alpha$ and $I^b_\alpha$ with 
$
\closure\big(J_\alpha(f_{\lambda_n})\big)\to \closure(J_\alpha)
$ 
and 
$
\closure\big(I^b_\alpha(f_{\lambda_n})\big)\to \closure(I^b_\alpha)
$ 
in the Hausdorff distance, and the restriction 
$
f_\lambda:J_\alpha \to I^b_\alpha
$ 
is a homeomorphism, where $f_\lambda$ is the regular part of $\widehat{F}(\lambda)$. 
Its continuous extension, still denoted $f_\lambda$, is an homeomorphism from $\closure(J_\alpha)$ onto 
$\closure(I^b_\alpha)$. If $|b_n-a_n|$ does not converge to zero, consider $a,b\in\closure(I^b_\alpha)$ with $a<b$ strictly and a subsequence with $a_n\to a$ and $b_n\to b$. Take $c,d\in\closure(I^b_\alpha)$ and a subsubsequence such that $(a_n,c_n)\to(a,c)$ and $(b_n,d_n)\to(b,d)$.  
The continuity of $\widehat{F}$ implies  
$
\distance\big(f_{\lambda_n},\widehat{F}(\lambda)\big)\to0
$ 
and therefore $c=f_\lambda(a)$ and $d=f_\lambda(b)$. But $|d_n-c_n|\to0$, thus $c=d$. Injectivity of $f_\lambda$ implies $a=b$. This is absurd, thus Point (1) is proved. 
In order to prove Point (2) observe that we have $\inf_{n\in\NN}|I^b_\alpha(f_{\lambda_n})|>0$, therefore the restriction 
$
f_\lambda:J_\alpha \to I^b_\alpha
$ 
is a homeomorphism between non-trivial intervals. If $|b_n-a_n|\to0$ but $\inf|d_n-b_n|>0$, then modulo subsequences there are points $a,c,d$ with $a_n,b_n\to a$, $c_n\to c$ and $d_n\to d$, where $c<d$ strictly. Thus $f_\lambda(a)=c$ but also $f_\lambda(a)=d$, which is absurd. The Lemma is proved.
\end{proof}

\subsection{Behavior at the boundary}
\label{SectionCyclicBehavior}

Let $\cV^{(\gamma)}$ be the standard configuration in Equation~\eqref{EquationReferenceConfiguration}. 
Recall that the critical points and critical values of $T_\gamma$ are respectively 
$\underline{v([\alpha,0])}$ and $\underline{v([\alpha,1])}$ with $\alpha\in\cA$.
\begin{lemma}
\label{LemmaCyclicBehavior(CriticalValues)}
Fix $\alpha\in\cA$. We have 
$$
\underline{v([\alpha,1]^\ast)}=\sup I^b_{\alpha}(T_\gamma)
\Leftrightarrow
\underline{v([\alpha,0]^\ast)}=\sup I^t_{\alpha}(T_\gamma).
$$
Otherwise for $[\chi,j]:=[\alpha,1]^\ast$ we have
$$
\underline{v([\chi,j])}\in \interior\big(I^b_\alpha(T_\gamma)\big)
\Leftrightarrow
\left\{
\begin{array}{l}
[\alpha,0]^\ast=[\chi,j-1]
\\
\underline{v([\chi,j-1])}\in \interior\big(I^t_\alpha(T_\gamma)\big).
\end{array}
\right.
$$
\end{lemma}

\begin{proof}
Recall that $\cV^{(\gamma)}$ is a closed orbit of $T_\gamma$. Moreover $T_\gamma^{-1}$ acts continuously on 
$I^b_{\alpha}(T_\gamma)$. Therefore $I^t_\alpha(T_\gamma)$ contains points in the orbit in its interior if and only if $I^b_\alpha(T_\gamma)$ does. 
\end{proof}

\begin{remark}
\label{RemarkCyclicBehavior(3cases)}
The $d+1$ endpoints of the continuity intervals of $T_\gamma^{-1}$ belong to $\cV^{(\gamma)}\cup\{1\}$, 
where excluding $1$ all such endpoints are critical values of $T_\gamma$. Given two consecutive elements 
$
\underline{v([\alpha,i])}<\underline{v([\alpha,i])^\ast}
$ 
of $\cV^{(\gamma)}$, there are three possible cases.
\begin{enumerate}
\item
$\underline{v([\alpha,i])}$ is an endpoint, thai is $i=1$ and $\underline{v([\alpha,1])}$ is a critical value. In particular $0=\underline{v([\alpha,1])}$ is a critical value for the letter $\alpha$ with 
$\pi_b(\alpha)=1$. 
\begin{enumerate}
\item 
If 
$
\underline{v([\alpha,1]^\ast)}=\sup I^b_{\alpha}(T_\gamma)
$ 
then Lemma~\ref{LemmaCyclicBehavior(CriticalValues)} implies 
$$
\underline{v([\alpha,0]^\ast)}=\sup I^t_{\alpha}(T_\gamma).
$$
\item 
Otherwise, setting $[\chi,j]:=[\alpha,1]^\ast$, Lemma~\ref{LemmaCyclicBehavior(CriticalValues)} implies 
$$
[\alpha,0]^\ast=[\chi,j-1].
$$
\end{enumerate}
\item
$\underline{v([\alpha,i])^\ast}$ is an endpoint. In this case $\underline{v([\alpha,i])}$ and 
$\underline{v([\alpha,i])^\ast}$ do not belong to the same continuity interval of $T_\gamma^{-1}$. Letting $\beta$ be the letter with 
$
\underline{v([\alpha,i])}\in I^b_\beta(T_\gamma)
$ 
we have $\underline{v([\alpha,i]^\ast)}=\sup I^b_\beta(T_\gamma)$. Here  
$\underline{v([\alpha,i]^\ast)}=1$ if $\pi^b(\alpha)=d$, while otherwise 
$\underline{v([\alpha,i]^\ast)}$ is a critical value of $T_\gamma$. 
In both cases, arguing as in Lemma~\ref{LemmaCyclicBehavior(CriticalValues)}, we have 
$$
\underline{v([\alpha,i-1]^\ast)}=\sup I^t_\beta(T_\gamma).
$$ 
\item
Both $\underline{v([\alpha,i])}$ and $\underline{v([\alpha,i])^\ast}$ belong to the interior of a continuity interval of $T_\gamma^{-1}$. In this case, setting $[\chi,j]:=[\alpha,i]^\ast$ and arguing as in Lemma~\ref{LemmaCyclicBehavior(CriticalValues)}, we have 
$$
[\alpha,i-1]^\ast=[\chi,j-1].
$$ 
\end{enumerate}
\end{remark}

Consider a degenerate configuration 
$
\cV\in\partial\cO(\gamma)
$ 
and its image $\cV':=\widehat{\cT}_\gamma(\cV)$. This corresponds to considering a sequence $\cV_n\in\cO(\gamma)$ with 
$\cV_n\to\cV$ and its image sequence $\cV'_n:=\cT_\gamma(\cV_n)$, so that 
$\cV'=\lim_{n\to\infty}\cV'_n$. Let $v([\alpha,i])$, $v'([\alpha,i])$, $v_n([\alpha,i])$ and $v'_n([\alpha,i])$ be the points respectively in 
$\cV$, $\cV'$, $\cV_n$ and $\cV'_n$ corresponding $[\alpha,i]\in\cI_\gamma$. The configuration $\cV$ satisfies $v([\alpha,i])=v([\alpha,i]^\ast)$ if and only if 
\begin{equation}
\label{EquationAssumptionLemma(1)CyclicBehavior}
\big|
v_n([\alpha,i]^\ast)
-
v_n([\alpha,i])
\big|\to0
\quad
\textrm{ as }
\quad
n\to\infty.
\end{equation}
The image configuration $\cV'=\widehat{\cT}_\gamma(\cV)$ satisfies   
$
v'([\alpha,i-1])=v'([\alpha,i-1]^\ast)
$ 
if and only if  
\begin{equation}
\label{EquationConsequenceLemma(1)CyclicBehavior}
\big|
v'_n([\alpha,i-1]^\ast)
-
v'_n([\alpha,i-1])
\big|\to0
\quad
\textrm{ as }
\quad
n\to\infty.
\end{equation}
The continuity of the map in Equation \eqref{EquationFromConfigurationToGIET(Degenerate)} implies 
$f_{\cV_n}\to D_\cV=(f_\cV,S_\cV)$ for $\cV_n\to\cV$. Flat intervals of $f_{\cV_n}$ have a counterpart in the limit and for any $\alpha\in\cA$ we define
$$
j_\alpha(D_\cV):=\lim_{n\to\infty}j_\alpha(f_{\cV_n}).
$$

\begin{lemma}
\label{LemmaCyclicBehavior}
Fix $\alpha\in\cA$. If $v([\alpha,1])=v([\alpha,1]^\ast)$, then 
$$
v'([\alpha,0]^\ast)=j_\alpha(D_{\cV}).
$$
For any other $[\alpha,i]$ with $i\not=1$, if $v([\alpha,i])=v([\alpha,i]^\ast)$, then 
$$
v'([\alpha,i-1])=v'([\alpha,i-1]^\ast).
$$
\end{lemma}

\begin{remark}
\label{RemarkCyclicBehavior}
In the first part of Lemma~\ref{LemmaCyclicBehavior}, assume that the flat interval $L_\alpha(f_{\cV_n})$ collapses to a point, either because the entire $\alpha$-brach of $f_{\cV_n}$ collapses or because the latter is an homeomorphism in the limit. In both cases we have $j_\alpha(D_{\cV})=v'([\alpha,0])$ and therefore $v'([\alpha,0])=v'([\alpha,0]^\ast)$. Notice also that for $i=0$ the assumption in the second part implies 
$j_\alpha(D_{\cV})=v'([\alpha,0])$.
\end{remark}

\begin{proof}
We start by proving the first statement, which corresponds to Case (1) in 
Remark~\ref{RemarkCyclicBehavior(3cases)}. In Case (1a), since both $\cV_n$ and $\cV^{(\gamma)}$ are configurations, we have
\begin{equation}
\label{EquationCase(1a)LemmaCyclicBehavior}
v_n([\alpha,1]^\ast)=\sup I^b_{\alpha}(f_{\cV_n})
\quad\text{ and }\quad
v'_n([\alpha,0]^\ast)=\sup I^t_{\alpha}(f_{\cV_n}).
\end{equation}
In this case, Condition~\eqref{EquationAssumptionLemma(1)CyclicBehavior} with $i=1$ means 
$|I^b_{\alpha}(f_{\cV_n})|\to0$, which implies $|I^t_{\alpha}(f_{\cV_n})|\to0$. 
Equation~\eqref{EquationConsequenceLemma(1)CyclicBehavior} is satisfied because 
$$
v'_n([\alpha,0]^\ast)-v'_n([\alpha,0])
=
\sup I^t_{\alpha}(f_{\cV_n}) - \min I^t_{\alpha}(f_{\cV_n})\to 0.
$$
This means $v'([\alpha,0])=v'([\alpha,0]^\ast)$, and Remark \ref{RemarkCyclicBehavior} gives 
$
v'([\alpha,0])=j_\alpha(D_\cV)=v'([\alpha,0]^\ast)
$.

In Case (1b) setting $[\chi,j]:=[\alpha,1]^\ast$ we have $[\alpha,0]^\ast=[\chi,j-1]$. Therefore we also have 
\begin{equation}
\label{EquationCase(1b)LemmaCyclicBehavior}
\left\{
\begin{array}{l}
v_n([\alpha,1])<v_n([\chi,j])=v_n([\alpha,1]^\ast)<\sup I^b_{\alpha}(f_{\cV_n})
\\
v'_n([\alpha,0])<v'_n([\chi,j-1])=v_n([\alpha,0]^\ast)<\sup I^t_{\alpha}(f_{\cV_n}).
\end{array}
\right.
\end{equation}
We obtain $v'([\alpha,0]^\ast)=j_\alpha(D_{\cV})$ because for any $n$ the interval 
$
I_n:=\big(v_n([\alpha,1]),v_n([\chi,j])\big)
$ 
satisfies the assumption in Point (1) of Lemma~\ref{LemmaControlExpansionInverseBranches}, so we have 
\begin{align*}
\big|v'_n([\alpha,0]^\ast)-j_\alpha(f_{\cV_n})\big|
&
=
\big|\big(j_\alpha(f_{\cV_n}),v'_n([\chi,j-1])\big)\big|
\\
&
=
\big|f_{\cV_n}^{-1}\big(v_n([\alpha,1]),v_n([\chi,j])\big)\big|
=
\big|f_{\cV_n}^{-1}(I_n)\big|\to0.
\end{align*}

We now prove the second part of the statement, considering separately the remaining cases in Remark~\ref{RemarkCyclicBehavior(3cases)}. In Case (2) we have 
$\underline{v([\alpha,i]^\ast)}=\sup I^b_\beta(T_\gamma)$ and
$\underline{v([\alpha,i-1]^\ast)}=\sup I^t_\beta(T_\gamma)$, thus we also have
\begin{equation}
\label{EquationCase(2)LemmaCyclicBehavior}
\left\{
\begin{array}{l}
\inf I^b_\beta(f_{\cV_n})<v_n([\alpha,i])<v_n([\alpha,i]^\ast)=\sup I^b_\beta(f_{\cV_n})
\\
\inf J_\beta(f_{\cV_n})<v'_n([\alpha,i-i])<v'_n([\alpha,i-1]^\ast)=\sup I^t_\beta(f_{\cV_n}).
\end{array}
\right.
\end{equation}
Condition~\eqref{EquationAssumptionLemma(1)CyclicBehavior} means that the interval 
$
I_n:=\big(v_n([\alpha,i]),\sup I^b_\beta(f_{\cV_n})\big)
$ 
satisfies the assumption in Point (1) of Lemma~\ref{LemmaControlExpansionInverseBranches}. 
Equation~\eqref{EquationConsequenceLemma(1)CyclicBehavior} follows since we have
$$
\big|
v'_n([\alpha,i-1]^\ast)-v'_n([\alpha,i-1])
\big|
=
\big|
\sup I^t_\beta(f_{\cV_n})
-
f_{\cV_n}^{-1}\big(v_n([\alpha,i])\big)
\big|
=
\big|f_{\cV_n}^{-1}\big(I_n\big)\big|\to0.
$$

In Case (3), setting $[\chi,j]:=[\alpha,i]^\ast$, we have $[\alpha,i-1]^\ast=[\chi,j-1]$. Therefore, if $\beta$ is the letter with $v_n([\alpha,i])\in I^b_\beta(f_{\cV_n})$, then we have also 
\begin{equation}
\label{EquationCase(3)LemmaCyclicBehavior}
\left\{
\begin{array}{l}
\inf I^b_\beta(f_{\cV_n})<v_n([\alpha,i])<v_n([\chi,j])<\sup I^b_\beta(f_{\cV_n})
\\
\inf J_\beta(f_{\cV_n})<v'_n([\alpha,i-i])<v'_n([\chi,j-1])<\sup I^t_\beta(f_{\cV_n}).
\end{array}
\right.
\end{equation}

Equation~\eqref{EquationConsequenceLemma(1)CyclicBehavior} follows considering the interval 
$
I_n:=\big(v_n([\alpha,i]),v_n([\chi,j])\big)
$ 
and applying Point (1) of Lemma~\ref{LemmaControlExpansionInverseBranches}. Details are left to the reader. The Lemma is proved. 
\end{proof}

\begin{lemma}
\label{LemmaCyclicBehaviorBis}
Assume that $v([\alpha,1]^\ast)>v([\alpha,1])$ strictly. Then we have the strict inequality
$$
v'([\alpha,0]^\ast)>j_\alpha(D_\cV).
$$
For any other $[\alpha,i]$ with $i\not=1$, if $v([\alpha,i]^\ast)>v([\alpha,i])$ strictly, then we have the strict inequality
$$
v'([\alpha,i-1]^\ast)>v'([\alpha,i-1]).
$$
\end{lemma}

\begin{proof}
Observe preliminarily that the assumption in both statements corresponds to 
$$
\inf_{n\geq0} v_n([\alpha,i]^\ast)-v_n([\alpha,i])>0.
$$
We prove the first part of the statement, which corresponds to Case (1) in Remark~\ref{RemarkCyclicBehavior(3cases)}. 
In Case (1a) Equation~\eqref{EquationCase(1a)LemmaCyclicBehavior} holds. The interval 
$
I_n:=\big(v_n([\alpha,1]),v_n([\alpha,1]^\ast)\big)
$ 
has preimage 
$$
f_{\cV_n}^{-1}(I_n)
=
\big(j_\alpha(f_{\cV_n}),v'_n([\alpha,0]^\ast)\big).
$$
The assumption means $\inf_{n\in\NN}|I_n|>0$, thus we have 
$
\inf_{n\in\NN}|f_{\cV_n}^{-1}(I_n)|>0
$ 
by Point (2) of Lemma~\ref{LemmaControlExpansionInverseBranches}. In the limit we have 
$v'([\alpha,0]^\ast)>j_\alpha(D_\cV)$ strictly. In Case (1b) setting $[\chi,j]:=[\alpha,1]^\ast$ we have $[\alpha,0]^\ast=[\chi,j-1]$ and Equation~\eqref{EquationCase(1b)LemmaCyclicBehavior} holds.
The interval
$$
I_n:=\big(v_n([\alpha,1]),v_n([\chi,j])\big)=\big(v_n([\alpha,1]),v_n([\alpha,1]^\ast)\big)
$$
has preimage 
$$
f_{\cV_n}^{-1}(I_n)=\big(j_\alpha(f_{\cV_n}),v'_n([\chi,j-1])\big)
=\big(j_\alpha(f_{\cV_n}),v'_n([\alpha,0]^\ast)\big).
$$ 
The assumption means $\inf_{n\in\NN}|I_n|>0$, therefore 
$
\inf_{n\in\NN}\big|f_{\cV_n}^{-1}(I_n)\big|>0
$ 
by Point (2) of Lemma~\ref{LemmaControlExpansionInverseBranches}. In the limit we have 
$v'([\alpha,0]^\ast)>j_\alpha(D_\cV)$.

Now we prove the second part of the statement, considering separately the remaining cases in Remark~\ref{RemarkCyclicBehavior(3cases)}. In Case (2) Equation~\eqref{EquationCase(2)LemmaCyclicBehavior} holds. The interval 
$$
I_n:=\big(v_n([\alpha,i]),v_n([\alpha,i]^\ast)\big)
=
\big(v_n([\alpha,i]),\sup I^b_\beta(f_{\cV_n})\big)
$$
has preimage 
$$
f_{\cV_n}^{-1}(I_n)=\big(v'_n([\alpha,i-1]),\sup I^t_\beta(f_{\cV_n})\big)
=\big(v'_n([\alpha,i-1]),v'_n([\alpha,i-1]^\ast)\big)
$$ 
The assumption means $\inf_{n\in\NN}|I_n|>0$, thus 
$
\inf_{n\in\NN}\big|f_{\cV_n}^{-1}(I_n)\big|>0
$ 
by Point (2) of Lemma~\ref{LemmaControlExpansionInverseBranches}. In the limit we have 
$v'([\alpha,0]^\ast)>v'([\alpha,0])$ strictly. In Case (3) Equation~\eqref{EquationCase(3)LemmaCyclicBehavior} holds. The interval 
$
I_n:=\big(v_n([\alpha,i]),v_n([\chi,j])\big)
$ 
has preimage
$$
f_{\cV_n}^{-1}(I_n)=\big(v'_n([\alpha,i-1]),v'_n([\chi,j-1])\big).
$$
We have $\inf_{n\in\NN}|I_n|>0$, thus 
$
\inf_{n\in\NN}\big|f_{\cV_n}^{-1}(I_n)\big|>0
$ 
by Point (2) of Lemma~\ref{LemmaControlExpansionInverseBranches}. In the limit we have 
$
v'([\alpha,i-1]^\ast)=v'([\chi,j-1])>v'([\alpha,i-1])
$ 
strictly. The Lemma is proved.
\end{proof}

\begin{proposition}
\label{PropositionCyclicBehaviour}
The extended map $\widehat{\cT}_\gamma$ in Equation \eqref{EquationDefinitionExtensionThurstonMap} 
does not fix any point in the boundary $\partial\cO_\gamma$.
\end{proposition}

\begin{proof}
Take $\cV\in\overline{\cO_\gamma}$ with $\widehat{\cT}_\gamma(\cV)=\cV$. If $\cV\in\partial\cO_\gamma$ then consider 
$
D_\cV=(f_\cV,S_\cV)\in\widehat{\cC}\big(\pi,[0,1)\big)
$. 
First assume $\cV\in\partial\Delta^N$, that is there exists $[\alpha,i]$ with 
$v([\alpha,i]^\ast)=v([\alpha,i])$. Recall that the map 
$[\chi,j]\mapsto[\chi,j+1]$ is a cyclic permutation of maximal length $N=\sharp\cI_\gamma$. 
Applying $i$ times $\widehat{\cT}_\gamma$ to $\cV$ Lemma~\ref{LemmaCyclicBehavior} gives $v([\alpha,0]^\ast)=j_\alpha(D_\cV)$. 
If $j_\alpha(D_\cV)>v([\alpha,0])$ then $v([\alpha,0]^\ast)>v([\alpha,0])$ strictly. Applying $N-i$ more times 
$\widehat{\cT}_\gamma$ to $\cV$ Lemma~\ref{LemmaCyclicBehaviorBis} gives 
$v([\alpha,i]^\ast)>v([\alpha,i])$, which is absurd. 
Otherwise $v([\alpha,0]^\ast)=v([\alpha,0])$, and we keep applying $\widehat{\cT}_\gamma$ to $\cV$ until Lemma~\ref{LemmaCyclicBehavior} gives the next letter $\beta$ with $v([\beta,0]^\ast)=j_\beta(D_\cV)$. Reasoning as above we see that we must have $v([\chi,j]^\ast)=v([\chi,j])$ for any 
$[\chi,j]$, that is $\cV$ defines a partition of $[0,1)$ into $N$ subintervals, but all the endpoints of these intervals coincide, in particular $0=1$, which is absurd. Now assume $\cV\in\Delta^N\cap\partial\cO_\gamma$. We have 
$
D_\cV=f_\cV\in\cC\big(\pi,[0,1)\big)
$ 
and $v([\chi,j]^\ast)>v([\chi,j])$ strictly for any $[\chi,j]$, moreover there exists $\alpha$ with 
$j_\alpha(g_\cV)=v([\alpha,0]^\ast)$. Since $\cT_\gamma(\cV)=\cV$ then 
Equation \eqref{EquationRelationFunctionsFandG} gives 
$j_\alpha(f_\cV)=j_\alpha(g_\cV)$. 
In conclusion $j_\alpha(f_\cV)=v([\alpha,0]^\ast)$. But $v([\alpha,1]^\ast)>v([\alpha,1])$, thus Lemma~\ref{LemmaCyclicBehaviorBis} gives $v([\alpha,0]^\ast)>j_\alpha(f_\cV)$ strictly, which is absurd. 
The Proposition is proved.
\end{proof}

\subsection{End of the proof of Theorem~\ref{TheoremRealizationRauzyPath}}
\label{SectionProofTheoremRealizationRauzyPath}

Theorem~\ref{TheoremRealizationRauzyPath} follows from the boundary behaviour determined in the previous 
\S~\ref{SectionCyclicBehavior} and from the next two Lemmas.

\begin{lemma}
\label{LemmaRetraction}
There exists a retraction $r:\overline{\Delta^N}\to\overline{\cO_\gamma}$, that is a continuous map such that 
$r(\cV)=\cV$ for any $\cV\in\overline{\cO_\gamma}$.
\end{lemma}

\begin{proof}
Each factor map in Equation \eqref{EquationFromConfigurationToGIET(SingularPoints)} can be continuously extended to the closure of its domain, thus by composition we get a continuous map
$$
\overline{\Delta^N} \xrightarrow{U(\cdot)} 
\overline{\nabla_\pi^t} \xrightarrow{(u^t\circ F)^{-1}} 
\overline{\Delta^d} \xrightarrow{\widehat{F}} 
\widehat{\cC}\big(\pi,[0,1)\big),
\qquad
\cV\mapsto D_\cV=(g_\cV,S_\cV)
$$
where the second factor is the extension to the closed simplex of the homeomorphism in Point (2) of Definition~\ref{DefinitionFullFamily} and the third factor is the extended map in Point (3). We interpret any 
$\cV\in\overline{\Delta^N}$ as a degenerate preconfigurations, that is the limit of a sequences of preconfigurations $\cV_n\in\Delta^N$. Take any $\cV\in\overline{\Delta^N}$ and let $D_\cV=(g_\cV,S_\cV)$ be the degenerate GIET given by the map above. Recall \S~\ref{SectionDescriptionAllowedDegenerations}. Letters 
$\alpha\in\cA\setminus\cA(D)$ correspond to collapsed branches of $D_\cV$, represented by points in $S_\cV$. On the other hand for any $\alpha\in\cA(D)$ the regular part $g_\cV$ has a branch as in Equation \eqref{EquationPiecewiseHomeomorphism}. For any $[\beta,i]\in\cI_\gamma$ and let 
$\alpha=\alpha([\beta,i])\in\cA$ be the letter such that
$
\underline{v([\beta,i])}\in I^t_\alpha(T_\gamma)
$. 
We define $\cV'=r(\cV)$ by defining its elements $v'([\beta,i])$ for $[\beta,i]\in\cI_\gamma$.
\begin{itemize}
\item 
If $\alpha=\alpha([\beta,i])\in\cA\setminus\cA(D)$ then we set 
$$
v'([\beta,i]):=v([\beta,i]).
$$
\item
If $\alpha=\alpha([\beta,i])\in\cA(D)$ then we set
$$
v'([\beta,i]):=
\left\{
\begin{array}{ccc}
v([\alpha,0])  & \text{if} & [\beta,i]=[\alpha,0]
\\
j_\alpha(g_\cV)  & \text{if} & v[\alpha,0]\leq v[\beta,i]\leq j_\alpha(g_\cV) 
\\
v([\beta,i]) & \text{if} &  j_\alpha(g_\cV)<v[\beta,i]< v([\alpha_\ast,0]),
\end{array}
\right.
$$
that is those $v([\beta,i])$ with $[\beta,i]\not=[\alpha,0]$ which are inside the flat interval $L_\alpha(g_\cV)$ are moved to its right endpoint.
\end{itemize}
It is clear that $\cV'\in\overline{\cO_\gamma}$. Moreover $\cV'=\cV$ for $\cV\in\overline{\cO_\gamma}$, indeed in this case Definition \ref{DefinitionGammaConfiguration} implies 
$$
v[\alpha,0]\leq v[\beta,i]\leq j_\alpha(g_\cV)
\Leftrightarrow
[\beta,i]=[\alpha,0].
$$
Finally the map $\cV\mapsto g_\cV\mapsto j_\alpha(g_\cV)$ is continuous, because it is composition of continuous maps, 
therefore the map $\cV\mapsto r(\cV)$ is continuous. The Lemma is proved.
\end{proof}

Since 
$
r(\overline{\Delta^N})\subset\overline{\Delta^N}
$ 
and 
$
\overline{\cO_\gamma}\subset\overline{\Delta^N}
$ 
then we have a continuous map 
$$
\widehat{\cT}_\gamma\circ r:\overline{\Delta^N}\to \overline{\Delta^N}.
$$
Since $\overline{\Delta^N}$ is compact and convex, then Brouwer's fixed point Theorem implies that there exists $\cV\in\overline{\Delta^N}$ with 
$
\widehat{\cT}_\gamma\circ r(\cV)=\cV
$. 
The map above has values inside $\overline{\cO_\gamma}$, thus $\cV\in\overline{\cO_\gamma}$. Therefore $r(\cV)=\cV$, that is we have a degenerate configuration $\cV$ with $\widehat{\cT}_\gamma(\cV)=\cV$. According to Proposition \ref{PropositionCyclicBehaviour} we must have $\cV\in\cO_\gamma$ and 
$
\cT_\gamma(\cV)=\widehat{\cT}_\gamma(\cV)=\cV
$. 
Therefore Theorem~\ref{TheoremRealizationRauzyPath} follows from the next Lemma~\ref{LemmaFixedPointsThurstonAndRauzyPath}.

\begin{lemma}
\label{LemmaFixedPointsThurstonAndRauzyPath}
Let $\cV\in\cO_\gamma$ be a configuration such that $\cT_\gamma(\cV)=\cV$. Then $f_\cV$ has the same Rauzy path as $T_\gamma$ up to time $r$, that is 
$$
\gamma(f_\cV,r)=\gamma=\gamma(T_\gamma,r).
$$
\end{lemma}

\begin{proof}
According to Proposition~\ref{PropositionPartitionDeterminesRauzyPath}, it is enough to prove 
$
\cP(f_\cV,r)\sim\cP(T_\gamma,r)
$. 
If 
$
\cT_\gamma(\cV)=\cV
$ 
then Equation~\eqref{EquationDefinitionThurstonMap} implies 
$
v([\alpha,i+1])=f_\cV\big(v([\alpha,i])\big)
$ 
for any $[\alpha,i]\in\cI_\gamma$. It follows that the points in $\cV$ form a closed orbit of $f_\cV$, and their dynamical order is the same as the cyclic dynamical order of the points in $\cV^{(\gamma)}$ under the action of $T_\gamma$. By assumption $\cV$ contains the critical values $v([\alpha,1])=u^b_\alpha(f_\cV)$ and the critical points 
$
v([\alpha,0])=v'([\alpha,0])=u^t_\alpha(f_\cV)
$ 
of $f_\cV$, where $\alpha\in\cA$ and 
$
u^b_\alpha(f_\cV)=f_\cV\big(u^t_\alpha(f_\cV)\big)
$. 
Therefore $\cV$ equals the set of points 
$$
f_\cV^i\big(u^t_\alpha(f_\cV)\big),\quad
\alpha\in\cA,\quad
-h(\alpha,r)\leq i\leq q^{(r)}_\alpha-h(\alpha,r)-1,
$$
where the integers $h(\alpha,r)$ are defined in Equation~\eqref{EquationHalfReturnTimes}. Moreover these points have the same geometrical order as the points of the standard configuration $\cV^{(\gamma)}$, which are
$$
T_\gamma^i\big(u^t_\alpha(T_\gamma)\big)
\quad,\quad
\alpha\in\cA
\quad,\quad
-h(\alpha,r)\leq i\leq q^{(r)}_\alpha-h(\alpha,r)-1.
$$
These two families of points are the endpoints of the intervals in the partitions $\cP(f_\cV,r)$ and 
$\cP(T_\gamma,r)$ respectively, thus the Lemma follows.
\end{proof}

\section{Translation surfaces and zippered rectangles}
\label{SectionTranslationSurfaces}

\subsection{Translation surfaces, strata and vertical vector field}
\label{SectionTranslationSurfacesStrataVerticalField}

We introduce some basic notions on translation surfaces. For a more complete overview see \cite{ForniMatheus}, \cite{Yoccoz} and \cite{Zorich}.

Consider $2n$ triangles $T_1,\dots,T_{2n}$ in $\RR^2$. Any $T_i$ has a positively oriented boundary 
$
\partial T_i=\gamma^{(i)}_1\cup\gamma^{(i)}_2\cup\gamma^{(i)}_3
$, 
where for $k=1,2,3$ any side $\gamma^{(i)}_k$  inherits the induced orientation, so that its initial point $\xi_0(\gamma^{(i)}_k)$ and ending point $\xi_1(\gamma^{(i)}_k)$ are well defined. In particular the differences 
$
z^{(i)}_k:=\xi_1(\gamma^{(i)}_k)-\xi_0(\gamma^{(i)}_k)
$ 
satisfy $z^{(i)}_1+z^{(i)}_2+z^{(i)}_3=0$ for any $i$. We assume that for any $z^{(i)}_k$ there exists $z^{(j)}_l$ such that $z^{(i)}_k=-z^{(j)}_l$, that is $\gamma^{(i)}_k$ and $\gamma^{(j)}_l$ are parallel, with same length and opposite orientation. The identifications between such pairs $\gamma^{(i)}_k$ and $\gamma^{(j)}_l$ induces an equivalence relation $\sim$ on the union of the boundaries of the triangles above. A \emph{translation surface} $X$ is the quotient space 
\begin{equation}
\label{EquationDefinitionTranslationSurface(QuotientSurface)}
X:=\bigg(\bigsqcup_{i=1}^{2n}T_i\bigg)\bigg/\sim.
\end{equation}
We assume that the identification gives a connected quotient space $X$. If $n=1$ then $X=P/\sim$ is a flat torus. In general a translation surface is a compact surface $X$ of genus $g\geq1$, with a metric which is flat outside of a finite set of points $\{p_1,\dots,p_m\}$, where the metric has a conical singularity with angle $2(k_j+1)\pi$, $k_j\in\NN$. Any $p_j$ corresponds to a subset of the vertices of $T_1,\dots,T_{2n}$, all identified to the same point in $X$ by Equation~\eqref{EquationDefinitionTranslationSurface(QuotientSurface)}. We have 
$
k_1+\dots+k_m=2g-2
$. 
Any triangle corresponds to an open set in $X\setminus\{p_1,\dots,p_m\}$ identified with an open subset of $\RR^2$. Representing the same surface $X$ by different sets of triangles we get changes of charts which are translations. These changes of charts extend to a maximal \emph{translation atlas} of maps 
$\varphi:U\to\RR^2$ from open sets   
$
U\subset X\setminus\{p_1,\dots,p_m\}
$ 
onto open sets of $\RR^2$ with the properties below.
\begin{itemize}
\item
Any coordinate change between two charts of the atlas is a translation of $\RR^2$.
\item
For any $1\leq j\leq m$ there exists a neighborhood $U_j$ of $p_j$, a ball $B(0,r_j)$ of radius $r_j>0$ around the origin $0\in\RR^2$ and a ramified covering 
\begin{equation}
\label{EquationDefinitionTranslationSurface(Ramification)}
\rho_j:\big(U_j,p_j\big)\to \big(B(0,r_j),0\big)
\end{equation}
of degree $k_j+1$ such that every injective restriction of $\rho_j$ is a chart of the atlas.
\end{itemize}

It is possible to see that the translation atlas extends to a holomorphic atlas at the conical points $p_1,\dots,p_m$ of $X$. Therefore the latter has a smooth atlas (see \S~1.1 in \cite{Yoccoz}). 

\subsubsection{Strata and their connected components}
\label{SectionConnectedComponentsStrata}

A \emph{stratum} is the set of all translation surfaces $X$ with the same orders $k_1,\dots,k_m$ of conical singularities. It is an affine \emph{orbifold} with real dimension $2\cdot(2g-1+m)$, the affine coordinates around any $X$ being the vectors 
$
(z^{(i)}_k)_{i=1,\dots,6n}\in\RR^{12n}
$ 
introduced above, modulo a system of linear equations with coefficients in $\QQ$. Strata have up to three connected components, which have been classified in 
\cite{KontsevichZorich} and are in bijection with Rauzy classes, in a sense which is explained in \cite{Veech}.

\subsubsection{Vertical vector field}

Equation~\eqref{EquationDefinitionTranslationSurface(QuotientSurface)} identifies the sides of the triangles $T_1,\dots,T_{2n}$ by translations. Thus the constant vector field $V_0:=(0,1)$ on the disjoint union 
$\bigsqcup_{i=1}^{2n}T_i$ induces a constant vector field on $X\setminus\{p_1,\dots,p_m\}$, that we still denote $V_0$. The first return of $V_0$ to a transversal segment in $X$ is an IET. In \cite{Veech},  Veech introduced a representation of translation surfaces as the disjoint unions of euclidean rectangles, with sides aligned along the direction of $V_0$ and its transversal, so that the identifications between transversal sides are given by an IET. This construction is recalled in \S~\ref{SectionZipperedRectangles} below, following \S~4.3 in \cite{Yoccoz}.

\subsection{Zippered rectangles}
\label{SectionZipperedRectangles}

Let $\pi$ be an admissible combinatorial datum over an alphabet $\cA$ with $d\geq2$ letters. A \emph{suspension datum} for the combinatorial datum $\pi$ is a vector $\tau\in\RR^\cA$ such that for any $\alpha$ and $\beta$ with $\pi_t(\alpha)\geq2$ and $\pi_b(\beta)\geq2$ respectively, we have 
$$
\sum_{\pi_t(\chi)<\pi_t(\alpha)}\tau_\chi>0
\quad
\textrm{ and }
\quad
\sum_{\pi_b(\chi')<\pi_b(\beta)}\tau_{\chi'}<0.
$$

Fix combinatorial-length-suspension data $(\pi,\lambda,\tau)$ and let $T=T(\pi,\lambda)$ be the corresponding IET. 
For $\alpha\in\cA$ let $I^t_\alpha$ and $I^b_\alpha$ be the right-open intervals in Equation~\eqref{EquationIntervalsIET}. The procedure below defines a translation surface 
$
X=X(\pi,\lambda,\tau)
$. 
An example is shown in the left part of Figure~\ref{FigureZipperedRectanglesAndSaddleSink}. 
Modulo the identification $\CC\simeq\RR^2$, for any $\alpha\in\cA$ consider  
$
\zeta_\alpha:=\lambda_\alpha+i\tau_\alpha 
$ 
and set
$$
\xi^t_\alpha:=
\sum_{\pi_t(\chi)<\pi_t(\alpha)}\zeta_\chi
\quad
\textrm{ and }
\quad
\xi^b_\beta:=
\sum_{\pi_b(\chi)<\pi_b(\beta)}\zeta_{\chi}.
$$
Let $h=h(\pi,\tau)\in\RR_+^\cA$ be the vector whose coordinates, for any $\alpha\in\cA$, are defined by
$$
h_\alpha:=
\sum_{\pi_t(\chi)<\pi_t(\alpha)}\tau_\chi
-
\sum_{\pi_b(\chi')<\pi_b(\alpha)}\tau_{\chi'}.
$$
For any $\alpha\in\cA$ we have 
$
h_\alpha=\im(\xi^t_\alpha)-\im(\xi^b_\alpha)>0
$ 
and 
$$
0\leq \im(\xi^t_\alpha)\leq h_\alpha
\quad
\textrm{ and }
\quad
0\geq \im(\xi^b_\alpha)\geq -h_\alpha.
$$
For $\alpha\in\cA$ define the rectangle 
$
R_\alpha^t:=I_\alpha^t\times[0,h_\alpha]\subset\RR^2
$. 
Proper identifications along the boundaries of these rectangles will define the translation surface $X$. The points 
$\xi^{t/b}_\alpha$ represent the coordinates of the conical singularities of $X$ in the charts associated to the rectangles. We have $\xi^t_\alpha=0$ if $\pi_t(\alpha)=1$ and 
$
\im(\xi^t_\alpha)>0
$ 
whenever $2\leq \pi_t(\alpha)\leq d$. Similarly $\xi^b_\beta=0$ if $\pi_b(\beta)=1$ and 
$
\im(\xi^b_\beta)<0
$ 
whenever $2\leq \pi_b(\beta)\leq d$. The identifications between the horizontal sides of the rectangles $R_\alpha^t$ for $\alpha\in\cA$ are given by 
$$
(x,h_\alpha)\sim (Tx,0)
\quad
\textrm{ iff }
\quad
x\in I^t_\alpha.
$$

In order to describe the identifications between the vertical sides of the rectangles above, it is convenient to introduce a copy of them setting 
$
R_\alpha^b:=I_\alpha^b\times[-h_\alpha,0]
$ 
for any $\alpha$. Such copies are different translation charts, but in the quotient surface we have $R_\alpha^t=R_\alpha^b$. For letters $\alpha$ and $\beta$ with respectively 
$
\pi_t(\alpha)\geq2
$ 
and 
$
\pi_b(\beta)\geq2
$ 
define the vertical segments 
$$
S_\alpha^t:=
\{\inf I_\alpha^t\}\times\big[0,\sum_{\pi_t(\chi)<\pi_t(\alpha)}\tau_\chi\big)
\quad
\textrm{ and }
\quad
S_\beta^b:=
\{\inf I_\beta^b\}\times\big(\sum_{\pi_b(\chi)<\pi_b(\beta)}\tau_\chi,0\big].
$$
Consider also the vertical segment 
$$
S^\ast:=\{\sum_{\chi\in\cA}\lambda_\chi\}\times \big(0,\tau^\ast\big)
\quad\text{ where }\quad
\tau_\ast:=\sum_{\chi\in\cA}\tau_\chi.
$$
\begin{enumerate}
\item
For $\alpha$ with $\pi_t(\alpha)\geq2$ let $\alpha'$ be the letter with 
$
\pi_t(\alpha')=\pi_t(\alpha)-1
$, 
then paste $R_{\alpha'}^t$ with $R_\alpha^t$ along the segment $S_\alpha^t$. 
\item
For $\beta$ with $\pi_b(\beta)\geq2$ let $\beta'$ be the letter with 
$
\pi_b(\beta')=\pi_b(\beta)-1
$, 
then paste $R_{\beta'}^b$ with $R_\beta^b$ along the segment $S_\beta^b$. Since any $R_\alpha^b$ is identified with $R_\alpha^t$, this induces a new gluing between $R_\beta^t$ and $R_{\beta'}^t$ missing in Point (1).
\item
If $\tau_\ast\geq0$, identify $S_\ast$ with $S_\alpha^t\setminus\partial R_{\beta_0}$, 
where $\beta_0$ and $\alpha$ are the letters with respectively $\pi_b(\beta_0)=d$ and  $\pi_t(\alpha)=\pi_t(\beta_0)+1$. If $\tau_\ast<0$, identify $S_\ast$ with
$
S_\beta^b\setminus\partial R_{\alpha_0}
$, 
where $\alpha_0$ and $\beta$ are the letters with respectively $\pi_t(\alpha_0)=d$ and 
$
\pi_b(\beta)=\pi_b(\alpha_0)+1
$.  
\end{enumerate}

\subsection{Marking a point on a translation surface}
\label{SectionMarkingPointTranslationSurface}

Fix $\chi\in\cA$. Let 
$
\widehat{\cA}:=\cA\setminus\{\chi\}\cup\{\chi_L,\chi_R\}
$ 
be the alphabet obtained from $\cA$ replacing the letter $\chi$ by the pair $\chi_L,\chi_R$, so that 
$\widehat{\cA}$ contains $d+1$ letters. Let $\widehat{\pi}$ be the combinatorial datum over $\widehat{\cA}$ obtained from $\pi$ marking the letter $\chi$. Such $\widehat{\pi}$ is admissible over 
$\widehat{\cA}$. For example, if $\cA=\{A,B,C,D\}$ and $\chi=C$, then 
$$
\pi:=
\begin{pmatrix}
A & B & C & D\\
D & C & B & A
\end{pmatrix}
\quad
\textrm{ and }
\quad
\widehat{\pi}:=
\begin{pmatrix}
A & B & C_L & C_R & D\\
D & C_L & C_R & B & A
\end{pmatrix}.
$$
Let $(\pi,\lambda,\tau)$ be data over $\cA$ representing a surface $X=X(\pi,\lambda,\tau)$ in $\cH$. The data $(\pi,\lambda)$ define an IET acting on 
$
\big[0,\sum_{\alpha\in\cA}\lambda_\alpha\big)
$. 
For fixed $u$ with 
\begin{equation}
\label{EquationParameterU}
u^t_\chi<u<u^t_\chi+\lambda_\chi
\end{equation} 
define length and suspension data 
$
(\widehat{\lambda},\widehat{\tau})
$ 
by 
\begin{equation}
\label{EquationLengthDataMarkedSurface}
\left\{
\begin{array}{l}
\widehat{\lambda}_\alpha:=\lambda_\alpha\text{ for }\alpha\not=\chi\\
\widehat{\lambda}_{\chi_L}:=u-u^t_\chi\\
\widehat{\lambda}_{\chi_R}:=u^t_\chi+\lambda_\chi-u
\end{array}
\right.
\quad
\text{ and }
\quad
\left\{
\begin{array}{l}
\widehat{\tau}_\alpha:=\tau_\alpha\text{ for }\alpha\not=\chi\\
\widehat{\tau}_{\chi_L}:=\tau_\chi\\
\widehat{\tau}_{\chi_R}:=0,
\end{array}
\right.
\end{equation}
so that the data 
$
(\widehat{\pi},\widehat{\lambda},\widehat{\tau})
$ 
represent the same surface $X$ together with a marked point. In the following, the datum $\tau$ is fixed, so that $\widehat{\tau}$ is fixed too, while $\lambda$ and $u$ vary, so that $\widehat{\lambda}$ spans the entire simplex $\Delta^{d+1}$. We have 
$
\lambda_\chi=\widehat{\lambda}_{\chi_L}+\widehat{\lambda}_{\chi_R}
$.

\section{The model for the saddle-sink node}
\label{SectionFlowBox}

\subsection{The saddle-sink node}

Let $\|(x,y)\|:=\max\{|x|,|y|\}$. For $h\in\RR$ let $\phi_h:\RR^2\to[0,1]$ be a smooth function such that 
$$
\left\{
\begin{array}{cccc}
&\phi_h(x,y)=1&\text{ for }&\|(x,y-h)\|\leq 2\\
&0<\phi_h(x,y)<1&\text{ for }&2<\|(x,y-h)\|<3\\
&\phi_h(x,y)=0&\text{ for }&\|(x,y-h)\|\geq 3.
\end{array}
\right.
$$
Assume also the extra conditions
\begin{align}
&
\label{EquationConditionCutFunction}
\left\{
\begin{array}{l}
\partial_x\phi_h(x,y)\not=0
\quad\text{for}\quad
2<\|(x,y-h)\|<3
\quad\text{and}\quad
|y-h|<2.
\\
\partial_y\phi_h(x,y)=0
\quad\text{for}\quad
|y-h|<2.
\end{array}
\right.
\\
&
\label{EquationDecayPartialDerivatives}
\partial^n_x\partial^m_y\varphi_h(x,y)\to 0
\quad\text{as}\quad
\|(x,y-h)\|\to 3
\quad\text{for any}\quad n,n\in\NN.
\end{align}
Equation \eqref{EquationConditionCutFunction} is used in the proof of Lemma \ref{LemmaUpperBoundIntegral}, while Equation \eqref{EquationDecayPartialDerivatives} gives the smoothness of the flow defined by Equation \eqref{EquationFlowY}. Consider the constant vector fields over $\RR^2$ given by $V_0:=(0,1)$ and $H^{(s)}:=(-s,s)$, where $s>0$. 
Let $Z:\RR^2\to\RR^2$ be as in Equation~\eqref{EquationSaddleSinkNode}, that is $Z(x,y):=(-x,y^2-1)$. Consider the smooth vector field  
$$
W^{(s)}(x,y):=\big(1-\phi_{-4}-\phi_4\big)\cdot V_0(x,y) + \phi_{-4}\cdot Z(x,y+4)+ \phi_4\cdot H^{(s)}(x,y).
$$

\begin{definition}
\label{DefinitionTrajectoriesAtSaddle}
Let $\sigma_1:\RR\to\RR^2$ be the trajectory of $W^{(s)}$ with $\sigma_1(t)\to(0,-3)$ for $t\to+\infty$ and such that for any $t$ big enough we have $d\sigma_1(t)/dt=\big(a(t),0\big)$ with $a(t)>0$. In other words $\sigma_1$ is the incoming trajectory to the saddle $(0,-3)$ from the left. Let $\sigma_2:\RR\to\RR^2$ be the outgoing $W^{(s)}$-trajectory from the saddle $(0,-3)$ in the upward direction, that is $\sigma_2(t)\to(0,-3)$ for $t\to-\infty$ and $d\sigma_2(t)/dt=\big(0,b(t)\big)$ with $b(t)>0$ for any $t$ close enough to $-\infty$. Finally let $\sigma_3:\RR\to\RR^2$ the incoming $W^{(s)}$-trajectory to the saddle $(0,-3)$ from the right, that is $\sigma_3(t)\to(0,-3)$ for $t\to+\infty$ and $d\sigma_3(t)/dt=\big(c(t),0\big)$ with $c(t)<0$ for any $t$ big enough. 
\end{definition}

The right part of Figure~\ref{FigureZipperedRectanglesAndSaddleSink} represents these trajectories for the vector field $W$ in Equation~\eqref{EquationFlowBox} below. Write $W^{(s)}$ in components as 
$W^{(s)}=(W^{(s)}_1,W^{(s)}_2)$.

\begin{lemma}
\label{LemmaFlowBox}
The vector field $W^{(s)}$ is smooth, moreover the following holds.
\begin{enumerate}
\item
The only zeros of $W^{(s)}$ are a sink at $(0,-5)$ and a saddle at $(0,-3)$.
\item
We have $W^{(s)}(x,y)=(0,1)$ for any $(x,y)$ outside the box $(-3,3)\times(-7,7)$.
\item
There exists an unique $t\in\RR$ such that $\sigma_1(t)$ belongs to the line $\{y=-7\}$. For such $t$ we have
$$
\sigma_1(t)=(\widetilde{u},-7)\text{ with }-3<\widetilde{u}<-2\text{ and }d\sigma_1(t)/dt=(0,1).
$$
\item
There exists an unique $t\in\RR$ such that $\sigma_3(t)$ belongs to the line $\{y=-7\}$. For such $t$ we have
$$
\sigma_3(t)=(\widetilde{c},-7)\text{ with }2<\widetilde{c}<3\text{ and }d\sigma_3(t)/dt=(0,1).
$$
\item
Denoting $\Omega$ the closed disc delimited by $\sigma_1$, $\sigma_3$ and the line $\{y=-7\}$, for $(x,y)$ outside 
$\Omega$ we have
$$
W_2^{(s)}(x,y)>0.
$$
\item
There exists an unique $t\in\RR$ such that $\sigma_2(t)$ belongs to the line $\{y=7\}$. For such $t$ we have
$$
\sigma_2(t)=(\widetilde{v},7)\text{ with }-3<\widetilde{v}<-2\text{ and }d\sigma_2(t)/dt=(0,1).
$$
Moreover, if $\widetilde{u}$ is as in Point (3), then there exists an unique value of the parameter $s>0$ such that $\widetilde{v}=\widetilde{u}$.
\end{enumerate}
\end{lemma}

\begin{proof}
The points $(0,-5)$ and $(0,-3)$ are the only zeros of $(x,y)\mapsto Z(x,y+4)$ inside $[-2,2]\times[-6,-2]$. Outside of such region, $Z$ and $V_0$ are parallel only along the line $\{x=0\}$, but on such line $W^{(s)}$ vanishes only on $(0,-3)$ and $(0,-5)$. Thus there are no other points with $W^{(s)}(x,y)=0$ with $y\leq0$. For $y>0$, $W^{(s)}(x,y)$ is the convex combination of $V_0$ and $H^{(s)}$, which are never parallel. Point(1) follows. Point(2) is obvious from the definitions. In order to prove Point (3) observe first that $\sigma_1$ never touches the vertical line $\{x=-3\}$, because the convex combination of $Z(x,y+4)$ and $V_0$ has horizontal component which vanishes for $x\to-3$. Let $O_1$ be the set of $(x,y)$ with $2<\|(x,y+4)\|<3$. For $(x,y)\in O_1$ we have $0<\phi_{-4}(x,y)<1$ and 
$$
W^{(s)}_2(x,y)=1+\phi_{-4}(x,y)\big((y+4)^2-2\big).
$$
Let $O_2\subset O_1$ be the subset satisfying the extra condition $|y+4|\geq 1$. For $(x,y)\in O_2$ we have 
$W^{(s)}_2(x,y)>0$ because $\phi_{-4}(x,y)\big((y+4)^2-2\big)>-1$. In $O_1\setminus O_2$ we have $|y+4|<1$. 
Thus we have $W^{(s)}_2(x,y)\leq 0$ in the subregion $F\subset O_1\setminus O_2$ of those $(x,y)$ with
$$
\phi_{-4}(x,y)\geq\big(2-(y+4)^2\big)^{-1}.
$$
Reversing time corresponds to consider trajectories of $-W^{(s)}$. The reverse of $\sigma_1$ starts from $(0,-3)$ and reaches $(-2,-3)$ along an horizontal segment. After $(-2,-3)$, it turns downwards and enters $O_1\setminus O_2$, because $W^{(s)}_2(x,2)>0$ for $-3<x<-2$. The left boundary of $O_1\setminus O_2$ is the vertical line $\{x=-3\}$, and the reverse of $\sigma_1$ cannot cross it. The right boundary of $O_1\setminus O_2$ is the smooth curve of equation  
$$
\phi_{-4}(x,y)-\big(2-(y+4)^2\big)^{-1}=0,
$$ 
and the reverse of $\sigma_1$ cannot cross it because for $(x,y)$ in such curve we have 
$$
-W^{(s)}(x,y)=\phi_{-4}(x,y)(x,0),
$$
which points to the left. In other words such curve acts as a repeller for trajectories of $-W^{(s)}$ coming from the left. In conclusion, the reverse of $\sigma_1$ can leave $O_1\setminus  O_2$ only crossing the horizontal line 
$\{y=-5\}$. After that, being inside $O_1$, we have $-W_1^{(s)}<0$ and $-W_2^{(s)}<0$, thus the reverse of $\sigma_1$ continues moving downwards until it intersects transversally the line $\{y=-7\}$ at a point $(\widetilde{u},-7)$ with $-3<\widetilde{u}<-2$. Moreover 
$
\phi_4(\widetilde{u},-7)=0
$, 
thus $W^{(s)}(\widetilde{u},-7)=V_0$. Point (3) is proved. Point (4) follows by the same argument. The proof of Point (3) also shows that $W_2^{(s)}(x,y)>0$ for those $(x,y)$ outside $\Omega$ with $y<-1$. For $y\geq-1$ we have $W_2^{(s)}(x,y)>0$ trivially because $W^{(s)}$ is the convex combination of $V_0$ and $H^{(s)}$. Point (5) is proved. The proof of the first part of Point (6) follows directly from Point (5).  In order to prove the last part, consider the continuous map $s\mapsto \widetilde{v}(s)$, where $\widetilde{v}$ is the horizontal coordinate of the intersection between $\sigma_2$ and the line $\{y=7\}$. We have $\lim_{s\to0}\widetilde{v}(s)=-2$ and $\lim_{s\to+\infty}\widetilde{v}(s)=-3$, thus the equality 
$
\widetilde{v}(s)=\widetilde{u}
$ 
follows from the Intermediate Value Theorem. The Lemma is proved.
\end{proof}

\subsection{The hitting map}

Let $s$ be the parameter determined in Point (6) of Lemma~\ref{LemmaFlowBox}. For such value we define the vector field $W$ setting
\begin{equation}
\label{EquationFlowBox}
W:=W^{(s)}.
\end{equation}
Write it in components as $W=(W_1,W_2)$, then let 
$
\big(t,(x,y)\big)\mapsto \Phi_W\big(t,(x,y)\big)
$ 
be its integral flow, that is $d/dt\Phi_W(t,x,y)=W\big(\Phi_W(t,x,y)\big)$ for any $t,x,y$ where the trajectory is defined. Since $W$ is smooth, then $\Phi_W$ is smooth too. For any $x\in(-3,\widetilde{u})\cup(\widetilde{c},3)$ let 
$\sigma_x$ be the positive trajectory of $W$ passing through $(x,-7)$, parametrized by 
$$
\sigma_x(t):=\Phi_W(t,x,-7).
$$

\begin{lemma}
\label{LemmaReturnMapFlowBox}
For any $x\in(-3,\widetilde{u})\cup(\widetilde{c},3)$ we have $t(x)>0$ defined by 
$$
t(x):=\min\{t>0:\sigma_x(t)\in(-3,3)\times\{7\}\}
$$
Moreover we have $\sigma_x\big(t(x)\big)=\big(g_W(x),7\big)$, where $g_W:[-3,3]\to [-3,3]$ is a continuous map, strictly increasing on 
$[-3,\widetilde{u})\cup[\widetilde{c},3)$ with $g_W(x)=\widetilde{u}$ for any $x\in[\widetilde{u},\widetilde{c}]$. 
\end{lemma}

\begin{proof}
By uniqueness of the solution of a smooth ODE, for $x\in(-3,\widetilde{u})\cup(\widetilde{c},3)$ the path $\sigma_x$ never crosses neither the vertical segments in the boundary of $[-3,3]\times[-7,7]$ nor the union 
$
\sigma_1\cup\sigma_2\cup\sigma_3
$. 
The vertical component of $d\sigma_x(t)/dt$ has a positive lower bound uniformly in $t$ (but not in $x$) according to Point (5) of Lemma~\ref{LemmaFlowBox}, therefore there exists $t(x)>0$ as in the statement and a strictly increasing map $g_W:(-3,\widetilde{u})\cup(\widetilde{c},3)\to (-3,3)$ with 
$\sigma_x\big(t(x)\big)=\big(g_W(x),7\big)$. Moreover $g_W$ is continuous by continuous dependence on the initial conditions. Considering the trajectories of $-W$ one can see that $g_W$ is surjective onto 
$(-3,3)\setminus\{\widetilde{u}\}$. The definitions of $\sigma_1$ and $\sigma_3$, together with Point (6) in Lemma Lemma~\ref{LemmaFlowBox} imply $g_W(-3,\widetilde{u})=(-3,\widetilde{u})$ and 
$g_W(\widetilde{c},3)=(\widetilde{u},3)$. Such map has an unique continuous monotone extension $g_W:I^t_\chi\to I^t_\chi$ as in the statement. The Lemma is proved.
\end{proof}

For $i=1,2$ let $\rho_i:\RR^2\to\RR$ be defined by $\rho_1(x,y):=x$ and  $\rho_2(x,y):=y$. 
The smooth function $H(t,x):=\rho_2\big(\Phi_W(t,x,-7)\big)$ gives the vertical coordinate at time $t$ of the trajectories of points $(x,-7)$ with $x\in(-3,\widetilde{u})\cup(\widetilde{u},3)$. 
We have $H\big(t(x),x\big)=7$.
Moreover Point (5) in Lemma \ref{LemmaFlowBox} gives 
$$
\partial_t H(t,x)=W_2\big(\Phi_W(t(x),x,0)\big)>0,
$$
therefore $x\mapsto t(x)$ is a smooth\footnote{The Implicit Function Theorem proves that $x\mapsto t(x)$ is continuously differentiable. Smoothness follows a posteriori because the derivative of $t(x)$ is a ratio of smooth functions.} function by the Implicit Function Theorem. The map $g_W$ is smooth on 
$(-3,\widetilde{u})\cup(\widetilde{c},3)$ because $g_W(x)=\rho_1\big(\Phi_W(t(x),x,-7)\big)$ for $x$ in such union. Point (5) of Lemma \ref{LemmaFlowBox} implies also that for any $x\in(-3,\widetilde{u})\cup(\widetilde{c},3)$ and any $t\in[0,t(x)]$ we have 
$$
\frac{d\rho_2\big(\sigma_x(t)\big)}{dt}=W_2\big(\sigma_x(t)\big)>0.
$$
Thus the map $t\mapsto y_x(t):=\rho_2\big(\sigma_x(t)\big)$ admits a smooth inverse $y\mapsto t_x(y)$, so that 
$\sigma_x$ can be parametrized by the height $y$ that it reaches. Therefore we have a well defined map 
\begin{equation}
\label{EquationReparametrizationByHight}
s_W:\big((-3,\widetilde{u})\cup(\widetilde{c},3)\big)\times [-7,7]\to [-3,3],
\qquad 
s_W(x,y):=\rho_1\big(\sigma_x\cap([-3,3]\times\{y\})\big),
\end{equation}
that is $s_W(x,y)$ is the horizontal component of the trajectory $\sigma_x$ when it reaches height $y$. We have
$g_W(x)=s_W(x,7)$ for 
$
x\in(-3,\widetilde{u})\cup(\widetilde{c},3)
$. 
Since $W$ is smooth with $W_2>0$ outside $\Omega$, then we get a smooth function $F:\Omega\to\RR$ setting 
$$
F(x,y):=\frac{W_1(x,y)}{W_2(x,y)}.
$$

\begin{lemma}
\label{LemmaFormulaDerivativeFirstReturn}
The map $(x,y)\mapsto s_W(x,y)$ is smooth. Moreover, for 
$x\in(-3,\widetilde{u})\cup(\widetilde{c},3)$ we have 
$$
\frac{dg_W(x)}{dx}=\exp\bigg(\int_{-7}^{7}\partial_x F\big(s_W(x,y),y\big)dy\bigg).
$$
\end{lemma}

\begin{proof}
For any fixed $x\in(-3,\widetilde{u})\cup(\widetilde{c},3)$ we have 
$
s_W(x,y)=\rho_1\circ\sigma_x\big(t_x(y)\big)
$. 
Therefore
\begin{align*}
\frac{\partial s_W(x,y)}{\partial y}
&
=
\frac{d\rho_1\circ\sigma_x(t)}{dt}\bigg|_{t=t_x(y)}
\frac{dt_x(y)}{dy}
=
\frac{d\rho_1\circ\sigma_x(t)}{dt}\bigg|_{t=t_x(y)}
\bigg(\frac{dy_x\big(t_x(y)\big)}{dt}\bigg)^{-1}
\\
&
=
\frac{W_1\big(\sigma_x(t_x(y))\big)}{W_2\big(\sigma_x(t_x(y))\big)}
=F\big(s_W(x,y),y\big)
\quad\text{ with }\quad s_W(x,0)=x.
\end{align*}
Thus $y\mapsto s_W(x,y)$ is the solution of a Cauchy problem with $F$ smooth. The smooth dependence on initial conditions implies that $(x,y)\mapsto s_W(x,y)$ is smooth too. Setting $V(x,y):=\partial_x s_W(x,y)$ we get
\begin{align*}
\partial_y V(x,y)
&
=\partial_y \partial_x s_W(x,y)=\partial_x \partial_y s_W(x,y)
\\
&
=\partial_x\Big(F\big(s_W(x,y),y\big)\Big)=\partial_x F\big(s_W(x,y),y\big)\partial_x s_W(x,y)
=\partial_x F\big(s_W(x,y),y\big)V(x,y).
\end{align*}
Therefore 
$$
\log\big(V(x,7)\big)-\log\big(V(x,-7)\big)=\int_{-7}^{7}\partial_x F\big(s(x,y),y\big)dy
$$
Since $s_W(x,-7)=x$ then $V(x,-7)=1$. Moreover $s_W(x,7)=g_W(x)$, so $V(x,7)=dg_W(x)/dx$ and the Lemma is proved.
\end{proof}

\begin{lemma}
\label{LemmaUpperBoundIntegral}
There exists a constant $C>0$ such that for any $x\in(-3,\widetilde{u})\cup(\widetilde{c},3)$ we have
$$
\int_{-7}^{7}\partial_x F\big(s_W(x,y),y\big)dy\leq C.
$$
Moreover 
$
\lim_{x\to \widetilde{u}_-}\int_{-7}^{7}\partial_x F\big(s(x,y),y\big)dy=
\lim_{x\to \widetilde{c}_+}\int_{-7}^{7}\partial_x F\big(s(x,y),y\big)dy=-\infty
$. 
In particular $g_W$ is continuously differentiable on $[-3,3]$.
\end{lemma}

\begin{proof}
For simplicity write $P(x,y):=\partial_x F(x,y)$ and $s(x,y)$ instead of $s_W(x,y)$. Observe that for $-7\leq y\leq -1$ we have
$W_2\big(s(x,y),y\big)=1+\phi_{-4}\big((y+4)^2-2\big)$ and $W_1\big(s(x,y),y\big)=-x\phi_{-4}$, therefore
\begin{equation}
\label{EquationLemmaUpperBoundIntegral}
P\big(s(x,y),y\big)
=
\frac
{-\phi_{-4}W_2\big(s(x,y),y\big)-s(x,y)\partial_x\phi_{-4}}
{\big(W_2\big(s(x,y),y\big)\big)^2}
\leq
\frac
{8+3\|\partial_x\phi_{-4}\|_\infty}
{\big(1+\phi_{-4}\big((y+4)^2-2\big)\big)^2}.
\end{equation}
Let $s$ be the parameter determined in Point (6) of Lemma~\ref{LemmaFlowBox}.
For $1<y<7$ we have $W_2(x,y)=1-\phi_{4}(s-1)$ and $W_1(x,y)=-\phi_{4}s$, thus 
$$
P\big(s(x,y),y\big)=
\frac{-s\partial_x\phi_4}{\big(1+\phi_{4}(s-1)\big)^2}\leq
\frac{s\|\partial_x\phi_4\|_\infty}{(\min\{1,s\})^2}.
$$ 
For $-1\leq y\leq 1$ we have $W_2(x,y)=1$ and $W_1(x,y)=0$, so $P\big(s(x,y),y\big)=0$. 
Finally for $-4+\sqrt{2}<y<-1$ we have $(y+4)^2-2>0$, and recalling Equation \eqref{EquationLemmaUpperBoundIntegral} we have 
$P\big(s(x,y),y\big)\leq 8+3\|\partial_x\phi_{-4}\|_\infty$. 
Therefore for any $w\geq-4+\sqrt{2}$ (below we use both $w=-4+\sqrt{2}$ and $w=-2$) we have  
$$
\int_{w}^{7}P\big(s(x,y),y\big)dy
\leq 
(11-\sqrt{2})\max\Big\{8+3\|\partial_x\phi_{-4}\|_\infty,\frac{s}{\min\{1,s\}^2}\Big\}.
$$

From now on, assume $-3<x<\widetilde{u}$. Fix $\delta>0$ such that $\phi_{-4}(x,y)\geq 1/2$ whenever $\|(x,y+4)\|\leq 2+\delta$, that is $(x,y)$ belongs to the square $S_\delta$ of center $(0,-4)$ and a vertex in $(-2-\delta,-6-\delta)$. Consider the compact set 
$
K_\delta:=[-3,-2-\delta]\times[-7,-4+\sqrt{2}]
$. 
According to Point (5) of Lemma \ref{LemmaFlowBox} define $\epsilon>0$ by
$$
\epsilon:=\min\Big\{W_2\big(s(x,y),y\big):\big(s(x,y),y\big)\in K_\delta\Big\}.
$$
Consider $x\in(-3,\widetilde{u})$ and assume that $\big(s(x,y),y\big)\in K_\delta$ for any $y\in[-7,-4+\sqrt{2}]$. In this case the Lemma follows because Equation \eqref{EquationLemmaUpperBoundIntegral} gives
$$
\int_{-7}^{-4+\sqrt{2}}P\big(s(x,y),y\big)dy
\leq 
(3+\sqrt{2})\frac{8+3\|\partial_x\phi_{-4}\|_\infty}{\epsilon^2}.
$$
Otherwise there exists $y(x)\in(-3,-4+\sqrt{2})$ such that 
$s\big(x,y(x)\big)=-2-\delta$, that is $\sigma_x$ enters in $S_\delta$. In this case, choosing $\delta$ small enough, there exists $z(x)$ with $y(x)<z(x)<-2$ such that $s\big(x,z(x)\big)=-2$, that is $\sigma_x$ enters in the square where $W(x,y)=Z(x,y+4)$. We have
$$
\begin{array}{c}
\int_{-7}^{-2}P\big(s(x,y),y\big)dy=A_x+B_x+C_x
\qquad\text{where}
\\
A_x:=\int_{-7}^{y(x)}P\big(s(x,y),y\big)dy,
\quad
B_x:=\int_{y(x)}^{z(x)}P\big(s(x,y),y\big)dy,
\quad
C_x:=\int_{z(x)}^{-2}P\big(s(x,y),y\big)dy.
\end{array}
$$
We have clearly 
$
A_x\leq (3+\sqrt{2})3\|\partial_x\phi_{-4}\|_\infty\epsilon^{-2}
$. 
Moreover for $y(x)\leq y<z(x)$ we have $\phi_{-4}\geq 1/2$ and $y+4\geq 1$. Therefore setting 
$M:=(8+3\|\partial_x\phi_{-4}\|_\infty)$, Equation \eqref{EquationLemmaUpperBoundIntegral} gives
\begin{align*}
B_x
&
\leq M\int_{y(x)}^{z(x)}
\frac{1}{\big(1+\phi_{-4}\big((y+4)^2-2\big)\big)^2}dy
&
\leq M\int_{y(x)}^{z(x)}
\frac{2(y+4)\phi_{-4}}{\big(1+\phi_{-4}\big((y+4)^2-2\big)\big)^2}dy
\\
&
=-M\int_{y(x)}^{z(x)}
\partial_y\bigg(\frac{1}{1+\phi_{-4}\big((y+4)^2-2\big)}\bigg)dy
&
\leq M\bigg(\frac{1}{\epsilon}-\frac{1}{1+\big((z(x)+4)^2-2\big)}\bigg)
\\
&
=M\bigg(\frac{1}{\epsilon}+\frac{1}{5+z(x)}\cdot\frac{1}{-3-z(x)}\bigg)
&
\leq M\bigg(\frac{1}{\epsilon}+\frac{-1/3}{3+z(x)}\bigg).
\end{align*}
In the second line, the equality uses the condition $\partial_y\phi_{-4}=0$ in Equation \eqref{EquationConditionCutFunction} and the inequality follows because 
$\phi_{-4}\big(-2,z(x)\big)=1$ and because $(-2-\delta,y(x))\in K_\delta$. 
Finally for $y(x)\leq y\leq -2$ we have 
$P\big(s(x,y),y\big)=\big(1-(y+4)^2\big)^{-1}$, therefore 
\begin{align*}
C_x
&
=\int_{z(x)}^{-2}\frac{1}{1-(y+4)^2}dy
&
=\int_{z(x)}^{-2}\frac{1}{5+y}\cdot\frac{1}{-3-y}dy
&
\leq\int_{z(x)}^{-2}\frac{1/3}{-3-y}dy
&
=\frac{\ln\big(3+z(x)\big)}{3}.
\end{align*}
The Lemma follows for $-3<x<\widetilde{u}$ because both $B_x$ and $C_x$ are uniformly bounded from above and moreover for $x\to\widetilde{u}_-$ we have $z(x)\to -3_+$ so that $B_x\to-\infty$ and $C_x\to-\infty$. For $\widetilde{c}<x<3$ the symmetric argument applies to $A_x,B_x,C_x$ because the restriction of the vector field $W$ to the region where $y<0$ is symmetric with respect to the reflection $(x,y)\mapsto(-x,y)$. The Lemma is proved.
\end{proof}

\begin{corollary}
The map $g_W$ in Lemma \ref{LemmaReturnMapFlowBox} is smooth on $[-3,3]$.
\end{corollary}

\begin{proof}
We only need to prove smoothness at $\widetilde{u}$ and $\widetilde{c}$, where Lemma \ref{LemmaUpperBoundIntegral} proves the continuous differentiability. The estimates for $A_x,B_x,C_x$ in the proof of Lemma \ref{LemmaUpperBoundIntegral} actually prove that there exists $a>0$ such that for $x\to\widetilde{u}_-$ we have 
$$
\ln\bigg(\frac{dg_W(x)}{dx}\bigg)=
\int_{-7}^{7}G\big(s_W(x,y),y\big)dy
\leq \frac{-a}{3+z(x)}
\leq \frac{-a}{\widetilde{u}-x}\to-\infty,
$$
where $|3+z(x)|<|\widetilde{u}-x|$ because trajectories get "compressed" onto $\sigma_1$ near the point $(-2,-3)$. This is enough to prove that all the derivatives of $g_W$ converge to zero as $x\to\widetilde{u}_-$. The same holds for $\widetilde{c}$ by the symmetric argument. Details are left to the reader.
\end{proof}

\section{Construction of Cherry flows}
\label{SectionConstructionCherryFlows}

\subsection{Setting}
\label{SectionSettingConstructionCherryFlows}

Fix a connected component $\cH$ of the moduli space of translation surfaces, with ramification indices $k_1,\dots,k_m$. Let $\cA$ be an alphabet with $d=m+1+\sum_{j=1}^mk_j$ letters and $\pi$ be an admissible combinatorial datum over $\cA$ representing surfaces in $\cH$ (see \S~\ref{SectionConnectedComponentsStrata}). 
As in \S~\ref{SectionMarkingPointTranslationSurface}, fix $\chi\in\cA$ and set 
$
\widehat{\cA}:=\cA\setminus\{\chi\}\cup\{\chi_L,\chi_R\}
$, 
then let $\widehat{\pi}$ be the combinatorial datum over $\widehat{\cA}$ obtained from $\pi$ marking the letter 
$\chi$. Assume that the Rauzy class $\widehat{\cR}$ of $\widehat{\pi}$ contains a cyclic element as in Definition~\ref{DefinitionCyclicCombinatorialDatum}. 
Fix a suspension datum $\tau$ for $\pi$, then let $\widehat{\tau}$ be as in Equation~\eqref{EquationLengthDataMarkedSurface}. For $\lambda\in\Delta^d$ and $u$ as in Equation~\eqref{EquationParameterU}, that is 
$
u^t_\chi<u<u^t_\chi+\lambda_\chi
$, 
we define a vector field $Y=Y_{(\lambda,u)}$ on the surface $X$ associated to the data $(\pi,\lambda,\tau)$ as in  \S~\ref{SectionZipperedRectangles}.

\subsection{A family of smooth vector fields}
\label{SectionDefinitionSmoothVectorField}

Let $p_1,\dots,p_m$ be the conical singularities of $X$, the angle at any $p_j$ being $2(k_j+1)\pi$. 
Define $\lambda_{\chi_L}=\lambda_{\chi_L}(\lambda,u)$ and $\lambda_{\chi_R}=\lambda_{\chi_R}(\lambda,u)$ as in 
Equation~\eqref{EquationLengthDataMarkedSurface}. Define $r_0=r_0(\lambda,u)$ by 
\begin{equation}
\label{EquationRadiusNeighbourhoodSingularities}
r_0:=\frac{1}{10}\cdot\min_{\alpha\in\widehat{\cA}}\widehat{\lambda}_\alpha.
\end{equation}
For $1\leq j\leq m$ let 
$
\rho_j:\big(U_j,p_j\big)\to \big(B(0,r_j),0\big)
$ 
be the ramified covering in Equation~\eqref{EquationDefinitionTranslationSurface(Ramification)}. For a proper choice of the suspension datum $\tau$, we have
$$
r_0\leq \min_{j=1,\dots,m}r_j.
$$

\subsubsection{Slowing down the vertical field at singularities}

Let 
$f:[0,r_0)\to[0,1]$ be a smooth function such that 
$$
\left\{
\begin{array}{l}
f^{(n)}(0)=0\quad\textrm{ for any }\quad n\in\NN\\
0<f(r)<1\quad\textrm{ for any }\quad 0<r<r_0/2\\
f(r)=1\quad\textrm{ for any }\quad r_0/2\leq r<r_0,
\end{array}
\right.
$$
where $f^{(n)}$ denotes the $n$-th derivative of $f$. The conditions above imply that we have a smooth function $F:B(0,r_0)\to\RR$ defined in polar coordinates by $F(r,\theta):=f(r)$. 
With an abuse of notation let $U_j$ be the pre-image of $B(0,r_0)$ under the ramified covering $\rho_j$ in 
Equation~\eqref{EquationDefinitionTranslationSurface(Ramification)}. For $1\leq j\leq m$, setting $G_j:=F\circ\rho_j$, we define a smooth function $G_j:U_j\to[0,1]$. Such function extends to a smooth function $G_j:X\to[0,1]$ such that 
$$
\left\{
\begin{array}{l}
G_j(p_j)=0\\
0<G_j(p)\leq1\textrm{ for any }p\in U_j\setminus \{p_j\}\\
G_j(p)=1\textrm{ for any }p\in X\setminus K_j,
\end{array}
\right.
$$
where $K_j\subset U_j$ is a compact set with $p_j\in K_j$. Let $V_0$ be the vertical flow determined by the data $(\pi,\lambda,\tau)$ and consider the smooth complete vector field
\begin{equation}
\label{EquationCompleteVerticalVectorField}
V:=G_1\cdot\dots\cdot G_m\cdot V_0.
\end{equation}

\subsubsection{Adding the saddle-sink node}
\label{SectionAddingSaddleSinkNode}

Recall \S~\ref{SectionZipperedRectangles} and for $\alpha\in\cA$ let $R_\alpha$ be the rectangles inside the surface $X=X(\pi,\lambda,\tau)$, each $R_\alpha$ having height $h_\alpha$. In the coordinates $(x,y)$ of the rectangle $R_\chi$, we define a point 
$(x_0,y_0)\in\interior(R_\chi)$ setting 
\begin{equation}
\label{EquationCenterSaddleSinkNode}
y_0:=\frac{h_\chi}{2}\quad\text{ and }\quad 
x_0:=u-\frac{r_0\widetilde{u}}{7},
\end{equation}
where $u$ is as in Equation~\eqref{EquationParameterU}, $\widetilde{u}$ is as in Point (3) of Lemma~\ref{LemmaFlowBox} and $r_0$ is defined in Equation~\eqref{EquationRadiusNeighbourhoodSingularities}. Consider the map
\begin{equation}
\label{EquationBijectionToFlowBox}
(x,y)\mapsto \varphi(x,y):=\bigg(\frac{7(x-x_0)}{r_0},\frac{7(y-y_0)}{r_0}\bigg)
\end{equation}
and let $Q$ be the preimage of the box $[-3,3]\times[-7,7]$, that is 
\begin{equation}
\label{EquationBoxQ}
Q:=\varphi^{-1}\big([-3,3]\times[-7,7]\big)\subset R_\chi.
\end{equation}
By the property of $r_0$, the box $Q$ is contained in $R_\chi$ and disjoint from the neighbourhoods $U_j$ with $j=1,\dots,m$ where the ramified covering 
$
\rho_j:\big(U_j,p_j\big)\to \big(B(0,r_j),0\big)
$ 
in Equation~\eqref{EquationDefinitionTranslationSurface(Ramification)} are defined. Let $W$ be the vector field in Equation~\eqref{EquationFlowBox} and consider 
$$
\widetilde{W}(x,y):=W\bigg(\frac{7(x-x_0)}{r_0},\frac{7(y-y_0)}{r_0}\bigg).
$$
Point (2) of Lemma~\ref{LemmaFlowBox} implies that $\widetilde{W}$ is a smooth vector field with 
$\widetilde{W}(x,y)=(0,1)$ for $(x,y)$ outside $Q$. Moreover according to Equation \eqref{EquationDecayPartialDerivatives} all partial derivatives of $\widetilde{W}(x,y)$ vanish as $(x,y)\to\partial Q$. Finally $Q$ is disjoint from $U_j$, $j=1,\dots,m$. Therefore we have a smooth vector field $Y_{(\lambda,u)}$ over $X$ defined by 
\begin{equation}
\label{EquationFlowY}
Y_{(\lambda,u)}(p):=
\left\{
\begin{array}{ccc}
\widetilde{W}(p) & \text{if} & p\in Q\\
V(p) & \text{if} & p\in X\setminus Q,
\end{array}
\right.
\end{equation}
where $V$ is defined in Equation~\eqref{EquationCompleteVerticalVectorField} and where, abusing the notation, we identify $(x,y)\in \interior(R_\chi)$ with the point $p\in X$ that they define. We have proved the following Lemma.

\begin{lemma}
\label{LemmaZerosVectorField}
The vector field $Y_{(\lambda,u)}$ in Equation~\eqref{EquationFlowY} is smooth, with a saddle-sink node, an isolated zero of index $-k_j$ at any conical singularity $p_j$ and no other zeros. 
\end{lemma}

\begin{figure}
\begin{tikzpicture}[scale=0.7]

\tikzset
{->-/.style={decoration={markings,mark=at position .5 with {\arrow{>}}},postaction={decorate}}}

\clip(-1,-6.1) rectangle (12,6.1);



\draw[-] (0,0) -- (10,0);
\node [circle,fill,inner sep=1pt] at (0,0) {};


\draw[->-,thick,dashed] (0,0) -- (2,1);
\node [circle,fill,inner sep=1pt] at (2,1) {};

\draw[->-,thick,dashed] (2,1) -- (3,4) node[pos=0.3,below,black] {};
\node [circle,fill,inner sep=1pt] at (3,4) {};

\draw[->-,thick,dashed] (3,4) -- (7,2) node[pos=0.7,above] {$\zeta_C$};
\node [circle,fill,inner sep=1pt] at (7,2) {};

\draw[->-,thin,dashed] (3,4) -- (4.5,2) node[pos=0.5,below] {$\zeta_{C_L}$};
\node [circle,fill,inner sep=1pt] at (4.5,2) {};
\draw[->-,thin,dashed] (4.5,2) -- (7,2) node[pos=0.7,below] {$\zeta_{C_R}$};

\draw[->-,thick,dashed] (7,2) -- (10,-2);
\node [circle,fill,inner sep=1pt] at (10,-2) {};


\draw[->-,thick,dashed] (0,0) -- (4,-2);
\node [circle,fill,inner sep=1pt] at (4,-2) {};

\draw[->-,thin,dashed] (0,0) -- (1.5,-2);
\node [circle,fill,inner sep=1pt] at (1.5,-2) {};
\draw[->-,thin,dashed] (1.5,-2) -- (4,-2);

\draw[->-,thick,dashed] (4,-2) -- (6,-1);
\node [circle,fill,inner sep=1pt] at (6,-1) {};

\draw[->-,thick,dashed] (6,-1) -- (9,-5);
\node [circle,fill,inner sep=1pt] at (9,-5) {};

\draw[->-,thick,dashed] (9,-5) -- (10,-2) node[pos=0.3,below,black] {};
\node [circle,fill,inner sep=1pt] at (10,-2) {};


\draw[-,very thin,dashed] (0,0) rectangle (2,2);
\node[] at (0.5,5) {$A$};
\draw[-,thick] (0,0) -- (0,2);
\draw[-,thick] (0,2) -- (2,2);
\draw[-,thick] (2,1) -- (2,2);

\draw[-,very thin,dashed] (2,0) rectangle (3,6);
\node[] at (2.5,5) {$B$};
\draw[-,thick] (2,1) -- (2,6);
\draw[-,thick] (2,6) -- (3,6);
\draw[-,thick] (3,4) -- (3,6);

\draw[-,very thin,dashed] (3,0) rectangle (7,4);
\node[] at (5,5) {$C$};
\draw[-,thick] (3,0) -- (3,0.5);
\node[] at (3.5,0.5) {$u^t_C$};
\draw[-,thick] (4.5,0) -- (4.5,0.5);
\node[] at (5.6,0.5) {$u=u^t_{C_R}$};
\draw[-,thick] (3,4) -- (7,4);
\draw[-,thick] (7,2) -- (7,4);

\draw[-,very thin,dashed] (7,0) rectangle (10,3);
\node[] at (8.5,5) {$D$};
\draw[-,thick] (7,2) -- (7,3);
\draw[-,thick] (7,3) -- (10,3);
\draw[-,thick] (10,0) -- (10,3);


\draw[-,very thin,dashed] (0,-4) rectangle (4,0);
\node[] at (2.5,-5) {$C$};
\draw[-,thick] (0,-4) -- (0,0);
\draw[-,thick] (0,-4) -- (4,-4);
\draw[-,thick] (4,-4) -- (4,-2);

\draw[-,very thin,dashed] (4,-2) rectangle (6,0);
\node[] at (5,-5) {$A$};
\draw[-,thick] (4,-2) -- (6,-2);
\draw[-,thick] (6,-2) -- (6,-1);

\draw[-,very thin,dashed] (6,-3) rectangle (9,0);
\node[] at (7.5,-5) {$D$};
\draw[-,thick] (6,-3) -- (6,-1);
\draw[-,thick] (6,-3) -- (9,-3);

\draw[-,very thin,dashed] (9,-6) rectangle (10,0);
\node[] at (9.5,-5) {$B$};
\draw[-,thick] (9,-6) -- (9,-5);
\draw[-,thick] (9,-6) -- (10,-6);
\draw[-,thick] (10,-6) -- (10,-2);
\end{tikzpicture}
{\begin{tikzpicture}[scale=0.037]

\tikzset
{->-/.style={decoration={markings,mark=at position .7 with {\arrow{>}}},postaction={decorate}}}



\draw[-] (-20,-20) -- (170,-20) {};
\draw[-] (-20,220) -- (170,220) {};
\draw[-] (170,-20) -- (170,220) {};
\draw[-] (-20,-20) -- (-20,220) {};

\draw[-,very thick] (-20,-20) -- (-20,-15);
\node[] at (-10,-10) {$u^t_C$};


\draw[ dashed] (40,25) rectangle (110,175);
\coordinate [label=right :$Q$]  (x_1) at (110,150);
\draw[-, very thick] (75,78) -- (75,105) {};
\draw[thick,dotted] (46,-20) -- (46,220) {};

\node[] at (75,227) {$R_C$};
\node[] at (110,-15) {$c$};
\node[] at (52,-15) {$u$};

\draw[-, very thick] (46,-20)-- (46,25) {};
\draw[-, very thick] (46,175)-- (46,220) {};
\draw[-, very thick] (104,-20)-- (104,25) {};


\draw[pattern=north west lines, pattern color=gray!80, very thin] (55,120) rectangle (95,160);

\draw[-, very thick] (75,105) .. controls (75,110) and (73,117)   ..  (70,120) {};
\draw[-, very thick] (70,120) --  (55,135) {};

\draw[->, very thick] (55,135) .. controls (46,146)   .. (46,175) {};

\coordinate [label=right :$\sigma_2$]  (x_1) at (75,100);


\draw[fill=gray!25,very thin] (55,40) rectangle (95,80);


\draw[-, very thick] (46,25) .. controls (46,70)  .. (55,70) {};
\draw[->, very thick] (55,70)-- (74,70) {};
\coordinate [label=right :$\sigma_1$]  (x_1) at (45,30);

\draw[-, very thick] (104,25) .. controls (104,70)  .. (95,70) {};
\draw[->, very thick] (95,70)-- (76,70) {};
\coordinate [label=left :$\sigma_3$]  (x_1) at (105,30);

\draw[->, very thick] (75,70)-- (75,80) {};
\draw[->, very thick] (75,70)-- (75,60) {};


\draw[->, very thick] (75,58)-- (75,51) {};
\draw[->, very thick] (75,42)-- (75,49) {};
\draw[->, very thick] (67,50)-- (75,50) {};
\draw[->, very thick] (83,50)-- (76,50) {};

\end{tikzpicture}}
\caption{On the left, a surface $X=X(\pi,\lambda,\tau)$ with $\cA=\{A,B,C,D\}$. 
Marking the point $(u,\tau_A+\tau_B+\tau_C)$ inside the rectangle $R_C$ produces data 
$
(\widehat{\pi},\widehat{\lambda},\widehat{\tau})
$ 
over $\widehat{\cA}=\{A,B,C_L,C_R,D\}$, which represent $X$ as marked surface. 
On the right, the trajectories of $Y_{(\lambda,u)}$ inside $R_C$. The dashed rectangle represents $Q$ 
in Equation~\eqref{EquationBoxQ}, where $Y_{(\lambda,u)}=\widetilde{W}$. Outside $Q$ we have $Y_{(\lambda,u)}=V$. Inside the upper square $Y_{(\lambda,u)}=(-s,s)$ with $s>0$. Inside the lower square $Y_{(\lambda,u)}=Z$ with $Z$ as in Equation \eqref{EquationSaddleSinkNode}, modulo a change of variable. Both $\sigma_1$ and $\sigma_2$ cross the boundary of $Q$ at points with identical horizontal component, thus the same is true for their intersection with the boundary of $R_C$.}
\label{FigureZipperedRectanglesAndSaddleSink}
\end{figure}
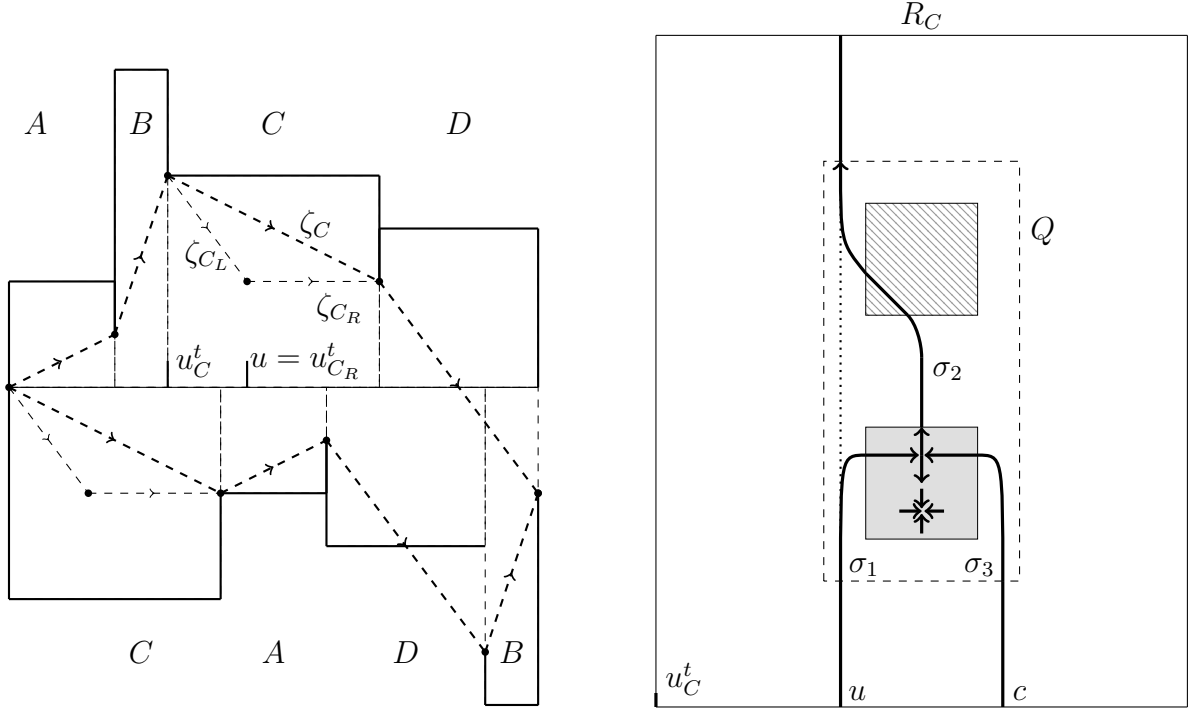

\subsection{The flow inside the rectangle $R_\chi$}
\label{SectionFlowInsideRectangle}

Fix $(\lambda,u)$ and let $Y_{(\lambda,u)}$ be the flow in Equation \eqref{EquationFlowY}. Set for simplicity $Y:=Y_{(\lambda,u)}$, then let 
$
\big(t,(x,y)\big)\mapsto \Phi_Y\big(t,(x,y)\big)
$ 
be its (smooth) integral flow. Let $\varphi$ be the map in Equation \eqref{EquationBijectionToFlowBox}, write it in components as $\varphi=(\varphi_1,\varphi_2)$, and recall that $Y=\widetilde{W}:=W\circ\varphi$ on the box $Q$ in Equation \eqref{EquationBoxQ}. Let $\Phi_W$ be the flow of $W$ in Equation \eqref{EquationFlowBox}. The map $\varphi$ is an homothethy with $D\varphi=\kappa\id$ and $\kappa=7/r_0$, thus for $(x,y)$ in $Q$ we have\footnote{Note that the speed of $Y$-trajectories inside $Q$ is 
$
\Big|d/dt\varphi^{-1}\Big(\Phi_W\big(\varphi(x,y),\kappa t\big)\Big)\Big|
=|D\varphi|^{-1}\cdot|W\circ\varphi|\cdot\kappa=|W|.
$}
\begin{equation}
\label{EquationConjugationFlows}
\varphi\big(\Phi_Y(x,y,t)\big)=\Phi_W\big(\varphi(x,y),\kappa t\big).
\end{equation}
Let $\sigma_1$, $\sigma_2$, $\sigma_3$ be the trajectories of $W$ in Definition~\ref{DefinitionTrajectoriesAtSaddle}. 
Abusing the notation, we call $\sigma_1$, $\sigma_2$ and $\sigma_3$ the corresponding trajectories of 
$Y$ inside the rectangle $R_\chi$, which are represented in the right part of Figure~\ref{FigureZipperedRectanglesAndSaddleSink}. Points (3) and (6) of Lemma~\ref{LemmaFlowBox} imply 
$$
\left\{
\begin{array}{l}
\sigma_1\cap\partial Q=\{(u,y_0- r_0)\}
\text{ and }
\widetilde{W}(u,y_0- r_0)=W(\widetilde{u},- 7)=(0,1)
\\
\sigma_2\cap\partial Q=\{(u,y_0+ r_0)\}
\text{ and }
\widetilde{W}(u,y_0+ r_0)=W(\widetilde{u},+ 7)=(0,1).
\end{array}
\right.
$$
The positive $V$-trajectory of $(u,0)\in I^t_\chi\times\{0\}$ meets $\partial Q$ in $(u,y_0- r_0)$, while the negative 
$V$-trajectory of $(u,h_\chi)\in I^t_\chi\times\{h_\chi\}$ meets $\partial Q$ in $(u,y_0+r_0)$. Since 
$Y=V$ outside $Q$, then 
\begin{equation}
\label{EquationMatchingOrbitsSigma1and2}
\sigma_1\cap I^t_\chi\times\{0\}=(u,0)
\quad\text{ and }\quad
\sigma_2\cap I^t_\chi\times\{h_\chi\}=(u,h_\chi).
\end{equation}
Set $a:=\inf I^t_\chi$ and $d:=\sup I^t_\chi$. Let $c\in(u,d)$ be such that 
$$
\sigma_3\cap I^t_\chi\times\{0\}=(c,0).
$$
Equations \eqref{EquationCenterSaddleSinkNode} and \eqref{EquationBijectionToFlowBox} give
$\varphi_1(u)=\widetilde{u}$. We also have $\varphi_1(c)=\widetilde{c}$.

\smallskip

Let $s_W(\cdot,\cdot)$ be the map in Equation \eqref{EquationReparametrizationByHight}. The time change $t\mapsto\kappa t$ does not change the trajectories of $W$, thus by Equation \eqref{EquationConjugationFlows} and Point (5) in Lemma~\ref{LemmaFlowBox} we have a map
$$
s_Y:\big((a,u)\cup(c,d)\big)\times [0,h_\chi]\to I^t_\chi,\qquad
(x,y)\mapsto s_Y(x,y),
$$
where $s_Y(x,y)$ is the horizontal coordinate of the $Y$-trajectory starting at $(x,0)\in I^t_\chi\times\{0\}$ when it reaches hight $y$. According to Equation \eqref{EquationConjugationFlows} for 
$(\widetilde{x},\widetilde{y})=\varphi(x,y)$ we have 
\begin{equation}
\label{EquationConjugationTrajectories}
\varphi\big(s_Y(x,y),y\big)=\big(s_W(\widetilde{x},\widetilde{y}),\widetilde{y}\big).
\end{equation}
For any $x\in(a,u)\cup(c,d)$ let $\sigma_x$ be the positive trajectory of $Y$ passing through
$(x,0)\in I^t_\chi\times\{0\}$. Observe that 
$$
(u,c)=\varphi_1^{-1}(\widetilde{u},\widetilde{c})
\subset 
\varphi_1^{-1}\big([-3,3]\big)
\subset(a,d).
$$

\begin{lemma}
\label{LemmaReturnInsideRectangle}
For any $x\in(a,u)\cup(c,d)$ we have $t(x)>0$ defined as
$$
t(x)=\min\{t>0:\sigma_x(t)\in I^t_\chi\times\{h_\chi\}\}.
$$
Moreover we have $\sigma_x\big(t(x)\big)=\big(g(x),h_\chi\big)$,
where $g:I^t_\chi\to I^t_\chi$ is a smooth map, strictly increasing on $[a,u)\cup[c,d)$ with $g(x)=u$ for any $x\in[u,c]$. 
Finally we have
$$
g(x)=
\left\{
\begin{array}{ccc}
x&\text{if}& x\in I^t_\chi\setminus\varphi_1^{-1}\big([-3,3]\big)\\
\cfrac{r_0}{7}g_W\Big(\cfrac{7(x-x_0)}{r_0}\Big)+x_0 &\text{if}& x\in \varphi_1^{-1}\big([-3,3]\big).
\end{array}
\right.
$$
\end{lemma}

\begin{proof}
The instant $t(x)>0$ exists because $Y=V$ outside $Q$, while inside $Q$ we have Lemma \ref{LemmaReturnMapFlowBox} and Equation \eqref{EquationConjugationFlows}. This gives a smooth strictly increasing map 
$g:(a,u)\cup(c,d)\to I^t_\chi$ surjective onto $\interior(I^t_\chi)\setminus\{u\}$ with 
$\sigma_x\big(t(x)\big)=\big(g(x),h_\chi\big)$. We extend to a continuous non decreasing map $g:I^t_\chi\to I^t_\chi$ setting $g(a):=a$ and $g(x):=u$ for $u\leq x\leq c$. Observe that for  
$
x\in\varphi_1^{-1}\big((-3,\widetilde{u})\cup(\widetilde{c},3)\big)
$ 
Equation \eqref{EquationConjugationTrajectories} implies
\begin{align*}
g(x)=s_Y(x,h_\chi)
&
=s_Y(x,y_0+r_0)
=\frac{r_0}{7}s_W\big(\varphi(x,y_0+r_0)\big)+x_0
\\
&=\frac{r_0}{7}s_W\Big(\frac{7(x-x_0)}{r_0},7\Big)+x_0
=\frac{r_0}{7}g_W\Big(\frac{7(x-x_0)}{r_0}\Big)+x_0,
\end{align*}
where the first equality holds because $Y=V$ outside $Q$, the second and the third hold because 
$\varphi_1(x)=7/r_0(x-x_0)$, while the last follows from $g_W(\cdot)=s_W(\cdot,7)$. Therefore $g$ is smooth in $u$ and $c$ because $g_W$ is smooth at $\widetilde{u}$ and $\widetilde{c}$. Moreover this proves the formula for $g$ inside 
$\varphi_1^{-1}\big([-3,3]\big)$. Finally $g(x)=x$ for $x\in I^t_\chi\setminus\varphi_1^{-1}\big([-3,3]\big)$ because $Y=V$ outside $Q$. The Lemma is proved.
\end{proof}

\subsection{Parametric family of first return maps}
\label{SectionParameterFamilyFirstReturnMaps}

Recall the setting in \S~\ref{SectionSettingConstructionCherryFlows}. In particular recall that $\pi$ is the combinatorial datum over $\cA$ obtained from $\widehat{\pi}$ collapsing the pair of letters $\chi_L,\chi_R$ to the letter $\chi$. Fix $(\lambda,u)$ with 
$\lambda\in\Delta^d$ and $u\in I^t_\chi$. Let $T:[0,1)\to[0,1)$ be the IET defined by $(\pi,\lambda)$. Let $g:I^t_\chi\to I^t_\chi$ be the map in Lemma~\ref{LemmaReturnInsideRectangle} and extend it to a continuous non decreasing function $g:[0,1)\to [0,1)$ setting 
$$
g|_{I^t_\alpha}:=\id
\quad
\textrm{ for any }
\quad
\alpha\not=\chi.
$$
Finally define $f_{(\lambda,u)}:[0,1)\to [0,1)$ by 
\begin{equation}
\label{EquationFirstReturnCherryFlow}
f_{(\lambda,u)}:=T\circ g.
\end{equation}

Consider the horizontal segment $I:=[0,1)\times\{0\}$ embedded in $X=X(\pi,\lambda,\tau)$. Extending the construction in \S~\ref{SectionFlowInsideRectangle}, for any $x\in [0,1)$ let $\sigma_x$ be the positive trajectory of $Y_{(\lambda,u)}$ starting from $(x,0)\in [0,1)\times\{0\}$, parametrized so that $\sigma_x(0)=(x,0)$. 
In the sense of Lemma~\ref{LemmaFirstReturnFlow}, the map $f_{(\lambda,u)}$ is the first return map to $I$ of the flow $Y_{(\lambda,u)}$. An example of such map is represented in the right part of Figure \ref{FigureIETsFlatPiece}. 
In the notation of \S~\ref{SectionFlowInsideRectangle}, define
$$
I^\ast:=\Big((a,u)\cup(c,d)\cup\bigcup_{\alpha\not=\chi}\interior I^t_\alpha\Big)\times\{0\}.
$$

\begin{lemma}
\label{LemmaFirstReturnFlow}
For any $x\in I^\ast$ we have $t(x)>0$ defined by
$$
t(x):=\inf\big\{t>0:\sigma_x(t)\in I\times\{0\}\big\}.
$$
Moreover for any $x\in I^\ast$ we have
$$
\sigma_x\big(t(x)\big)=\big(f_{(\lambda,u)}(x),0\big)
$$
\end{lemma}

\begin{proof}
On $X\setminus R_\chi$ we have $Y_{(\pi,\lambda)}=V$ and the trajectories of $V$ are the same as those of $V_0$, with a different parametrization. On the other hand, on $R_\chi$ Lemma~\ref{LemmaReturnInsideRectangle} applies. Thus $t(x)$ is well defined for $x\in I^\ast$. Moreover $T$ is the first return of $V$ to $I$, thus for $\alpha\not=\chi$ and $x\in\interior I^t_\alpha$ we have 
$$
\sigma_x\big(t(x)\big)=\big(T(x),0\big)
$$
On the other hand, for $x\in(a,u)\cup(c,d)$ Lemma~\ref{LemmaReturnInsideRectangle} gives 
$
\sigma_x\big(t(x)\big)=\big(g(x),h_\chi\big).
$ 
and on the quotient surface we have 
$
I^t_\alpha\times\{h_\chi\}\simeq I^b_\alpha\times\{0\}\subset I
$, 
so that 
$$
\sigma_x\big(t(x)\big)=\big(g(x),h_\chi\big)\simeq \big( T(g(x)),0\big).
$$
\end{proof}

Varying the parameters that define $Y_{(\lambda,u)}$ and than taking its first return as in Equation \eqref{EquationFirstReturnCherryFlow}, we get a parametric family of maps $(\lambda,u)\mapsto f_{(\lambda,u)}$. In order to make the description of such parameter family consistent with the terminology of Definition \ref{DefinitionFullFamily}, we introduce the following notation. 
For any $\widehat{\lambda}\in\Delta^{d+1}$ consider the corresponding pair 
$(\lambda,u)$, where we recall that the map $(\lambda,u)\mapsto\widehat{\lambda}$ in Equation \eqref{EquationLengthDataMarkedSurface} is a linear isomorphism. Let $T$ be the IET defined by $(\pi,\lambda)$ and for any $\alpha\in\cA$ let $T_\alpha$ be the restriction of $T$ to its continuity interval $I^t_\alpha$. Define the maps
$$
g_{\chi_L}:=g|_{[a,u)}:[a,u)\to[a,u)
\quad
\textrm{ and }
\quad
g_{\chi_R}:=g|_{[u,d)}:[u,d)\to[u,d).
$$
We describe the above parameter family as a map
$$
F:\Delta^{d+1}\to\cC\big(\widehat{\pi},[0,1)\big)
,\qquad
\widehat{\lambda}\mapsto f=F(\widehat{\lambda}),
$$
where $f=F(\widehat{\lambda})$ is the GIET whose branches $f_\alpha$ with $\alpha\in\widehat{\cA}$ are given by 
\begin{equation}
\label{EquationFirstReturnCherryFlowBIS}
f_\alpha:=
\left\{
\begin{array}{cccc}
&
T_\alpha
&
\quad\textrm{ if }\quad
&
\alpha\not=\chi_L,\chi_R
\\
&
T_\chi\circ g_{\chi_L}
&
\quad\textrm{ for }\quad
&
\alpha=\chi_L
\\
&
T_\chi\circ g_{\chi_R}
&
\quad\textrm{ for }\quad
&
\alpha=\chi_R.
\end{array}
\right.
\end{equation}
Of course $f=F(\widehat{\lambda})$ is equal to the map in Equation \eqref{EquationFirstReturnCherryFlow}.

\begin{remark}
There exists an universal constant $C>0$ such that for any $\widehat{\lambda}\in\Delta^{d+1}$ and any $\alpha\in\widehat{\cA}$ the Lipschitz constant 
of the branch $f_\alpha$ of $F(\widehat{\lambda})$ satisfies
$$
\lipschitz(f_\alpha)\leq C.
$$
Indeed, according to Equation \eqref{EquationFirstReturnCherryFlowBIS} we have 
$\lipschitz(f_\alpha)=1$ for any $\alpha\not=\chi_L,\chi_R$. Moreover, since $g_W$ is a smooth function, Lemma \ref{LemmaReturnInsideRectangle} implies 
$$
\max\big\{\lipschitz(f_{\chi_L}),\lipschitz(f_{\chi_R})\big\}
\leq
\lipschitz\bigg(
\cfrac{r_0}{7}g_W\Big(\cfrac{7(x-x_0)}{r_0}\Big)+x_0
\bigg)
=
\lipschitz(g_W)<+\infty.
$$
\end{remark}

\subsection{Extension to the boundary}

Recall that for any $\alpha\in\widehat{\cA}$ Equation \eqref{EquationSingularitiesTimeZeroIET} defines linear maps $u^t_\alpha,u^b_\alpha:\Delta^{d+1}\to[0,1)$  given by 
$$
u^t_\alpha(\widehat{\lambda}):=
\sum_{\widehat{\pi}_t(\beta)\leq\widehat{\pi}_t(\alpha)-1}\widehat{\lambda}_\beta
\quad\text{and}\quad
u^b_\alpha(\widehat{\lambda}):=
\sum_{\widehat{\pi}_b(\beta)\leq\widehat{\pi}_b(\alpha)-1}\widehat{\lambda}_\beta.
$$
These maps extend continuously to maps $u^t_\alpha,u^b_\alpha:\overline{\Delta^{d+1}}\to[0,1]$. 
For any $\widehat{\lambda}\in\partial\Delta^{d+1}$ let 
$
\widehat{\cA}'=\widehat{\cA}'(\widehat{\lambda})
$ 
be the set of $\alpha\in\widehat{\cA}$ such that $\widehat{\lambda}_\alpha=0$ and $\widehat{\pi}'$ be the reduction of $\widehat{\pi}$ obtained removing the letters of $\widehat{\cA}'$ (see \S \ref{SectionDescriptionAllowedDegenerations}). Let $T_{\widehat{\lambda}}$ be the IET over $\widehat{\pi}'$ that sends the interval 
$
I^t_\alpha:=
\big[u^t_\alpha(\widehat{\lambda}),u^t_\alpha(\widehat{\lambda})+\widehat{\lambda}_\alpha\big)
$ 
onto 
$
I^b_\alpha:=
\big[u^b_\alpha(\widehat{\lambda}),u^b_\alpha(\widehat{\lambda})+\widehat{\lambda}_\alpha\big)
$ 
via a translation for any $\alpha\in\widehat{\cA}\setminus\widehat{\cA}'$. Let $S_{\widehat{\lambda}}=\{s_\alpha:\alpha\in\widehat{\cA}'\}$ be the set of points defined by 
$
s_\alpha:=\big(u^t_\alpha(\widehat{\lambda}),u^t_\alpha(\widehat{\lambda})\big)
$. 
Finally consider the degeneration 
$$
D_{\widehat{\lambda}}:=
(T_{\widehat{\lambda}},S_{\widehat{\lambda}})
\in
\widehat{\cC}\big(\widehat{\pi},[0,1)\big).
$$ 
Let $F:\Delta^{d+1}\to\cC\big(\widehat{\pi},[0,1)\big)$ be the map in Equation \eqref{EquationFirstReturnCherryFlowBIS}. We extend it by the map
\begin{equation}
\label{EquationExtensionFullFamily}
\widehat{F}:\overline{\Delta^{d+1}}\to\widehat{\cC}\big(\widehat{\pi},[0,1)\big)
\quad\textrm{ where }\quad
\widehat{F}(\widehat{\lambda}):=
\left\{
\begin{array}{l}
F(\widehat{\lambda})
\quad\textrm{ if }\quad
\widehat{\lambda}\in\Delta^{d+1}
\\
D_{\widehat{\lambda}}
\quad\textrm{ if }\quad
\widehat{\lambda}\in\partial\Delta^{d+1}.
\end{array}
\right.
\end{equation}

\begin{lemma}
\label{LemmaContinuousDependenceFirstReturn}
The map $\widehat{F}$ is continuous.
\end{lemma}

\begin{proof}
We first prove the continuity of the map $F$ in Equation \eqref{EquationFirstReturnCherryFlowBIS}. This corresponds to prove that for any $\alpha\in\widehat{\cA}$ the branch $f_\alpha$ of 
$
f_{\widehat{\lambda}}=F(\widehat{\lambda})
$ 
depends continuously on 
$\widehat{\lambda}$ with respect to the Hausdorff distance, and that the function 
$
\widehat{\lambda}\mapsto j_\alpha(f_{\widehat{\lambda}})
$ 
is continuous. The continuous dependence of the $\alpha$-branch $f_\alpha$ is obvious for letters 
$\alpha\not=\chi_L,\chi_R$. For $\alpha=\chi_L,\chi_R$ it is enough to prove that the map $g:I^t_\chi\to I^t_\chi$ in Lemma \ref{LemmaReturnInsideRectangle} depends continuously on the parameters $(\lambda,u)$. 
This is true because $r_0$ in Equation \eqref{EquationRadiusNeighbourhoodSingularities} and $x_0,y_0$ in Equation \eqref{EquationCenterSaddleSinkNode} do (actually $y_0$ is constant), and Lemma \ref{LemmaReturnInsideRectangle} gives a formula for $g$ in terms of such quantities. Moreover for any $\widehat{\lambda}\in\Delta^{d+1}$ the only flat interval of  $f_{\widehat{\lambda}}$ is in the branch $\alpha=\chi_R$ and Lemma \ref{LemmaReturnInsideRectangle} gives 
$
j_{\chi_R}(f_{\widehat{\lambda}})=\varphi_1^{-1}(\widetilde{c})=x_0+r_0\widetilde{c}/7
$, 
which varies continuously with $\widehat{\lambda}$ because $x_0$ and $r_0$ do ($\widetilde{c}$ is a constant).

Now we prove the continuity of the map $\widehat{F}$ in Equation \eqref{EquationExtensionFullFamily}. 
Fix $\widehat{\lambda}\in\partial\Delta^{d+1}$. We must prove that if
$
\widetilde{\lambda}\in\overline{\Delta^{d+1}}
$ 
is close to $\widehat{\lambda}$ then 
$
\widehat{F}(\widetilde{\lambda})
$ 
is close to 
$
D_{\widehat{\lambda}}=(T_{\widehat{\lambda}},S_{\widehat{\lambda}})
$. 
This is true if  $\widetilde{\lambda}\in\partial\Delta^{d+1}$, because the restriction of $\widehat{F}$ to $\partial\Delta^{d+1}$ is continuous by continuity of the extended maps $u^t_\alpha,u^b_\alpha:\overline{\Delta^{d+1}}\to[0,1]$, $\alpha\in\cA$. Otherwise, if $\widetilde{\lambda}\in\Delta^{d+1}$, we observe that 
$r_0(\widetilde{\lambda})\to0$ as $\widetilde{\lambda}\to\partial\Delta^{d+1}$, so that 
$\big|\varphi_1^{-1}\big([-3,3]\big)\big|\to0$ and $g\to\id_\chi$, where $\id_\chi$ denotes the identity on the interval $I^t_\chi$. This proves that the continuous branches of $f_{\widetilde{\lambda}}$ converge to the branches of $D_{\widehat{\lambda}}$. Moreover 
$
\widetilde{\lambda}\mapsto j_\alpha(f_{\widetilde{\lambda}})
$ 
is continuous at $\widetilde{\lambda}=\widehat{\lambda}$ because 
$$
j_{\chi_R}(f_{\widetilde{\lambda}})-u^t_{\chi_R}(f_{\widetilde{\lambda}})=
\frac{r_0}{7}\big|\widetilde{c}-\widetilde{u}\big|\to0
\quad\text{ as }\quad
\widetilde{\lambda}\to\widehat{\lambda},
$$
where $\widetilde{u}$ is also a constant. The Lemma is proved.
\end{proof}

\subsection{End of the proof of Theorem~\ref{TheoremMainTheoremCherryFlow}}
\label{SectionEndProofTheoremCherryFlows}

Lemma \ref{LemmaContinuousDependenceFirstReturn} proves that the parametric family 
$
F:\Delta^{d+1}\to\cC\big(\widehat{\pi},[0,1)\big)
$ 
in Equation \eqref{EquationFirstReturnCherryFlowBIS} is continuous and that Point (3) in Definition \ref{DefinitionFullFamily} is satisfied. Moreover for $\widehat{\lambda}\in\Delta^{d+1}$ let $\widehat{T}$ be the IET determined by the data 
$(\widehat{\pi},\widehat{\lambda})$. The map
$
\widehat{\lambda}\mapsto u^b\big(F(\widehat{\lambda})\big)=u^b(\widehat{T}) 
$
is a linear isomorphism from $\Delta^{d+1}$ onto $\nabla_{\widehat{\pi}}$. Thus in particular its extension to the closure is an homeomorphism. The map 
$
\widehat{\lambda}\mapsto u^t\big(F(\widehat{\lambda})\big)=u^t(\widehat{T}) 
$ 
is also a bijection from $\Delta^{d+1}$ onto $\nabla^t_{\widehat{\pi}}$. Points (1) and (2) in Definition \ref{DefinitionFullFamily} are satisfied.

\smallskip

Point (1) of Theorem \ref{TheoremMainTheoremCherryFlow} follows from Lemma \ref{LemmaZerosVectorField}. Point (2) from Lemma \ref{LemmaFirstReturnFlow}. Finally Point (3) follows from Theorem \ref{TheoremFullFamilyTheorem} because the parametric family 
$
F:\Delta^{d+1}\to\cC\big(\widehat{\pi},[0,1)\big)
$ 
satisfies the assumptions in Definition \ref{DefinitionFullFamily}. Theorem \ref{TheoremMainTheoremCherryFlow} is proved  $\qed$

\subsection{Proof of Corollary \ref{Cor:uniqueQMSandmeasures}}
\label{SectionProofOfCorollary}

Let $T$ be a Keane IET with combinatorial datum $\pi$ and $\widehat{\lambda}\in\Delta^{d+1}$ be as in Point (3) of Theorem \ref{TheoremMainTheoremCherryFlow}. Let $f=f_{\widehat{\lambda}}$ be the map in Equation \eqref{EquationFirstReturnCherryFlow} and $h:[0,1)\to[0,1)$ be the semi-conjugacy to $T$ in Equation \eqref{EquationSemiconjugation}. Since $T$ satisfies the Keane property, then it is minimal and has no periodic points. Let $(p,t)\mapsto\Phi(p,t)$ be the flow of $Y_{\widehat{\lambda}}$. We prove that the flow of 
$Y_{\widehat{\lambda}}$ has only one quasi-minimal set. Set
\begin{equation}
\label{EquationQuasiMinimalSet}
K:=[0,1)\setminus
\bigcup_{y\in[0,1)}\operatorname{Int}h^{-1}(y).
\end{equation}
We first prove existence. Since $h$ is non-decreasing, there are at
most countably many $y$ for which $h^{-1}(y)$ is a non-trivial
interval. Hence we can choose $y$ such that $h^{-1}(y)=\{x\}$ for some $x\in[0,1)$. 
By minimality of $T$, there exists $n_j\to\infty$ such that
$T^{n_j}(y)\to y$. We have 
$h(f^{n_j}(x))=T^{n_j}(y)\to y$. Since $h^{-1}(y)=\{x\}$, it follows that $f^{n_j}(x)\to x$.
Thus $x$ is recurrent for $f$ and hence also for the
flow. Therefore 
$
Q:=\overline{\{\Phi_t(x):t\in\mathbb R\}}
$ 
is a quasi-minimal set. Notice that $Q$ is invariant, since the closure
of a full orbit is invariant under the flow. 
We now prove uniqueness. Let $x$ be an intersection with the section
of any non-trivial recurrent orbit and set
$
A:=\overline{\{f^n(x):n\geq0\}}
$. 
We claim that $A=K$. First, $A\subset K$. Indeed, if $A$ meets 
$\operatorname{Int}h^{-1}(y)$, the orbit of $x$ would enter this
interval infinitely many times. Thus we have 
$
T^{n_i}(h(x))=T^{n_j}(h(x))=y
$ 
for some $n_i<n_j$ and hence $T^{n_j-n_i}(y)=y$, contradicting the absence of periodic
points for $T$. For the inclusion $K\subset A$, let $z\in K$ have singleton fiber, that is $h^{-1}(h(z))=\{z\}$. By minimality, take $n_j\to\infty$ with $T^{n_j}(h(x))\to h(z)$. We get $h(f^{n_j}(x))\to h(z)$. Since 
$h^{-1}(h(z))=\{z\}$, then $f^{n_j}(x)\to z$. Thus $z\in A$. 
It remains to observe that the points with singleton fibers are dense
in $K$. Indeed, if $x_0\in K$ and $U$ is a neighborhood of $x_0$,
then $h$ is not constant on $U$. Hence $h(U)$ contains a non-trivial
interval, and, since only countably many values have non-trivial
fibers, $U$ contains a point with singleton fiber. Since $A$ is
closed, it follows that $K\subset A$. Therefore $A=K$. Thus every non-trivial recurrent orbit intersects the section in a set with closure equal to $K$. Consequently, every
such orbit has the same closure, which is equal to 
$$
\cQ:=\overline{\{\Phi_t(z):z\in K,\ t\in\mathbb R\}}.
$$

Let $\cM(f)$ and $\cM(T)$ be the convex sets of probability invariant measures under $f$ and $T$ respectively. If $h^{-1}(y)$ is a non-trivial interval then we must have $\mu\big(h^{-1}(y)\big)=0$ for any $\mu\in\cM(f)$, indeed otherwise by $f$-invariance of $\mu$ the iterates of $h^{-1}(y)$ would give a sequence of disjoint intervals with infinite total $\mu$-measure. Recalling that there are at most countably many $y$ such that $h^{-1}(y)$ is a non-trivial interval, let $C\subset K$ be the complement in $[0,1)$ of such intervals. For $\mu_1,\mu_2\in\cM(f)$ with $h_\ast\mu_1=h_\ast\mu_2$ and a Borel set $B\subset[0,1)$ we have 
$$
\mu_1(B)=\mu_1(B\cap C)=h_\ast\mu_1\big(h(B\cap C)\big)=h_\ast\mu_2\big(h(B\cap C)\big)=\mu_2(B\cap C)=\mu_2(B).
$$
It follows that the linear map $h_\ast:\cM(f)\to\cM(T)$ is injective. Such map is also surjective and thus an isomorphism. Indeed given any $\nu\in\cM(T)$ and any probability measure $\lambda$ with $h_\ast\lambda=\nu$, non-necessarily $f$-invariant, any weak-$\ast$ limit $\mu$ of the average 
$
1/n\sum_{j=0}^{n-1}f^j_\ast\lambda
$ 
belongs to $\cM(f)$ and satisfies $h_\ast\mu=\nu$. For any ergodic $\mu\in\cM(f)$ we obtain an ergodic flow-invariant measure $\widehat{\mu}$ on $X$ defined on observables $\varphi$ by 
$$
\widehat{\mu}(\varphi):=
\int_{x=0}^1\Big(\int_0^{t(x)}\varphi\big(\Phi(x,t)\big)dt\Big)d\mu(x),
$$
where $t(x)$ is defined as in Lemma \ref{LemmaFirstReturnFlow} and where $\widehat{\mu}$ can be normalized to a probability measure if and only if 
$$
\int_{x=0}^1t(x)d\mu(x)<+\infty.
$$
Since $\mu$ is supported on $C\subset K$ then $\widehat{\mu}$ is supported on the quasi-minimal set $\cQ$. Conversely, if $\widehat{\mu}$ is an ergodic measure supported on $\cQ$, we define an ergodic measure for $f$ setting 
$$
\mu(B):=\lim_{\epsilon\to0}\frac{1}{\epsilon} 
\widehat{\mu}\bigg(\Big\{\Phi(p,t):p\in B,0\leq t\leq \epsilon\Big\}\bigg).
$$
The above measure $\mu$ is finite even if $\widehat{\mu}$ is infinite. Corollary \ref{Cor:uniqueQMSandmeasures} is proved. $\qed$

\appendix

\section{Estimates for the Hausdorff distance}
\label{SectionEstimatesHausdorffDistance}

A modulus of continuity for a continuous map $h:J\to I$ between intervals $I,J\subset[0,1]$ is a monotone function  
$\omega_h:[0,\infty)\to[0,\infty)$ with $\lim_{\delta\to0_+}\omega_h(\delta)=0$ and such that 
$|h(x)-h(y)|\leq \omega_h(|x-y|)$ for any $x,y\in I$. Take $h_0:I_0\to J_0$ and $h:I\to J$ as above and two points 
$x_0\in I_0$ and $x\in I$. Let $x'\in I_0$ be such that 
$$
\big\|\big(x,h(x)\big)-\big(x',h_0(x')\big\|_\infty\leq \distance(h,h_0),
$$
where for simplicity we write 
$
\distance(h,h_0):=\distance\big(\overline{\text{Graph}(h)},\overline{\text{Graph}(h_0)}\big)
$. 
We have
$$
\left\{
\begin{array}{l}
|x'-x_0|\leq |x'-x|+|x-x_0|\leq \distance(h,h_0)+|x-x_0|
\\
\big|h(x)-h_0(x_0)\big|
\leq 
\big|h(x)-h_0(x')\big| + \big|h_0(x')-h_0(x_0)\big|
\leq 
\distance(h,h_0) + \omega_{h_0}\big(|x'-x_0|\big).
\end{array}
\right.
$$
Therefore
\begin{equation}
\label{EquationHausdorffDistanceBoundsDifference}
\big|h(x)-h_0(x_0)\big|
\leq 
\distance(h,h_0) + \omega_{h_0}\Big(\distance(h,h_0)+|x-x_0|\Big).
\end{equation}
Given $f\in\cC\big(\pi,[0,1)\big)$, Equation \eqref{EquationHausdorffDistanceBoundsDifference} is relevant for the branches 
$f_\alpha:I_\alpha^t\to I_\alpha^b$ in Equation \eqref{EquationPiecewiseHomeomorphism} and for their inverse branches $f_\alpha^{-1}:I_\alpha^b\to J_\alpha^b$.

\begin{lemma}
\label{LemmaHausdorffDistanceRestrictions}
For any $f,f'\in\cC\big(\pi,[0,1)\big)$ and any $\chi\in\cA$ the restrictions satisfy
$$
\distance\big(f_\chi|_{J_\chi(f)},f'_\chi|_{J_\chi(f)}\big)
\leq
\max\Big\{\distance\big(f_\chi,f'_\chi\big),\big|j_\chi(f)-j_\chi(f')\big|\Big\}.
$$
\end{lemma}

\begin{proof}
Fix $x\in J_\chi(f)$, that is $j_\chi(f)\leq x<u^b_\chi(f)$. It is enough to find $y\in J_\chi(f')$ such that 
$$
\big\|\big(x,f_\chi(x)\big)-\big(y,f'_\chi(y)\big)\big\|
\leq
\max\Big\{\distance\big(f_\chi,f'_\chi\big),\big|j_\chi(f)-j_\chi(f')\big|\Big\},
$$
indeed the estimate above holds for any $x\in J_\chi(f)$ and moreover the role of $f$ and $f'$ are completely symmetric. Let $x'\in I^t_\chi(f')$ such that 
$$
\big\|\big(x,f_\chi(x)\big)-\big(x',f'_\chi(x')\big)\big\|
\leq
\distance(f_\chi,f'_\chi).
$$
If $x'\in J_\chi(f')$ then the required estimate holds with $y:=x'$. If otherwise $u^t_\chi(f)<x'\leq j_\chi(f')$, then $f'(x')=f'\big(j_\chi(f')\big)$ and
$
\big|f(x)-f'\big(j_\chi(f')\big)\big|
=
\big|f(x)-f'(x')\big|
$. 
Moreover
$$
\big|x-j_\chi(f')\big|
=
\left\{
\begin{array}{cccc}
x-j_\chi(f') \leq x-x' & = |x-x'| & \text{if} & x\geq j_\chi(f')
\\
j_\chi(f')-x \leq j_\chi(f')-j_\chi(f) & = \big|j_\chi(f')-j_\chi(f)\big| & \text{if} & x< j_\chi(f').
\end{array}
\right.
$$
It follows that 
$$
\Big\|\big(x,f_\chi(x)\big)-\Big(j_\chi(f'),f'_\chi\big(j_\chi(f')\big)\Big)\Big\|
\leq
\max\Big\{\distance\big(f_\chi,f'_\chi\big),\big|j_\chi(f)-j_\chi(f')\big|\Big\},
$$
so that the required estimate holds with $y:=j_\chi(f')$. The Lemma is proved.
\end{proof}

The inversion $f\mapsto f^{-1}$ is an isometry for the Hausdorff distance, thus the next Corollary \ref{CorollaryHausdorffDistanceInverse} follows immediately.

\begin{corollary}
\label{CorollaryHausdorffDistanceInverse}
For any $f,f'\in\cC\big(\pi,[0,1)\big)$ and any $\chi\in\cA$ we have 
$$
\distance\big(f_\chi^{-1},(f'_\chi)^{-1}\big)
\leq
\max\Big\{\distance\big(f_\chi,f'_\chi\big),\big|j_\chi(f)-j_\chi(f')\big|\Big\}.
$$
\end{corollary}

\section{Uniform modulus of continuity}
\label{AppendixUniformContinuityModulus}

Let $\cF$ be a family of continuous maps $f:I\to J$ between intervals $I,J\subset[0,1]$, which is compact with respect to the Hausdorff distance. For any $f\in\cF$ denote $I_f,J_f$ respectively its domain and image. 
Fix $\epsilon>0$. For any $f\in\cF$ let $\delta_f>0$ such that 
$
\big|f(x)-f(y)\big|\leq\epsilon/3
$ 
for $x,y$ with $|x-y|\leq2\delta_f$. Set 
$
r_f:=\min\big\{\delta_f/2,\epsilon/3\big\},
$
then let $\cB(f)$ be the ball of radius $r_f$ around $f$. Consider any map $h\in\cB(f)$. For any $x_1,x_2\in I_h$ with $|x_1-x_2|\leq \delta_f$ consider $y,z\in I_f$ such that 
$$
\left\{
\begin{array}{l}
\max\Big\{|x_1-y|,\big|h(x_1)-f(y)\big|\Big\}\leq r_f
\\
\max\Big\{|x_2-z|,\big|h(x_2)-f(z)\big|\Big\}\leq r_f.
\end{array}
\right.
$$ 
We have
$
|y-z|\leq|y-x_1|+|x_1-x_2|+|x_2-z|\leq 2r_f+\delta_f\leq 2\delta_f
$, 
therefore 
$$
\big|h(x_1)-h(x_2)\big|
\leq
\big|h(x_1)-f(y)\big|+\big|f(y)-f(z)\big|+\big|f(z)-h(x_2)\big|
\leq
2r_f+\frac{\epsilon}{3}\leq\epsilon.
$$
By compactness, let $f_1,\dots,f_N$ be maps in $\cF$ such that 
$
\cF\subset \cB(f_1)\cup\dots\cup\cB(f_N)
$, 
then set
$$
\delta:=\min\big\{d_{f_1},\dots,\delta_{f_N}\big\}.
$$
In conclusion, take any $g\in\cF$. Let $f_k$ be such that $g\in\cB(f_k)$. For any $x_1,x_2\in I_g$ with 
$|x_1-x_2|\leq\delta\leq\delta_{f_k}$ we have
$
\big|g(x_1)-g(x_2)\big|\leq\epsilon
$. 
This proves that we have a uniform modulus of continuity $\omega(\cdot):[0,1]\to[0,1]$ for $\cF$, where for any 
$\delta>0$ we set 
$$
\omega(\delta):=\sup\Big\{\big|g(x)-g(y)\big|;g\in\cF,x,y\in I_g\text{ with }|x-y|\leq\delta\Big\}.
$$

\subsection{Acknowledgments*}

We are grateful to S. Marmi, to 
the Scuola Normale Superiore and to Unicredit Bank for funding and logistic support during our stay at 
the De Giorgi Center in Pisa, where this project was developed.
L.M. was supported by the European Union - NextGenerationEU under the National Recovery and Resilience
Plan (PNRR) - Mission 4 Education and research - Component 2 From
research to business - Investment 1.1 Notice Prin 2022 - DD N. 104 del
2/2/2022, from title "Geometry and topology of manifolds", proposal code
2022NMPLT8 - CUP J53D23003820001. L.P. was partially supported by the VR grant 67578, Renormalization in dynamics. We are also grateful to C. Ulcigrai for numerous discussions on this and related topics.

\end{document}